\documentclass[12pt]{article}

\usepackage[left=1.28in,top=.8in,right=1.28in,bottom=.8in]{geometry}

\usepackage{times}
\usepackage{amsmath}
\usepackage{amsthm,amscd}
\usepackage{amssymb}
\usepackage{latexsym}
\usepackage{enumerate}
\usepackage{enumitem}
\usepackage{graphics}
\usepackage{color}
\usepackage{colortbl}
\usepackage[toc,page]{appendix}
\usepackage{graphicx}
\usepackage{hyperref}
\usepackage{tikz-cd}
\usepackage[noadjust]{cite}
\usepackage{caption}
\usepackage[normalem]{ulem}
\usepackage{hyperref}
\hypersetup{
    colorlinks=true,
    urlcolor=black,
    citecolor=blue,   
    linkcolor=blue,
}
\usepackage{float}

\newtheorem{theorem}{Theorem}[section]
\newtheorem{lemma}[theorem]{Lemma}
\newtheorem{proposition}[theorem]{Proposition}

\newtheorem{remark}[theorem]{Remark}

\newtheorem{question}[theorem]{Question}

\newcommand{\brac}[1]{\left( #1\right)}

\newcommand{\sqbrac}[1]{\left[ #1\right]}

\newcommand{\mbb}[1]{\mathbb{#1}}
\newcommand{\mrm}[1]{\mathrm{#1}}

\newcommand{\re}[1]{\textcolor{red}{#1}}
\newcommand{\bl}[1]{\textcolor{blue}{#1}}

\DeclareMathOperator{\tL}{\tilde{L}}

\DeclareMathOperator{\kernel}{kernel}

\DeclareMathOperator{\divi}{div}
\DeclareMathOperator{\Gal}{Gal}

\DeclareMathOperator{\SK1}{SK_1}
\DeclareMathOperator{\SL1}{SL_1}

\DeclareMathOperator{\ram}{ram}
\DeclareMathOperator{\N}{N}
\DeclareMathOperator{\M}{M}
\DeclareMathOperator{\HH}{H}
\DeclareMathOperator{\Nrd}{Nrd}
\DeclareMathOperator{\Cor}{Cor}
\DeclareMathOperator{\rbc}{rbc}
\DeclareMathOperator{\nr}{nr}
\DeclareMathOperator{\Br}{Br}
\DeclareMathOperator{\ind}{index}
\DeclareMathOperator{\inv}{inv}

\DeclareMathOperator{\charac}{char}
\DeclareMathOperator{\Int}{Int}
\DeclareMathOperator{\Spec}{Spec}
\DeclareMathOperator{\supp}{support}
\DeclareMathOperator{\cd}{cd}

\title{Reduced Whitehead groups of prime exponent algebras over $p$-adic curves} 
\author{ Nivedita Bhaskhar\\ \small{Department of Mathematics, University of California Los Angeles, Los Angeles, CA 90095-1555, USA.}\\ \href{mailto:nbhaskh@math.ucla.edu}{\small{\tt nbhaskh@math.ucla.edu}}}

\date{}

\begin{document}
\maketitle

\begin{abstract}
Let $F$ be the function field of a curve over a $p$-adic field. Let $D/F$ be a central division algebra of prime exponent $\ell$ which is different from $p$. Assume that $F$ contains a primitive ${{\ell}^2}^{\mathrm{th}}$ root of unity. Then the abstract group $\SK1(D):=\frac{\SL1(D)}{\left[D^*, D^*\right]}$ is trivial. \end{abstract}

\section{Introduction}
The group $\SL1(A)$ of reduced norm one elements of a finite dimensional central simple algebra $A$ over a field $K$ is one of the main and well-studied examples of simply connected almost simple algebraic groups of type A.  The commutator subgroup $[A^*,A^*]$ is clearly contained in $\SL1(A)$. Whether the reverse inclusion holds is however a far more subtle and difficult question to tackle. This problem was formulated by Tannaka and Artin independently in terms of $\SK1(A)$ which is defined to be the abstract quotient group $\frac{\SL1(A)}{\left[A^*, A^*\right]}$. 

\begin{question}[Tannaka-Artin, 1943]
Is $\SK1(A)$ trivial?
\end{question}

The Tannaka-Artin problem can be rephrased as a special case of the more general Kneser-Tits problem. For $G$, a semisimple simply connected isotropic $K$-group, let $G^+(K)$ denote the normal subgroup generated by the conjugates of the $K$-points of the unipotent radical of a proper $K$ parabolic of $G$. One defines the reduced Whitehead group to be  $W(G, K) := \frac{G(K)}{G^+(K)}$. The Kneser-Tits problem asks whether $W(G, K)$ is trivial.

The Tannaka-Artin problem was answered affirmatively for square-free index algebras over arbitrary fields (\cite{W}). It was also shown that $\SK1(A)$ was trivial for \textit{all} central simple algebras $A$ defined over local or global fields (\cite{Nakayamma},\cite{W}) and it was widely believed that the Tannaka-Artin question had a positive answer in general. However Platonov's famous example (\cite{Plat78}) of a biquarternion division algebra $D$ over an iterated Laurent-series field $\mathbb{Q}_p((x))((y))$ with non-trivial $\SK1(D)$ negatively settled the Tannaka-Artin problem and also gave rise to the first example of a non-rational simply connected almost simple algebraic $K$-group. Note that the cohomological dimension of the base field under consideration is $4$. However, in the same paper by Platonov, it was also shown that the Tannaka-Artin problem has a positive answer for central simple algebras over fields of cohomological dimension $\leq 2$.

In 1991, Suslin conjectured that if the index of the central simple algebra $D/K$ is not square free, then $\SK1(D)$ is \textit{generically non-trivial}, i.e, there exists a field extension $F/K$ such that $\SK1(D\otimes_K F)$ is non-trivial (\cite{Suslin91}). More formally, the Suslin invariant 

\[\rho : \SK1(D)\to \frac{\kernel\left[\HH^4_{et}\left(K, \mu_n^{\otimes 3}\right) \to \HH^4_{et}\left(K(Y), \mu_n^{\otimes 3}\right) \right]}{[D] \bullet \HH^2\left(K, \mu_n\right)},\]

where $Y$ is the Severi-Brauer variety defined by $D$, a central division algebra of degree $n$, was conjectured to send the generic element to a non-trivial image. Suslin's conjecture was settled affirmatively by Merkurjev for algebras with indices divisible by $4$ in (\cite{Me03}, \cite{Me06}).

In the case when the index of $D$ is $4$, it is known that $\rho$ is in fact an isomorphism (Rost, Chapter 17 \cite{KMRT} ; \cite{Merkurjev}). Hence if $\cd K\leq 3$, then $\SK1(D)=\{0\}$. This led Suslin to ask whether $\SK1(D)=\{0\}$ for any central simple algebra $D$ of index $\ell^2$ where $\ell$ is a prime, over fields of  cohomological dimension $3$ (\cite{Suslin91}).

In this paper, we settle this question affirmatively for exponent $\ell$ algebras over function fields of $p$-adic curves where $\ell$ is any odd prime not equal to $p$, assuming that our base field contains a primitive ${\ell^2}^{\mathrm{th}}$ root of unity (Theorem \ref{theoremSK1}). The proof, whose strategy is outlined below, relies on the techniques of patching as developed by Harbater-Hartmann-Krashen (HHK) in (\cite{HH}, \cite{HHK09}, \cite{HHK14} \& \cite{HHK15}) and exploits the arithmetic of the base field to show triviality of the reduced Whitehead group.

Let $F=K(X)$ be the function field of a smooth projective geometrically integral curve $X$ over a $p$-adic field $K$. Let $D$ denote a central division algebra over $F$ of exponent $\ell$ where $\ell$ is an odd prime different from $p$. Let $z\in \SL1(D)$ lie in some maximal subfield $M$ of $D$. We would like to show that $z$ is a product of commutators. The results of Saltman and Wang (\cite{S97}, \cite{S98}, \cite{W}) along with standard Galois theory techniques help reduce to the case when $D$ has index $\ell^2$ and $M$ contains a sub-cyclic degree $\ell$ extension $Y/F$. Let $\N_{M/Y}(z)=a$, which therefore has further norm one to $F$. We then modify Platonov's argument in (\cite{Plat76}) adapting it to our situation as follows:

We split $a$ into a product of suitable elements $a_1$ and $a_2$ in $Y$,  where the case of each $a_i$ is easier to handle. More precisely, we find elements $a_1,a_2\in Y$ and degree $\ell$ sub-field extensions $E_1/F$, $E_2/F$ in $D$ which commute with $Y$ such that $a_j$ is a norm from $YE_j$ of a product of commutators for each $j=1,2$. One can think of having \textit{moved} the problem over to the fields $E_j$s, which by construction are more ``amenable" and where we can solve the problem. We then modify $z$ by commutators so that the modified $z$ (and hence also the original $z$) is a product of commutators (Proposition \ref{propositionEconsequence}).  The required $E_j$s and $a_j$s are constructed by HHK patching by prescribing compatible local data for an appropriate model $\mathcal{X}$ of $X$.

We now briefly mention what each section in the paper is about. The second section collects lemmata about the shape of units, norms of field extensions and reduced norms of algebras defined over some special complete fields encountered in the patching set-up. It also contains some class field theory lemmata which will be useful in approximating local data to get global objects. The third section sets forth patching notations, fixes a preliminary model $\mathcal{X}$ of $X$ arranging some necessary divisors to be in good shape and gives the initial reductions which help simplify the problem. It also spells out the overall strategy adopted in the proof (mentioned above) in more precise detail. 

\newpage

The fourth and fifth sections classify into types, codimension one and closed points of $\mathcal{X}$ lying on the special fiber. Here, we also understand the configuration of the cyclic sub-extension $Y/F$ and the shape of the norm one element $a\in Y$ at the fraction fields of the local rings at these points completed at their maximal ideals. The sixth section discusses further blowing up the model at closed points to eliminate certain types of closed points from the classification. It also constructs a partial dual graph and outputs a nine-colouring of it, which will help in ensuring compatibility of the local data at the \textit{branches} in the patching problem. The seventh section gives patching data $(a_{1,P}, a_{2,P}, E_{1,P}, E_{2,P})$ at closed points $P$ while the next two discuss their structure over the branches.

The tenth and eleven sections give patching data $(a_{1,\eta}, a_{2,\eta}, E_{1,\eta}, E_{2,\eta})$ at codimension one points $\eta$ of $\mathcal{X}$ lying on the special fiber. We patch the data in the twelfth section by \textit{spreading}  $(a_{1,\eta}, a_{2,\eta}, E_{1,\eta}, E_{2,\eta})$ to work over open sets $U_{\eta}\ni \eta$ of the special fiber to get the required elements $a_1, a_2\in Y$ and extensions $E_1, E_2/F$. The final section uses patching again to finally solve the problem over the $E_j$s.
\section{Lemmata}

\subsection{Notations and terminology} \label{sectionnotation}
Let $\ell$ be a prime and let $J$ be a field which is not of characteristic $\ell$ containing $\rho$, a primitive $\ell^{\mathrm{th}}$ root of unity. Then for $a, b\in J^*$, we let the symbol $(a,b)$ denote the $J$-cyclic $\ell$-algebra
\[(a,b) = J\langle i, j | i^{\ell}=a, j^{\ell}=b , ij=\rho ji  \rangle.\]

If $E/J$ is a cyclic extension of degree $\ell$ with $\Gal(E/J)=\langle \sigma\rangle$ and $b\in J^*$, we let the symbol  $\left(E,\sigma, b\right)$ (or $(E,b)$ if the automorphism $\sigma$ is clear from the context) denote the $J$-cyclic $\ell$ algebra
\[\left(E,\sigma, b\right) = \bigoplus_{i=0}^{\ell-1} u^iE, \ u^{\ell}=b, \ eu=u\sigma(e)\ \forall \ e\in E.\]

We also note that for central simple algebras (abbreviated as CSAs) $D_1, D_2$ over $J$, we use $D_1=D_2$ to mean equality in $\Br(J)$, i.e $D_1=D_2$ denotes that $D_1/J$ and $D_2/J$ are Brauer equivalent.

Let $F$ be a complete discretely valued field with ring of integers $R$ and residue field $k$.  Let $\ell$ be a prime which is not equal to $\charac(k)$ such that $F$ contains a primitive $\ell^{\mathrm{th}}$ root of unity. Let $\alpha \in \Br(F)$ be an element of order $\ell$ which is ramified at $R$. Recall the residue map $\partial_F : \HH^2(F, \mu_{\ell}) \to \HH^1(k, \mathbb{Z}/\ell{Z})$. Let $\partial_F(\alpha) = \left(\overline{E}/k, \overline{\sigma}\right)$ where $\overline{E}/k$ is a cyclic extension of degree $\ell$ with Galois group generated by $\overline{\sigma}$. 

\textbf{Residual extension}: 
There is a unique unramified cyclic extension $E/F$ of degree $\ell$ with residue field $\overline{E}$. We call $E$ the \textit{lift of residue} of $\alpha $ at $R$ or the \textit{residual extension} of $\alpha$ at $R$.

\textbf{Residual Brauer class}: 
We define the residual class of $\alpha$ (depending on the choice of a parameter of $R$) as in (\cite{S07}). Given a parameter $\pi$ of $R$, let $L$ denote the totally ramified extension $F\left(\sqrt[\ell]{\pi}\right)$ and $S$ denote the ring of integers of $L$ with residue field also $k$. Then $\alpha_L := \alpha\otimes_F L$ is unramified and hence is in $\Br(S)$.

Let $\beta \in \Br(k)$ denote the image of $\alpha_L$. Then the \textit{residual Brauer class} of $\alpha$, denoted $\alpha_{\rbc}$, is defined to be the image of $\beta$ in the unramified cohomology group $\HH^2_{\nr}\left(F, \mu_{\ell}\right)$ under the isomorphism $i_F : \HH^2(k, \mu_{\ell})\to \HH^2_{\nr}(F, \mu_{\ell}), \  \beta\leadsto \alpha_{\rbc}$.

\begin{lemma}[\cite{S07}, Proof of Proposition 0.6]
\label{lemmaresdiualbrauerclass}
\[ \alpha = \alpha_{\rbc} +  \left(E, \sigma, \pi\right) \ \mathrm{in} \ \Br(F). \]
\end{lemma}

\subsection{Norms, reduced norms and index computations}

\begin{lemma}[cf. \cite{PPS}, Lemma 2.7]\label{lemmasimultaneousnormoneandlthpower}
Let $F$ be a field and $\ell$, a prime not equal to the characteristic of $F$ . Let $Y/F$ be a cyclic extension of $F$ or the split extension of degree $\ell$ and $\psi$, a generator of the Galois group of $Y/F$. Suppose that there exists an integer $m \geq 1$ such that $F$ does not contain a primitive ${\ell^{m}}^{\mrm{th}}$ root of unity. Let $\mu\in Y$ with $\N_{Y/F}(\mu)=1$. Further assume that
\begin{itemize}
\item
If $Y/F$ is split, then $\mu = \brac{g_i^{\ell}}\in \prod F$ for some $g_i\in F$.
\item
If $Y/F$ is not split, then $\mu=g^{\ell^{2m}}$ for some $g\in Y$.
\end{itemize}
Then there exists $h\in Y/F$ such that $\mu = h^{-\ell}\psi\brac{h}^{\ell}$.
\end{lemma}


\begin{lemma}[Totally ramified extensions (dim 1)] Let $R$ be a complete discretely valued ring with fraction field $K$ and residue field $k$. Let $\ell$ be a prime which is not divisible by $\charac(k)$ such that $K$ contains a primitive $\ell^{2\mathrm{th}}$ root of unity. Let $L/K$ be a totally ramified extension of degree $\ell$ and let $S$ be the integral closure of $R$ in $L$.  Then 
\begin{enumerate}
\item[a.]
$L\simeq K\left(\sqrt[\ell]{\pi}\right)$ for some parameter $\pi$ of $K$,
\item[b.]
If $x\in R^*$ is a norm from $L$, then $x\in K^{*\ell}$,
\item[c.]
Norm one elements in $L$ are $\ell^{\mrm{th}}$ powers in $S^*$.
\end{enumerate}
\label{lemmanormoneramified-dim1} 
\label{lemmanormfromramified-dim1} 
\label{totallyramifieddegreel}
\end{lemma}
\begin{proof}
a. follows from (\cite{PPS}, Lemma 2.4), while b. and c. are easy consequences of Hensel's lemma.  \end{proof} 

\begin{lemma}
\label{lemmanormfromE-split}
Let $A$ be a complete regular local ring of $\dim 2$ with fraction field $F$ and finite residue field $k$. Let $L/F$ be a cyclic extension of $F$ of degree $\ell$ unramified on $A$, where $\ell$ is a prime not divisible by $\charac(k)$. If $a\in A^*$, then it is a norm from  $L$.
\end{lemma}
\begin{proof} Let $\sigma$ be a generator of $\Gal(L/F)$. Since $a\in A^*$, the cyclic algebra $\brac{L, \sigma, a}$ is unramified and hence trivial in $\Br\brac{F}$.\end{proof}

\begin{lemma}[Norm one elements of an unramified extension]\label{lemmanormoneunramified}
Let $A$ be a complete regular local ring with fraction field $F$ and finite residue field $k$. Let $\ell$ be a prime which is not divisible by $\charac(k)$. Assume $F$ contains a primitive $\ell^{\mrm{th}}$ root of unity. If $Y$ is a degree $\ell$ field extension of $F$ unramified on $A$, then norm one elements of $Y/F$ which are integral over $A$ are $\ell^{\mathrm{th}}$ powers in $Y$.
\end{lemma}
\begin{proof}
Let $B$ denote the integral closure of $A$ in $Y$ and let $k_1$ be its residue field. Let $c\in Y$ be integral over $A$ such that $\N_{Y/F}(c)=1$. Hence $c\in B^*$, the minimal polynomial $g(t)$ of $c$ in $Y/F$ lies in $A[t]$ and is monic and irreducible. By the Henselian property of $A$, $\overline{g}(t)$ is irreducible (of the same degree) in $k[t]$ and is therefore the minimal polynomial of $\overline{c}$. 

Since $Y/F$ is unramified, $[Y:F] = [k_1:k] =\ell$ and therefore $\N_{k_1/k}(\overline{c}) =  (-1)^{\ell} \overline{g}(0)^{[k_1: k(\bar{c})]} = \overline{(-1)^{\ell} g(0)^{[Y:F(c)]}}  = \overline{\N_{Y/F}(c)} = 1$. Now $k_1/k$ is an extension of finite fields and hence the norm map $\N : k_1^*\to k^*$ is surjective. Since $\N$ is also multiplicative, it induces a surjective map of groups $\tilde{\N} : \frac{k_1^*}{k_1^{*\ell} }\to \frac{k^*}{k^{*\ell}}$. Since ${\ell}$ is a prime not divisible by $\charac(k)$ and $F$ contains a primitive $\ell^{\mrm{th}}$ root of unity, $\ell$ divides $|k^*|$ and $|k_1^*|$. Thus both $\frac{k_1^*}{k_1^{*{\ell}} }$ and $\frac{k_1^*}{k_1^{*{\ell}} }$ are cyclic groups of order ${\ell}$ which shows that $\tilde{N}$ is injective as well. 

Since $\N_{k_1/k}(\overline{c}) = 1$, this shows that $\overline{c} = \lambda^{\ell}$ for some $\lambda\in k_1$. Using the fact that $B$ is Henselian as well, we see that $c$ is also therefore an ${\ell}^{\mathrm{th}}$ power in $Y$. \end{proof}

\begin{lemma}[Norm one elements (dim 2)]
\label{lemmanormfromE-nonsplit}
\label{lemmanormoneramified} 
Let $A$ be a complete regular local ring of $\dim 2$ with fraction field $F$ and finite residue field $k$. Let $\ell$ be a prime not divisible by $\charac(k)$. Assume $F$ contains a primitive $\ell^{2\mrm{th}}$ root of unity. Let $Y=F\left(\sqrt[\ell]{u\pi^i\delta^j}\right)$ be a degree $\ell$ field extension of $F$ where $u\in A^*$, $\left(\pi,\delta\right)$ form a system of parameters of $A$ and $0\leq i, j\leq \ell-1$. Let $b\in Y$ be such that it is integral over $A$. If $\N_{Y/F}(b)=1$, then $b\in Y^{*\ell}$. \end{lemma}
\begin{proof}
We split it into two cases depending on the ramification of $Y/F$. If $Y/F$ is unramified and nonsplit, then by Lemma \ref{lemmanormoneunramified}, $b$ in an $\ell^{\mathrm{th}}$ power.

If $Y/F$ is ramified, then $Y=F\left(\sqrt[\ell]{u\pi^i\delta^j}\right)$ where $0\leq i, j\leq \ell-1$ with at least one of them non-zero. Let $B$ denote the integral closure of $A$ in $Y$. It is a complete local ring (\cite{HS06}, Theorem 4.3.4) with maximal ideal $\mathcal{M}_B$ and residue field $k$. Let $b\in B$ such that $\N_{Y/F}(b)=1$. Let $a\in A^*$ be such that $\overline{a}=\overline{b}$. Thus $ba^{-1} \simeq 1 \mod \mathcal{M}_B$. Since $B$ is complete and $\charac(k)\neq \ell$, $b=a\lambda^{{\ell}^2}$ for some $\lambda\in B$. 
This implies that $\N_{Y/F}(b) = \left(a \N_{Y/F}(\lambda)^{\ell}\right)^{\ell} = 1$. Thus $a\N_{Y/F}(\lambda)^{\ell}=\rho$ where $\rho$ is an ${\ell}^{\mrm{th}}$ root of unity. Hence $a$ is equal to $\rho$ up to ${\ell}^{\mrm{th}}$ powers in $Y$. Since $F$ contains a primitive $\ell^{2\mrm{th}}$ root of unity, this shows $a$ and hence $b$ is an $\ell^{\mathrm{th}}$ power in $Y$. \end{proof}

\begin{lemma}[Reduced norms of an unramified algebra]\label{lemmareducednormsunramified}
Let $R$ be a complete discretely valued ring with fraction field $K$ and residue field $k$ of cohomological dimension $\leq 2$. Let $D_0$ be an unramified central simple algebra over $K$ of index $\ell$ where $\ell$ is a prime not divisible by $\charac(k)$. Then every unit $u\in R^*$ is a reduced norm from $D_0$
\end{lemma}
\begin{proof} By the results of Merkurjev and Suslin (\cite{Se}, Chapter II, Sec 4.5, Pg 88), the reduced norm of $\overline{D_0}$ is surjective. Thus the polynomial $\overline{\Nrd_{D_0}(\mathbf{x}) - u}=0$ has a solution over $k$. By Hensel's Lemma, there exists a solution over $K$. \end{proof}


\begin{lemma}[Splitting fields]
\label{lemmasplittingfields2}
Let $A$ be a complete regular local ring of $\dim 2$ with fraction field $F$ and finite residue field $k$. Let $\ell$ be a prime not divisible by $\charac(k)$ such that $F$ contains a primitive $\ell^{\mrm{th}}$ root of unity. Let $D = (v, \pi)$ be an $\ell$ torsion algebra over $F$ and $E=F\left(\sqrt[\ell]{u\pi^i\delta^j}\right)$ be a degree $\ell$ field extension of $F$ where $u, v\in A^*$, $\left(\pi,\delta\right)$ form a system of parameters of $A$ and $0\leq i, j\leq \ell-1$. Let $\widehat{A_{(\pi)}}$ be completion of $A_{(\pi)}$ at its maximal ideal and let its fraction field be denoted by $F_{B}$, which is a complete discretely valued field with parameter $\pi$ and residue field $k_B$. If $D\otimes_F (E\otimes_F F_{B})$ is split, then so is $D\otimes_F E$.
\end{lemma}
\begin{proof} If $i=j=0$, then $E$ is the unique unramified (on $A$) extension of $F$. Therefore $v\in E^{*\ell}$ and hence $D\otimes E = 0$. 

If $i=0, j\neq 0$, then $E\otimes F_{B}/F_B$ is an unramified extension, $\overline{E_B} :=\overline{E\otimes F_B}$ is a totally ramified extension over $k_B$ and the further residue field of $\overline{E_B}$ is $k$ . Note that  $D\otimes E\otimes F_{B}=0$ implies that the residue $\overline{v}\in \overline{E_B}^{*\ell}$ and hence $\overline{\overline{v}}\in k^{*\ell}$. This implies $v\in A^{*\ell}$ and hence $D = 0$ to begin with.

If $i\neq 0$, without loss of generality we can assume $i=1$ and $0\leq j < \ell$. Thus $D\otimes E = \brac{v, u^{-1}\delta^{-j}} = \brac{\delta, v^j}\in \Br(E)$. Note that $E\otimes F_B/F_B$ is totally ramified and $\overline{E_B} = k_B$. Since $D\otimes E\otimes F_B = 0$, we see that $\brac{\delta, v^j} = 0 \in \Br(E\otimes F_B)$ and hence $\brac{\overline{\delta}, \overline{v}^j}\in \Br(k_B)$. This implies $\overline{\overline{v}}^j\in k^{*\ell}$ and hence $v\in A^{*\ell}$. Hence $D\otimes E=0$. \end{proof}

\begin{lemma}[Index formula, \cite{JW90}]
\label{lemmaindexformula}
Let $R$ be a complete discretely valued ring with fraction field $F$. Let $E$ be a cyclic unramified extension of $F$ of degree $m$ and let $\alpha = \alpha' + \left(E,\sigma, \pi \right)$ in $\Br(F)$ where $\pi$ is a parameter of $R$, $\sigma$ is a generator of $E/F$ and $\alpha'$ is a central simple algebra of degree $n$ unramified at $R$. Assume $mn$ is invertible in $R$. Then $\ind\left(\alpha\right) = \ind(\alpha'\otimes_F E)\left[E:F\right]$. 
\end{lemma}




\subsection{Approximating local data}
For the rest of this section, $\ell$ will denote an odd prime, $F$, a global field with $\charac(F)\neq \ell$ containing a primitive $\ell^{\mrm{th}}$ root of unity and $D'$, a central simple algebra over $F$ of index dividing $\ell$. $F_v$ will denote the completion of $F$ at a place $v$ of $F$ and $k_v$, its residue field. $T=\{v_1, v_2, \ldots ,v_r\}$ will be a finite set of places of $F$ such that $\ell\neq \charac(k_{v_i})$ for each $i\leq r$ and $D'\otimes F_{v}$ is split for every place $v\not\in T$.

\begin{lemma}[An approximate cyclic extension]\label{lemmauglyinitial}

Suppose that there exists $u'\in F^*$ and cyclic or split extensions $E_{v_i}/F_{v_i}$ of degree $\ell$ for each $v_i\in T$ such that
\begin{itemize}
\item
$u'$ is a norm from $E_{v_i}/F_{v_i}$ for each $v_i\in T$,
\item
$D'\otimes_F E_{v_i}$ is split for each $v_i\in T$.
\end{itemize}

Then there exists a cyclic field extension $E/F$ of degree $\ell$ such that 
\begin{itemize}
\item
$E\otimes_F F_{v_i} \simeq E_i$ for each $v_i\in T$,
\item
$u'$ is a norm from $E/F$,
\item
$D'\otimes_F E$ is split.
\end{itemize}

\end{lemma}

\begin{proof}
Without loss of generality, assume that there exists a $v\in T$ such that $E_v/F_v$ is a field extension. This can be done by expanding $T$ to include a place $v$ of $F$ where $u'\in \mathcal{O}_{F_v}^*$ and choosing $E_v$ to be the unique cyclic unramified field extension of degree $\ell$ over $F_v$.

Pick $w_v$ to be so that the given $E_v\simeq \frac{F_v[t]}{\brac{t^{\ell}-w_v}} $ for each $v\in T$. If $u'\in F^{*\ell}$, using weak approximation pick $w\in F$ so that  up to $\ell^{\mrm{th}}$ powers, it matches $w_v\in F_v$ for each $v\in T$. Then the field $E=F[t]/\left(t^{\ell}-w\right)$ satisfies the lemma. So we assume $u'\not\in F^{*\ell}$ in the rest of the proof.

For each place $v\in T$, by hypothesis we know $\left(w_v, u'\right)=0\in \Br\left(F_v\right)$. Hence pick $\theta_v\in \brac{F\left(\sqrt[\ell]{u'}\right)\otimes F_v}^{*}$ so that $\N_{F\left(\sqrt[\ell]{u'}\right)\otimes F_v/F_v}\left(\theta_v\right) = w_v$. By weak approximation, find $\theta\in F\left(\sqrt[\ell]{u'}\right)$ so that it matches $\theta_v$ up to $\ell^{\mrm{th}}$ powers. Set $w=\N_{F\left(\sqrt[\ell]{u'}\right)/F}(\theta)$. Thus $w$ matches with $w_v$ up to $\ell^{\mrm{th}}$ powers and $(u',w)=0\in \Br(F)$.

Set $E=F\left(\sqrt[\ell]{w}\right)$. This is a cyclic Galois extension of $F$ which approximates the $E_v$s for each $v\in T$. By hypothesis, $D'$ is split at places not in $T$ and $E_v\otimes_F D'$ is split for every $v\in T$. Thus $E$ splits $D'$. \end{proof}

\begin{lemma}[Another approximate cyclic extension]
\label{lemmaugly}
Let $Y'=F\left(\sqrt[\ell]{u'}\right)$ be a cyclic field extension of degree $\ell$ where $u'\in  F^*\setminus F^{*\ell}$. Let $a'\in Y'^{*}\setminus Y'^{*\ell}$ and $L$ be the Galois closure of the compositum $Y'\left(\sqrt[\ell]{a'}\right)$ over $F$.  Suppose that for each $v\in T$, there exist $w_v\in F_v^*$ and extensions $E_v := \frac{F_{v}[t]}{(t^{\ell}-w_{v})}$ of $F_v$ with the following properties: 

\begin{itemize}
\item
$w_v$ is a norm from $L\otimes F_v/F_v$
\item
$D'\otimes_F E_v $ is split
\item
$\left(w_v, a'\right)$ is split over $Y'\otimes F_v$.
\end{itemize}

Then there exists a cyclic field extension $E/F$ of degree $\ell$ such that
\begin{itemize}
\item
$E\otimes_F F_{v} \simeq E_v$ for each $v\in T$,
\item
$u'$ is a norm from $E/F$,
\item
$D'\otimes_F E$ is split,
\item
$a'$ is a norm from $E\otimes_F Y'/ Y'$.
\end{itemize}

\end{lemma}

\begin{proof}
Without loss of generality, assume that there exists a $v\in T$ such that $E_v/F_v$ is a field extension. This can be done by expanding $T$ to include a place $v$ of $F$ with the following properties : 1) $u'\in \mathcal{O}_{F_v}^*$, 2) $a'\in \mathcal{O}_{Y'_x}^*$ for any place $x$ of $Y'$ lying over $v$, 3) $L\otimes F_v$ is a unramified extension of $F_v$ (or a product of unramified extensions over $F_v$) and choosing $w_v\in O_{F_v}^*\setminus O_{F_v}^{*\ell}$ and $E_v$ to be the unique cyclic unramified field extension of degree $\ell$ over $F_v$.

Let $z_v\in (L\otimes F_v)^*$ such that $\N_{L\otimes F_v/F_v}(z_v)=w_v$. By weak approximation, find $z\in L$ so that it matches up to $\ell^{\mrm{th}}$ powers with $z_v$ for each $v\in T$. Set $\theta := \N_{L/Y'}(z)$ and set $w:= \N_{Y'/F}(\theta) = \N_{L/F}(z)$ . Thus $w$ matches with the $w_v$ upto $\ell^{\mrm{th}}$ powers. Clearly $w$ is a norm  from $Y'=F\brac{\sqrt[\ell]{u'}}$ also and hence $(u',w)=0\in \Br(F)$. Set $E=F\left(\sqrt[\ell]{w}\right)$. Hence $u'$ is a norm from $E/F$. 

Note that $E$ is an extension of $F$ which approximates the given $E_v$ for each $v\in T$. Since there exists some $v\in T$ such that $E_v$ is a field, $E/F$ is a nonsplit field extension, which is clearly cyclic of degree $\ell$. By hypothesis, $D'/F$ is split at places not in $T$ and $E_v\otimes_F D'$ is split for every $v\in T$. Thus $E$ splits $D'$.

As $\theta=\N_{L/Y'}(z)$ and $Y' \subseteq Y'\left(\sqrt[\ell]{a'}\right)\subseteq L$, we have that $(a',\theta)=0\in \Br(Y')$. Given any $\psi\in \Gal(Y'/F)$, extend it to some $\tilde{\psi}\in \Gal(L/F)$. Then $\N_{L/Y'}\left(\tilde{\psi}(z)\right)=\psi(\theta)$. Hence $\psi(\theta)$ is a norm from $L/Y'$ and so also from $Y'\left(\sqrt[\ell]{a'}\right)/Y'$. Therefore 
\[\left(a', \psi(\theta) \right) = 0 \in \Br(Y') \ \forall \ \psi \in \Gal(Y'/F).\]

Finally, since $\N_{Y'/F}(\theta)=w$ and $Y'/F$ is Galois, we have that $\prod_{\psi\in \Gal(Y/F)}\psi(\theta) = w$. Therefore \[\prod_{\psi\in \Gal(Y',F)} \left(a', \psi(\theta)\right) = (a',w) = 0 \in \Br(Y').\]\end{proof}

\begin{lemma}[Invariant algebras of global fields]\label{lemmainvariantalgebras}
Let $E/F$ be a cyclic extension of global fields of degree $\ell$, where $\ell$ is a prime not divisible by any of the residual characteristics of $F$. Further assume that $F$ contains a primitive $\ell^{\mrm{th}}$ root of unity. Let $E_w$ denote the completion of $E$ at any place $w$ of $E$ and $\mathcal{O}_{E_w}$, its valuation ring. Let $\Gal\left(E/F\right) = \langle \sigma \rangle$. Let $u\in F^{*}$, $b\in E^{*}$ be such that

\begin{itemize}
\item
At every place $w$  of $E$ where $E/F$ is ramified, $u$ is an $\ell^{\mathrm{th}}$ power in ${E}_w^{*}$,
\item
At every place $w$ of ${E}$ where ${E}/F$ is unramified and inert, $u\in \mathcal{O}_{{E}_w}^*$ upto $\ell^{\mathrm{th}}$ powers in ${E}_w^{*}$,
\item
$(u, b) = (u,  \sigma (b)) \ \mathrm{in} \  \HH^2 \left( {E}, \mu_{\ell} \right)$.
\end{itemize}

Additionally, let $T_0$ be a finite set of places of $F$ such that for each place $v\in T_0$, one is given $f_v\in F_v^{*}$ such that for  any place $w$ of ${E}$ lying above $v$, $(u, b) = (u, f_v)\ \mathrm{in} \ \HH^2\left({E}_w, \mu_{\ell}\right)$.

Then there exists $f\in F^*$ such that 
\begin{enumerate}
\item
$f=f_v\theta_v^{\ell}$ in $F_v$ for some $\theta_v\in F_v$ for each $v\in T_0$,
\item
$(u, b) = (u, f) \ \mathrm{in} \ \HH^2\left({E}, \mu_{\ell}\right)$.
\end{enumerate}
\end{lemma}

\begin{proof}
By Kummer theory, ${E}=F\left(\sqrt[\ell]{\psi}\right)$ for some $\psi$ which generates  $\frac{F^*}{F^{*\ell}}$. Note that if $u\in {E}^{*\ell}$, we can choose $f$ by weak approximation such that $f$ matches $f_v$ up to $\ell^{\mathrm{th}}$ powers. So for the remainder of the proof, we assume that $u\not\in {E}^{*\ell}$. We also note that if $v\in T_0$ splits completely in ${E}$, then the hypothesis that $(u,b)=(u,\sigma(b))$ guarantees that the same $f_v$ works for each place $w$ above $v$.


Let $T$ denote the union of $T_0$ and the finite set of places $v$ of $F$ which satisfy both the following conditions: 1) $v$ is either unramified and inert or completely split in ${E}$, 2) There exists a place $w$ of ${E}$ lying above $v$ at which either $u$ or $b$ is not a unit in $\mathcal{O}_{{E}_w}$.

For each place $v\in T$, we find $f_v\in F_v^*$  as follows:\\

\textbf{Case 0} : For each $v\in T_0$, we choose the $f_v$ given by the hypothesis.\\

\textbf{Case I} : Let $v\in T\setminus T_0$ be a place of $F$ which is unramified and inert in ${E}$ and let $w$ be the place above $v$. Let $b= \tilde{b}_w \pi_v^s$ where $\tilde{b} \in \mathcal{O}_{{E}_w}^*$ and $\pi_v$ is a parameter for $F_v$. Set $f_v = \pi_v^{s}$. Since by hypothesis, $u$ is in $\mathcal{O}_{{E}_w}^*$ upto $\ell^{\mathrm{th}}$ powers, we have that

\[\left(u, b\right) = \left(u, \tilde{b}_w\right)+\left(u, \pi_v^s\right) =  \left(u, f_v\right) \in \HH^2\left({E}_w, \mu_{\ell}\right) \]

\textbf{Case II} : Let $v\in T\setminus T_0$ be a place of $F$ which splits in ${E}$. Thus ${E}\otimes_F F_v = \prod_{i=1}^{\ell} F_v$ and let $b\in {E} = \left(b_1,b_2, \ldots, b_{\ell} \right)\in {E}\otimes_F F_v$. Thus $(u,b) = (u, \sigma(b))$ implies 
\[\left(u,b_1\right) = \left(u, b_2\right) = \ldots =   \left(u, b_{\ell}\right) \in \HH^2\left(F_v, \mu_{\ell}\right) \hfill{\phantom{- - - - - - - -}(*)}\]

Set $f_v = b_1$. And thus $\left(u, f_v\right)$ matches $(u,b_i)$ over $F_v$ for each $i$.

Since by hypothesis, $u\in {E}_w^{*\ell}$ for every $w/v$ totally ramified, we have that for each place $w$ lying over a place $v$ not in $T$, $(u,b)$ is split over ${E}_w$. 

Note that since $u\not\in {E}^{*\ell}$, it is not in $F^{*\ell}$ either and hence $F\left(\sqrt[\ell]{u}\right)$ is a cyclic Galois extension of degree $\ell$ over $F$. Then $L:=F\left(\sqrt[\ell]{u}, \sqrt[\ell]{\psi} \right)$ is a Galois extension over $F$ with Galois group $\frac{\mbb{Z}}{\ell\mbb{Z}} \times \frac{\mbb{Z}}{\ell\mbb{Z}}$. 

By Chebotarev density, pick a place $\tilde{v}$ of $F$ (there are infinitely many!) which is not in $T$ such that 1) $\tilde{v}$ does not ramify in $L$, 2) $\sigma' \in \Gal(L/F) \leftrightarrow (1,1) \in  \frac{\mbb{Z}}{\ell\mbb{Z}} \times \frac{\mbb{Z}}{\ell\mbb{Z}}$ is the \textit{Frobenius automorphism $\mrm{Frob}_{\tilde{v}}$} of $L_x/F_{\tilde{v}}$ where $x$ is any place lying above $\tilde{v}$. [Since $L/F$ is abelian, the Frobenius automorphism does not depend on the choice of $x$]

Note that the residue field extension degree $\left[l_x : k_{\tilde{v}}\right]\leq \ell$. For if ${E}_w$ is nonsplit unramified extension of $F_{\tilde{v}}$, then since $u\in \mathcal{O}_{F_{\tilde{v}}}^*$, we have $u\in {E}_w^{*\ell}$.

We have chosen $\sigma'$ to be the non-trivial automorphism of $L/F$ of order $\ell$ such that $\sigma'\left(\sqrt[\ell]{u}\right) = \rho\ \sqrt[\ell]{u} $ and $\sigma'\left(\sqrt[\ell]{\psi}\right) = \rho'\ \sqrt[\ell]{\psi}$, where $\rho, \rho'$ are primitive $\ell^{\mrm{th}}$ roots of unity.  This, by the choice of $\tilde{v}$ gives rise to the Frobenius automorphism of the residue field extensions $l_x/k_{\tilde{v}}$. Thus $l_x/k_{\tilde{v}}$ is a non-trivial extension, i.e., $\tilde{v}$ is not completely split in $L$. (The choice of Frobenius for a trivial extension is the identity map.)

Thus the residue field extension $l_x/k_{\tilde{v}}$ is degree $\ell$ with Galois group generated by $\overline{\sigma'}$. Note that since $\sigma'$ fixes neither $\sqrt[\ell]{u}$ nor $\sqrt[\ell]{\psi}$, $u$ and $\psi$ are not $\ell^{\mrm{th}}$ powers in $F_v$. Thus ${E}_w$ is unramified, nonsplit over $F_{\tilde{v}}$ and $u\not\in F_{\tilde{v}}^{*\ell}$.

\subsubsection*{Finding \textit{an} $f'$}
Our first goal is to find an algebra $\alpha=(u,f') \in \Br(F)$ such that $\alpha\otimes_F {E} = (u,b) \in \Br({E})$. We find $\alpha$ by prescribing its shape $\alpha_v$ locally so that $\alpha_v\otimes_F {E} = (u,b) \in \Br\left({E}_w\right)$ where $w$ is any place lying over $v$.

For $v\in T$, choose $\alpha_v = \left(u, f_v\right)$. So $\alpha_v \otimes_F {E} = (u,b)$ in $\Br\left({E}_w\right)$. For $v\not\in T$ and $v\neq \tilde{v}$, choose $\alpha_v = 0\in \Br\left(F_v\right)$. This matches with $(u,b)$ over ${E}_w$ since the latter is also split at these places. For $v = \tilde{v}$, let $\pi_{\tilde{v}}$ be a parameter of $F_{\tilde{v}}$. Choose $\alpha_{\tilde{v}} = \left(u, \pi_{\tilde{v}}^s \right)$ for an appropriate $s$ so that $\sum_{v\in \Omega_F} \inv\left(\alpha_v\right) = 0 \in \frac{\mbb{Q}}{\mbb{Z}}$. This can be done since $u$ is a unit at $\tilde{v}$ and $u\not\in F_{\tilde{v}}^{*\ell}$.

Note that $(u,b)$ is split over ${E}_w$ for any $w|\tilde{v}$. Now since $\Cor$ is injective for local fields, $\alpha_{\tilde{v}}:= \left(u, \pi_{\tilde{v}}^s \right)$ is split over ${E}_w$ for $w|\tilde{v}$ because $\Cor : \HH^2\left({E}_w, \mu_{\ell}\right)\to \HH^2\left(F_{\tilde{v}}, \mu_{\ell}\right)$ sends  $\left(u, \pi_{\tilde{v}}^s  \right)\leadsto \left(u, \pi_{\tilde{v}}^{s\ell}\right) = 0$.

By the Albert-Hasse-Brauer-Noether theorem, there exists an $\alpha\in \Br(F)$ of order dividing $\ell$ such that $\alpha\otimes_F F_v = \alpha_v \in \Br\left(F_v\right)$. Also note that locally at each place, it is split by $F\left(\sqrt[\ell]{u}\right)$, hence there exists an $f'\in F$ such that $\alpha = (u, f') \in \Br(F)$ since $F$ contains a primitive $\ell^{\mrm{th}}$ root of unity.

\subsubsection*{Modifying $f'$ so that it approximates $f_v$ for each $v\in T$}
By the choice of $f'$, we have that $(u,f') = \left(u,f_v\right)\in \Br\left(F_v\right)$ for each $v\in T$. Hence for each $v\in T$, there exists $w_v \in F\left(\sqrt[\ell]{u}\right)\otimes F_v$  such that  $\N_{F_v\left(\sqrt[\ell]{u}\right)\otimes F_v/F_v}\left(w_v\right) = f'^{-1}f_v$.

By weak approximation, there exists a $w\in F\left(\sqrt[\ell]{u}\right)$ such that for each $v\in T$, $w=w_v\gamma_v^{\ell}$ for $\gamma_v\in F\left(\sqrt[\ell]{u}\right)\otimes F_v$. This therefore implies $\N_{F\left(\sqrt[\ell]{u}\right)/F}(w) \in \ f'^{-1} f_v F_v^{*\ell}\ \forall \ v\in T$.

Finally, set $f=f'\N(w)$. Therefore $(u,f') = (u,f)\in \Br(F)$ and $f \in f_vF_v^{*\ell}\ \forall \ v\in T$. \end{proof}

\section{Reductions and strategies}

\subsection{The set-up}
Let $K$ be a $p$-adic field with ring of integers $\mathcal{O}_K$ and residue field $k$. Let $F=K(X)$ be the function field of a smooth projective geometrically integral curve $X$ over $K$. Let $D$ denote a central division algebra over $F$ of exponent $\ell$ where $\ell$ is a prime different from $p$. We want to prove triviality of $\SK1(D)$. Since it is known that the index of $D$ divides ${\ell}^2$ (\cite{S97}, \cite{S98}) and that $\SK1(D)$ is trivial for square-free index algebras (\cite{W}), we assume from now on that the index of $D$ is ${\ell}^2$.

Note that in the case when $\ell=2$, the works of Merkurjev and Rost (\cite{Me03}, \cite{Me06}, Rost, Chapter 17 \cite{KMRT} ; \cite{Merkurjev}) lead to the more general result that $\SK1(D)=\{0\}$ over cohomological dimension $3$ fields. Thus, in this paper, we assume that $\ell\neq 2$. \textit{We also make an additional assumption that $F$ contains a primitive ${\ell^{2}}^{\mrm{th}}$ root of unity.}

Let $z\in \SL1(D)$ and let $M$ be a maximal subfield of $D$ containing $z$. Thus  $\N_{M/F}(z) = \Nrd_{D}(z)= 1$. We would like to show $z\in [D^*, D^*]$. Using (\cite{Plat76}, Lemma 2.2, Section 2.4) and (\cite{Albert}, Chapter IV, Theorem 31), by a coprime to $\ell$ base change, we assume that $M$ contains a cyclic degree $\ell$ sub-extension $Y/F$ with $\Gal(Y/F) = \langle \psi \rangle$. Since $F$ contains a primitive $\ell^{\mrm{th}}$ root of unity, by Kummer theory, we have $Y = F\brac{\sqrt[\ell]{y}}$ for some $y\in F^*$. Since $\N_{M/F}(z)=1$, the element $a := \N_{M/Y}(z)$ is a norm one element of $Y/F$ and by Hilbert 90, $a = b^{-1}\psi(b)$ for some $b\in Y$. 

\subsubsection{A preliminary model}  
 By resolution of singularities (\cite{Lip75}), there exists a regular integral scheme\footnote{We would like to note in advance that we will finally work over a new model obtained from $\mathcal{X}$ by repeatedly blowing up closed points.}  $\mathcal{X}$ with function field $F$ equipped with a proper, flat and projective morphism  $\mathcal{X}\to \Spec\mathcal{O}_K$. Let $X_0$ denote its reduced special fiber. For each $x\in \mathcal{X}$, let the regular local ring at $x$ on $\mathcal{X}$ be denoted by $A_x := \mathcal{O}_{\mathcal{X},x}$. Let the completion of $A_x$ at its maximal ideal be denoted by $\widehat{A_{x}}$, the fraction field of $\widehat{A_{x}}$ by $F_x$ and the residue field of $\widehat{A_{x}}$ by $k_x$. We also let $D_x$ (resp. $Y_x$) denote $D\otimes_F F_x$ (resp. $Y\otimes_F F_x$). If $\eta\in X_0$ is a codimension one point of $\mathcal{X}$ and $P\in X_0$ is a closed point of $\mathcal{X}$ with $P$ lying in the Zariski closure of $\eta$ in $\mathcal{X}$, we let $F_{P,\eta}$ denote the \textit{branch} field. More explicitly, if $\left(\pi_{\eta}\right) \in \widehat{A_P}$ denotes a prime defining $\eta$, then localization at this prime ideal yields a discrete valuation ring $\widehat{A_P}_{(\pi_{\eta})}$. Completing this discrete valuation ring at its maximal ideal and further taking its field of fractions yields the branch field $F_{P,\eta}$. Thus $F_P$ and $F_{\eta}$ are both subfields of $F_{P,\eta}$. Let $k_{P,\eta}$ denote the residue field of $F_{P,\eta}$.

Since $\mathcal{X}$ is normal, any codimension one point $x\in \mathcal{X}^{(1)}$ extends to a\footnote{In case the prime corresponding to $x$ splits in $Y$, then $x$ defines $\ell$ valuations $v_{x_1}, v_{x_2}, \ldots, v_{x_{\ell}}$ on $Y$. However $v_{x_i}(b)=v_{x_1}(\psi^{-i+1}(b))$. Set $v_{(x)} := v_{x_1}$.} discrete valuation $v_{(x)}$ on $Y$. Define $\supp_{\mathcal{X}}(b):=\{x\in \mathcal{X}^{(1)} |  \max_i\brac{|v_{(x)}(\psi^{i}(b))|} > 0\}$ and let $\mathcal{J}_{\mathcal{X}} := \sum_{x\in \supp_{\mathcal{X}}(b) } x$. Further set $\mathcal{H}_{\mathcal{X}}$ to be the divisor corresponding to the union of the reduced special fiber $X_{0}$, $\divi_{\mathcal{X}}(y)$, $\mathcal{J}_{\mathcal{X}}$, the ramification locus of $M$ and the ramification divisor of $D$ in $\mathcal{X}$.

\begin{proposition} 
\label{propgoodshape}
There exists a regular proper model $\mathcal{X}$ of $X$ over $\mathcal{O}_K$ such that $\mathcal{H}_{\mathcal{X}}$ is a union of regular curves in normal crossing in $\mathcal{X}$. Further, let $h : \mathcal{Y}\to \mathcal{X}$ denote the normal closure of the model $\mathcal{X}$ in $Y$. Let $x\in \mathcal{X}$ of codimension $1\leq i\leq 2$ and let $B_x$ denote the integral closure of $A_x=\mathcal{O}_{\mathcal{X},x}$ in $Y$. Then the following hold:

\begin{enumerate}
\item[a.]
If $Y_x$ is a field, then $h^{-1}(x)=\{y\}$ where $y\in \mathcal{Y}$ of codimension $i$ and $B_x$ is a local ring and isomorphic to  $\mathcal{O}_{\mathcal{Y}, y}$.
\item[b.]
If $Y_x\simeq \prod F_x$, then $h^{-1}(x)=\{y_1, y_2, \ldots, y_{\ell}\}$, a set of $\ell$ points in $\mathcal{Y}$ of codimension $i$ and $B_x$ is semi-local with $\ell$ maximal ideals $m_{y_i}$ for $1\leq i \leq \ell$. Further, $(B_x)_{m_{y_i}}\simeq \mathcal{O}_{\mathcal{Y}, y_i}$.
\end{enumerate}

\end{proposition}
\begin{proof}
Fix a preliminary regular proper model $\mathcal{X}'$ of $X$ over $\mathcal{O}_K$. Construct $\mathcal{X}$ by blowing up $\mathcal{X}'$ at closed points of $\mathcal{X}'$ repeatedly ($p : \mathcal{X} \to \mathcal{X}'$) such that $\mathcal{H}'':= p^{-1}\brac{\mathcal{H}_{\mathcal{X}'}}$ is a union of regular curves in normal crossing. To prove that $\mathcal{H}_{\mathcal{X}}$ is in good shape, it suffices to show that $\mathcal{H}_{\mathcal{X}}\subseteq \mathcal{H}''$. By construction, the union of $X_0$, $\divi_{\mathcal{X}}(y)$, $\ram_{\mathcal{X}}(M)$ and $\ram_{\mathcal{X}}([D])$ lies in $\mathcal{H}''$.



Now let $\beta\in \mathcal{J}_{\mathcal{X}}$. If $\beta$ is the generic point of the strict transform of a curve in $\mathcal{X}'$, then $p\brac{\beta}\in \mathcal{J}_{\mathcal{X}'}$ and hence $\beta\in \mathcal{H}''$. On the other hand, if $\beta$ lies on an exceptional curve of $p: \mathcal{X}\to \mathcal{X}'$, then clearly $\beta\in \mathcal{H}''$. Hence $\mathcal{H}_{\mathcal{X}}$ is in good shape. 

We give the proof for the case when $x=P$, a closed point in $\mathcal{X}$. The proof for the case when $x$ has codimension one is similar. Let $U\subset \mathcal{X}$ be an open affine neighbourhood containing $P$ with coordinate ring $A$. Thus $h^{-1}(U)$ is affine with coordinate ring, say $B$, which is the integral closure of $A$ in $Y$. Thus it follows that the integral closure of the local ring $A_P$ in $Y$ is $B$ localized at the multiplicatively closed set $A\setminus P$ which we denoted by $B_P$. Since $B_P$ is integral over $A_P$, the maximal ideals of $B_P$ contract to the unique maximal ideal of $A_P$ and hence correspond to the points in $h^{-1}(P)$. Since $\Gal(Y/F)\simeq \mathbb{Z}/\ell \mathbb{Z}$ acts transitively on $h^{-1}(P)$, it is clear that $h^{-1}(P)$ is either a singleton or a set of size $\ell$. 

Now it only remains to compare the shape of $Y_P:= Y\otimes_F F_P$ and the size of $h^{-1}(P)$. By (\href{https://stacks.math.columbia.edu/tag/07N9}{Lemma 07N9}, stacks-project), $B_P\otimes_{A_P}\widehat{A_P}\simeq_{Q_i\in h^{-1}(P)} \prod \widehat{\mathcal{O}_{\mathcal{Y}, Q_i}}$ which is a (local) domain iff $|h^{-1}(P)|=1$. 

We have the following injective\footnote{As $\widehat{A_P}/A_P$ and $Y/A_P$ are flat and $M\otimes_R N\simeq M\otimes_S N$ for $S$-modules $M, N$ where $S$ is a localisation of $R$} $A_P$-morphism: $B_P\otimes_{A_P} \widehat{A_P}\hookrightarrow Y\otimes_{A_P}\widehat{A_P}\hookrightarrow Y\otimes_{A_P} F_P \simeq Y\otimes_F F_P := Y_P$. Thus if $Y_P$ is a field, $B_P\otimes_{A_P} \widehat{A_P}$ has to be a domain and hence $|h^{-1}(P)|=1$. Conversely, if $h^{-1}(P)=1$, then $B_P\otimes_{A_P} \widehat{A_P}$ has to be a local domain. The above injection shows that $Y_P\simeq Y\otimes_{A_P} F_P$ lies in the fraction field of $B_P\otimes_{A_P} \widehat{A_P}$. Hence $Y_P$ is a domain and hence a field. 
\end{proof}

We continue to work this model $\mathcal{X}$ till the end of Section 5.

\begin{lemma}
\label{lemmaatetasplit-1}
Let $P$ be a closed point in $\mathcal{X}$ lying on the Zariski closure of a codimension one point $\eta\in \mathcal{X}$. If $Y_{\eta}\simeq \prod F_{\eta}$, then $Y_P\simeq \prod F_P$. \end{lemma}
\begin{proof} Let $(\pi_P, \delta_P)$ be a system of parameters of $A_P$ such that $\pi_P$ cuts out the curve $\overline{\eta}$ at $P$. Recall that $Y=F(\sqrt[\ell]{y})$ and that $\divi(y)$ is arranged to be in good shape in $\mathcal{X}$.  Since $Y_{\eta}$ is split, so is $Y\otimes F_{P,\eta}$. Thus we can assume that up to $\ell^{\mathrm{th}}$ powers, $y = v_P\delta_P^j$ for some unit $v_P\in \widehat{A_P}^*$ and $0\leq j < \ell$ with $\overline{y}\in k_{P,\eta}^{*\ell}$. Recall that $k_{P,\eta}$ is a complete discretely valued field with $\overline{\delta_P}$ as a parameter. Thus $j=0$ and since $v_P\in \widehat{A_P}^*$, $\overline{v_P}\in k_P^{*\ell}$. Hence $v_P\in \widehat{A_P}^{*\ell}$. This immediately implies that $Y_P$ is split. \end{proof}

\subsubsection{Fixing parameters}
\label{subsubsection-modelparameter}

Let ${S_0=\{P_1,P_2,\ldots, P_m\}}$ denote the finite set of closed points of intersection of distinct irreducible curves in ${\mathcal{H}}_{\mathcal{X}}$. Expand $S_0$ if necessary so that it includes at least one closed point from each irreducible curve in ${\mathcal{H}}_{\mathcal{X}}$. We call the elements in the set $S_0$ to be \textit{intersection} points.

Let $N'_0$ denote the set of all codimension one points of $\mathcal{X}$  which lie in ${\mathcal{H}}_{\mathcal{X}}$ and let ${N_0}$ denote the subset $N'_0\cap X_0$. Using (\cite{S98}, Lemma), for each $\eta\in N'_0$, choose a function $\pi_{\eta}\in F$ such that $\divi_{\mathcal{X}}(\pi_{\eta}) = \overline{\eta} +  E_{\eta}$ where $E_{\eta}$ avoids $\mathcal{N}'\cup S_0$. Thus $\pi_{\eta}$ is a parameter of $F_{\eta}$ for each such $\eta$.

Further if $P\in S_0$ lies on two distinct irreducible curves $C$ and $C'$ in ${\mathcal{H}}_{\mathcal{X}}$ with generic points $\eta$ and $\eta'$ respectively. Then $(\pi_{\eta}, \pi_{\eta'})$ form a system of  parameters of $A_P$. If $P\in  S_0$ lies on exactly on one irreducible curve $C$ of $\mathcal{H}_{\mathcal{X}}$ with generic point $\eta$,  then $\pi_{\eta}$ can be extended to a system of parameters $(\pi_{\eta}, \pi_{\eta'})$ of $A_P$ for some prime $\pi_{\eta'}$ defining a curve $C'$ with generic point $\eta'$ cutting $C$ transversally.

We choose this system of parameters for each $P\in S_0$. Let $\pi_P:= \pi_{\eta}$ and $\delta_P:= \pi_{\eta'}$. Since ${\mathcal{H}}_{\mathcal{X}}$ is in good shape and at $P$, the division algebra is ramified at most along $C$ and $E$, using (\cite{S97}, Proposition 1.2) we see that there exist $\alpha'\in \Br\brac{A_P}$, $u_P, v_P\in A_P^*$ and integer $0 < m < \ell$ such that  
\[[D]\in \{ \alpha', \alpha' + \brac{u_P, \pi_{P}}, \alpha' + \brac{v_P, \delta_{P}}, \alpha' + \brac{u_P, \pi_{P}} + \brac{v_P, \delta_{P}}, \alpha'+ \brac{u_P\pi_{P}^m, v_P\delta_{P}}\}\subseteq \Br(F).\]

\subsection{The shape of $a$}
The following propositions specify the shape of $a=\N_{M/Y}(z)$ (which is an element of norm one in $Y/F$) over the model $\mathcal{X}$.

\begin{proposition} 
\label{propaisaunitifYnotsplit}
Let $x\in X_0$ be such that $Y_x$ is a field extension of $F_x$. Let $\widehat{B_x}$ be the integral closure of $\widehat{A_x}$ in $Y_x$. Then $a\in \widehat{B_x}^*$.
\end{proposition}
\begin{proof} Let us first look at the case when $x\in X_0$ is a  codimension one point of $\mathcal{X}$. Thus $F_x$ is a complete discretely valued field and therefore so is $Y_x$. Let $\pi_{Y_x}$ be a parameter of $Y_x$ and $\pi_{F_x}$ be a parameter of $F_x$. Thus $a = u_x \pi_{Y_x}^{j}$ for some $u_x\in \widehat{B_x}^*$ and $j\in \mathbb{Z}$. Let $e$ be the ramification degree of $Y_x/F_x$. Then there exists $v_x\in \widehat{A_x}^*$ such that $1 = \N_{Y_x/F_x}(a) = \N_{Y_x/F_x}\left(u_x \pi_{Y_x}^j\right) = v_x \pi_{F_x}^{\frac{j\ell}{e}}$. This implies $\frac{j\ell}{e} = 0$ which shows that $j=0$ and that $a\in \widehat{B_x}^*$.

Now let $x=P\in X_0$ be a closed point of $\mathcal{X}$ and let $B_P$ denote the integral closure of $A_P$ in $Y$. By Proposition \ref{propgoodshape}, $B_P$ is local and isomorphic to $\mathcal{O}_{\mathcal{Y}, Q}$ where $h : \mathcal{Y}\to \mathcal{X}$ denotes the normal closure of $\mathcal{X}$ in $Y$ and $h^{-1}(P)=\{Q\}$.

If $P\not\in \mathcal{H}_{\mathcal{X}}$, then $a = b^{-1}\psi(b)\in {(B_P)_{I}}^*$ for any height one prime ideal $I$ of $B_P$. Since $B_P$ is normal, we have $\cap_{I} (B_P)_{I} = B_P$. Therefore $a\in B_P$ and further since $a$ is not contained in any height one prime ideal, $a\in B_P^*$. 

Let $P\in \mathcal{H}_{\mathcal{X}}$ and $(\pi_P, \delta_P)$ be a system of parameters of $A_P$ such that they cut out the irreducible curves in $\mathcal{H}_{\mathcal{X}}$ on which $P$ lies.  Thus $\divi_{\Spec B_P}(a)$ is supported at most along primes of $B_P$ lying over $(\pi_P)$ and $(\delta_P)$. By Lemma \ref{lemmaatetasplit-1} and Proposition \ref{propgoodshape}, there exists exactly one prime lying over $(\pi_P)$ and one over $(\delta_P)$. Since $B_P$ is normal and $\N_{Y/F}(a)=1$, we see that $a\in B_P^*$.

The canonical $A_P$-morphism $i: B_P\to Y\otimes_F F_P = Y_P$ sending $b'\leadsto b'\otimes 1$ is an injection. Since $B_P$ is integral over $A_P$, we see that $i(B_P)$ is integral over $\widehat{A_P}$ and hence $i(B_P)\subseteq \widehat{B_P}$. Hence $a\in B_P^*$ implies $a\in \widehat{B_P}^*$ also. \end{proof}

\begin{proposition} Let $P\in S_0$ such that $Y_P\simeq \prod_{i=1}^{\ell} F_P$.  Let $\brac{\pi_P, \delta_P}$ be the system of regular parameters at $A_P$ fixed as in Section \ref{subsubsection-modelparameter} and let $a = \brac{a'_i}_i$ where $a'_i\in F_P$. Then there exist $z_{i,P}\in \widehat{A_P}^*$ and $m_i, n_i\in \mathbb{Z}$ such that $a'_i = z_{i,P}\pi_P^{m_i}\delta_P^{n_i}$. Further $\sum m_i=\sum n_i=0$ and $\prod z_{i,P}=1$.
\label{propaisingoodshape}
\end{proposition}
\begin{proof}
By Proposition \ref{propgoodshape}, if $h : \mathcal{Y}\to \mathcal{X}$ denotes the normal closure of $\mathcal{X}$ in $Y$, then $h^{-1}(P)=\{Q_1, \ldots, Q_{\ell}\}$. Further if $B_P$ denotes the integral closure of $A_P$ in $Y$, $B_P$ is semi-local with maximal ideals $\{m_{Q_1}, \ldots, m_{Q_{\ell}}\}$ with $(B_P)_{m_{Q_i}}\simeq \mathcal{O}_{\mathcal{Y}, Q_i}$. 

Let $(\pi_P, \delta_P)$ be a system of parameters of $A_P$ such that they cut out the irreducible curves in $\mathcal{H}_{\mathcal{X}}$ on which $P$ lies. As in the proof of Proposition \ref{propaisaunitifYnotsplit}, $\divi_{\Spec B_P}(a)$ is supported at most along primes lying above $(\pi_P)$ and $(\delta_P)$. 

Since $Y=F(\sqrt[\ell]{y})$ where $\divi(y)$ is arranged to be in good shape in $\mathcal{X}$ and $Y_P$ is split, $\mathcal{O}_{\mathcal{Y}, Q_i}$ is a regular local ring. Further $\widehat{\mathcal{O}_{\mathcal{Y}, Q_i}}\simeq \widehat{A_P}$. Let $(\pi'_{Q_i}, \delta'_{Q_i})$ be a system of regular parameters where $\pi'_{Q_i}$ (resp. $\delta'_{Q_i}$) lies over $\pi_P$ (resp. $\delta_P)$. Using the identification\footnote{Note that if $Y\subseteq F_P$ via a different $Q_i$, then the new identification of $Y\otimes F_P\simeq_{Q_i}\prod F_P$ differs from the old one $Y\otimes F_P \simeq_{Q_1} \prod F_P$ by an automorphism $\prod F_P \simeq \prod F_P$ permuting the components.} $Y\subseteq F_P$ (via $Q_1$ say), we identify $Y\otimes F_P$ with $\prod F_P$. 

Note that $\pi'_{Q_i}\in Y\otimes F_P$ gets identified  with $\brac{\pi'_{Q_i}, \pi'_{Q_{i+1}}, \ldots, \pi'_{Q_{i-1}}}\in \prod F_P$ where each $\pi'_{Q_j}$ is supported at most along $(\pi_P)$ in $\widehat{A_P}$. Similarly $\delta'_{Q_i}$ gets identified with $\brac{\delta'_{Q_i}, \delta'_{Q_{i+1}}, \ldots, \delta'_{Q_{i-1}}}$ with each $\delta'_{Q_j}$ being supported at most along $(\delta_P)$ in $\widehat{A_P}$. Since $a$ has norm $1$, the proposition about the shape of $a$ follows.\end{proof}

\newpage

\begin{proposition} Let $P\in X_0\setminus S_0$ be a closed point of $\mathcal{X}$ such that it lies on exactly one irreducible curve (say $C$)  of $\mathcal{H}_{\mathcal{X}}$. Further assume $Y_P\simeq \prod_{i=1}^{\ell} F_P$.  Let $\brac{\pi_{P}, \delta_P}$ be a system of regular parameters at $A_P$ such that $\pi_P$ defines $C$ at $P$. Let $a = \brac{a'_i}_i$ where $a'_i\in F_P$. Then there exist $z_{i,P}\in \widehat{A_P}^*$ and $m_i\in \mathbb{Z}$ such that $a'_i = z_{i,P}\pi_P^{m_i}$. Further $\sum m_i=0$ and $\prod z_{i,P}=1$.
\label{propaisingoodshape2}
\end{proposition}
\begin{proof}
The proof is similar to that of the previous proposition except that $a$ is now supported at most at primes lying above $\pi_P$.
\end{proof}

\subsection{Strategy}
Recall that we have $z\in M\cap \SL1(D)$ with $\N_{M/Y}(z)=a$ and $\N_{Y/F}(a)=1$ where $M$ is a maximal subfield containing a cyclic subfield $Y$ of degree $\ell$. The goal is to show $z\in \sqbrac{D^*,D^*}$. We would like to split $a$ into a product of \textit{suitable} elements $a_1$ and $a_2$ lying in \textit{nicer} subfields $E_1$ and $E_2$ respectively. More precisely, we would like to find elements $a_1,a_2\in Y$ and field extensions $E_1, E_2$ such that for each $i=1,2$, the following hold:

\begin{enumerate}
\item 
$a_1a_2=a$.
\item
$E_i/F$ is a subfield of $D$ of degree $\ell$.
\item
$E_i\subseteq C_D(Y)$ and $D\otimes E_i\otimes Y$ is split.
\item
There exists $\theta_i\in YE_i \subseteq D$ such that $\N_{YE_i/Y}(\theta_i) = a_i$.
\item 
$\theta_i \in \left[D^*,D^*\right]$.
\end{enumerate}

Note that properties (3), (4) and (5) force that $a_i \in \Nrd_{Y}\left(C_D(Y)\right)$ and that $\N_{Y/F}\left(a_i\right)=1$. The construction of such subfields $E_i/F$  is useful in modifying $z$ by commutators so that it is a product of commutators, as shown by the proposition below.

\begin{proposition}
\label{propositionEconsequence} Let $D, M, Y, z, a$ be as before. If there exist elements $a_1,a_2\in Y$ and subfields $E_1/F$ and $E_2/F$ with properties (1) - (5) above, then $z$ is a product of commutators.
\end{proposition}

\begin{proof}
Let $D':= C_D(Y)$ which is a central divison algebra of index $\ell$ over $Y$. Since $E_i$ commutes with $Y$ in $D$ , $\theta_i \in D'$. Since $z\in D'$, we have that $z\theta_2^{-1}\theta_1^{-1}\in D'$. Note that $YE_i$ and $M$ are maximal subfields of $D'/Y$. Thus \[\Nrd_{D'}\left(z\theta_2^{-1}\theta_1^{-1}\right) = \N_{M/Y}(z) \N_{YE_2/Y}\left(\theta_2^{-1}\right) \N_{YE_1/Y}\left(\theta_1^{-1}\right) = aa_2^{-1}a_1^{-1}  = 1.\] 

Since $D'$ is a central division algebra with square-free index, every reduced norm one element is a product of commutators (\cite{W}). Thus $z\theta_2^{-1}\theta_1^{-1} \subseteq \left[D'^*, D'^*\right] \subseteq \left[D^*, D^*\right]$. Since each $\theta_i\in \sqbrac{D^*, D^*}$ by hypothesis, $z\in \left[D^*, D^*\right]$ also. 
\end{proof}

The rest of the paper is devoted to constructing $E_i$ and $a_i$ satisfying properties (1)-(5) listed above. This is done by applying the techniques of patching developed by Harbater-Hartmann-Krashen.

\section{At codimension one points}
Recall that $N'_0$ denotes the set of all codimension one points of $\mathcal{X}$ which lie in $\mathcal{H}_{\mathcal{X}}$. For each $\eta\in N'_0$, let $\pi_{\eta}$ be the parameter of $F_{\eta}$ fixed as in Section \ref{subsubsection-modelparameter}. 

\subsubsection*{Classification of points of $N'_0$}
\label{subsectioncodimonepointsclassification}
We say that $\eta \in N'_0$ is of 
\begin{itemize}
\item \textbf{Type 0} if the index of $D_{\eta}$ is $1$. Thus $\eta\not\in \ram_{\mathcal{X}}(D)$.
\item \textbf{Type 1} if the index of $D_{\eta}$ is $\ell$. We further classify these points into subtypes. 
\begin{itemize}
\item
\textbf{Type 1a}:  if $\eta\not\in \ram_{\mathcal{X}}(D)$. Thus $D_{\eta}/F_{\eta}$ is an unramified index $\ell$ CSA.
\item
\textbf{Type 1b}: if $\eta \in \ram_{\mathcal{X}}(D)$. Thus $D_{\eta} = D_0 + (u_{\eta}, \pi_{\eta}) = \M_{\ell}\left(u_{\eta}, v_{\eta}\pi_{\eta}\right) $ where $u_{\eta}, v_{\eta}$ are units in $\widehat{A_{\eta}}$ and $D_0/F_{\eta}$ is an unramified CSA.
\end{itemize}
\item \textbf{Type 2} if the index of $D_{\eta}$ is $\ell^2$. Thus $\eta\in \ram_{\mathcal{X}}(D)$ and $D_{\eta} = D_0 + (u_{\eta}, \pi_{\eta})$ where $u_{\eta}\in \widehat{A_{\eta}}^*$ is not an $\ell^{\mathrm{th}}$ power and $D_0/F_{\eta}$ is an unramified CSA such that $D_0\otimes F_{\eta}(\sqrt[\ell]{u_{\eta}})$ has index $\ell$ (Lemma \ref{lemmaindexformula}).
\end{itemize}

\subsubsection*{Shapes of $Y$ and $a$}
\label{subsectionYDateta}
For $\eta\in N'_0$, let $\widehat{B_{\eta}}$ denote the integral closure of $\widehat{A_{\eta}}$ in $Y_{\eta}$ whenever the latter is a field extension of $F_{\eta}$. If $Y_{\eta} \simeq \prod F_{\eta}$, we let $a = \left(a'_{i,\eta}\right)_i$ where $a'_{i,\eta}\in F_{\eta}$. Since $\N_{Y/F}(a)=1$, we have $\prod a'_{i,\eta}=1 \in F_{\eta}$. We now classify $Y_{\eta}$ into four types as follows:

\begin{itemize}
\item \textbf{RAM} : $Y_{\eta}$ is of Type RAM if $Y_{\eta}/F_{\eta}$ is a ramified extension. 
\item \textbf{RES} : Let $\eta$ be of Type 1b or 2 (i.e $\eta \in \ram_{\mathcal{X}}(D)$). Then $Y_{\eta}$ is of Type RES if it is the lift of residues as defined in Section \ref{sectionnotation}. In particular, it is an unramified nonsplit extension of $F_{\eta}$.
\item\textbf{SPLIT} : $Y_{\eta}$ is of Type SPLIT if $Y_{\eta}\simeq \prod_{i=1}^{\ell} F_{\eta}$.
\item\textbf{NONRES} : $Y_{\eta}$ is of Type NONRES if it is none of the above types. That is, it is an unramified nonsplit extension of $F_{\eta}$ and if $\eta\in \ram_{\mathcal{X}}(D)$, it is NOT the lift of residues.
\end{itemize}

\begin{remark} \label{remarktype2Ynotsplit} Thus if $\eta$ is of Type 2, then $Y_{\eta}$ cannot be of Type SPLIT.
\end{remark}

\begin{lemma}[Along $\eta$ of Type 1a]
\label{lemmaalong1a}
Let $\eta\in N'_0$ be of Type 1a. Further assume that $Y_{\eta}\simeq \prod F_{\eta}$. Let $a = \left(a'_{i,\eta}\right)\in \prod F_{\eta}$ where each $a'_{i,\eta}\in F_{\eta}$. Then $a'_{i,\eta} = z_{i,\eta}\pi_{\eta}^{\ell m'_i}\in F_{\eta}$ where  $z_{i,\eta}\in \widehat{A_{\eta}}^*$ and $m'_i\in \mathbb{Z}$.
\end{lemma}

\begin{proof}
Let $a'_{i,\eta} = z_{i,\eta}\pi_{\eta}^{m_i}$ for $z_{i,\eta}\in \widehat{A_{\eta}}^*$ and $m_i\in \mathbb{Z}$. Since $a$ is a reduced norm from $D\otimes Y$, we have $\brac{ z_{i,\eta}\pi_{\eta}^{m_i}}[D_{\eta}] =  0 \in \HH^3\brac{F_{\eta}, \mu_{\ell}}$ for each $i$. Since $D_{\eta}$ is unramified and has index $\ell$, by Lemma \ref{lemmareducednormsunramified} $\brac{ z_{i,\eta}}[D_{\eta}] =  0$. Thus, taking residues along $\pi_{\eta}$ shows that $m_i \cong 0 \mod \ell$.
\end{proof}

For ease of reference, we summarize possible shapes of $Y$ and $a$ at points of $N'_0$ in the following table (cf. Lemmata \ref{lemmanormoneramified-dim1}, \ref{lemmaalong1a}, Proposition \ref{propaisaunitifYnotsplit}) where we use the notations that $w'_{\eta}, z_{\eta}\in \widehat{B_{\eta}}^*$,   $u_{\eta}, z_{i, \eta} \in \widehat{A_{\eta}}^*$ and $u_{\eta}\not\in\widehat{A_{\eta}}^{\ell}$, $m_i, m'_i \in \mathbb{Z}$ and $D_0/F_{\eta}$ is an unramified CSA.   Further $\N_{Y_{\eta}/F_{\eta}}\left(z_{\eta}\right) =1$,  $\N_{Y_{\eta}/F_{\eta}}\left(w'_{\eta}\right)^{\ell} =1$, $\sum_{i=1}^{\ell} m_i = \sum_{i=1}^{\ell} m'_i  = 0$ and $\prod_{i=1}^{\ell} z_{i, \eta} = 1$.

\begin{center}
$
\begin{array}{|c|c|c|c|c|}
\hline 
\mrm{{Type\ of \ \eta}} & {D_{\eta}} & \mrm{{More\ information}} & {Y_{\eta}} & a \in \ Y_{\eta} \\
\hline
0 & D_0 & \ind(D_0)=1 &{\mrm{RAM}} & {w'}_{\eta}^{\ell} \\
\hline
0 &  D_0  & \ind(D_0)=1 & {\mrm{SPLIT}} & {\left(a'_{i,\eta} = z_{i, \eta}\pi_{\eta}^{m_i}\right)_{i}}\\
\hline
0 &  D_0  & \ind(D_0)=1 &{\mrm{NONRES}} & z_{\eta} \\
\hline
\hline
1a & {D_0}&  \ind(D_0)=\ell \ & {\mrm{RAM}} &  {w'}_{\eta}^{\ell} \\
\hline
1a &  {D_0} & \ind(D_0)=\ell  \ &{\mrm{SPLIT}} & {\left(a'_{i,\eta} = z_{i, \eta}\pi_{\eta}^{\ell m'_i}\right)_{i}}\\
\hline
1a &  {D_0} &  \ind(D_0)=\ell   \ &{\mrm{NONRES}} & z_{\eta} \\
\hline
\hline
1b & D_0 + \left(u_{\eta}, \pi_{\eta}\right) & \ind(D_0\otimes F_{\eta}(\sqrt[\ell]u_{\eta})))=1 & {\mrm{RAM}} &   {w'}_{\eta}^{\ell}\\
\hline
1b &  D_0 + \left(u_{\eta}, \pi_{\eta}\right) &\ind(D_0\otimes F_{\eta}(\sqrt[\ell]u_{\eta})))=1 & {\mrm{RES}} & z_{\eta} \\
\hline
1b &  D_0 + \left(u_{\eta}, \pi_{\eta}\right) & \ind(D_0\otimes F_{\eta}(\sqrt[\ell]u_{\eta})))=1 &{\mrm{SPLIT}} & {\left(a'_{i,\eta} = z_{i, \eta}\pi_{\eta}^{m_i}\right)_{i}}\\
\hline
1b &   D_0 + \left(u_{\eta}, \pi_{\eta}\right) &\ind(D_0\otimes F_{\eta}(\sqrt[\ell]u_{\eta})))=1 & {\mrm{NONRES}} & z_{\eta} \\
\hline
\hline
2 & D_0 + \left(u_{\eta}, \pi_{\eta}\right) & \ind(D_0\otimes F_{\eta}(\sqrt[\ell]u_{\eta})))=\ell &{\mrm{RAM}} &  {w'}_{\eta}^{\ell}\\
\hline
2 &  D_0 + \left(u_{\eta}, \pi_{\eta}\right) &  \ind(D_0\otimes F_{\eta}(\sqrt[\ell]u_{\eta})))=\ell  & {\mrm{RES}} & z_{\eta} \\
\hline
2 &   D_0 + \left(u_{\eta}, \pi_{\eta}\right) & \ind(D_0\otimes F_{\eta}(\sqrt[\ell]u_{\eta})))=\ell  & {\mrm{NONRES}} & z_{\eta} \\
\hline
\end{array}$
\label{TableaandYatcodimonepoints}
\captionof{table}{Shape of $D$, $Y$ and $a$ at $\eta\in N'_0$}
\end{center}

\subsubsection*{Fixing residual Brauer classes for points in $N'_0$ along which $D$ is ramified}
\label{section-fixingresidualbrauerclass}
For each $\eta\in N'_0$ of Type 1b or 2, we define $\beta_{rbc, \eta}\in \Br(F_{\eta})$ as follows: 

If $Y_{\eta}$ is RAM, (so $Y_{\eta} = F_{\eta}\brac{\sqrt[\ell]{w_{\eta}\pi_{\eta}}}$ for some $w_{\eta}\in \widehat{A_{\eta}}^*$), then there exists an unramified algebra ${D_{0}}_{\eta}$ such that $D_{\eta} = {D_{0}}_{\eta} + \brac{u_{\eta}, w_{\eta}\pi_{\eta}} \in \Br\brac{F_{\eta}}$. Set $\beta_{rbc, \eta}=[{D_{0}}_{\eta}] \in \Br\brac{F_{\eta}}$. In all other cases, set $\beta_{rbc, \eta}$ to be the residual Brauer class of $D$ with respect to parameter $\pi_{\eta}$ (cf. Section \ref{sectionnotation}). Note that $\beta_{rbc, \eta}$ has index at most $\ell$. 

\section{At closed points}
Recall that $S_0$ denotes the finite set of closed points lying on $\mathcal{H}_{\mathcal{X}}$ chosen as in Section \ref{subsubsection-modelparameter}. We refer to points $P$ in $S_0$ as \textit{marked points} occasionally. In this section, we classify points in $S_0$ following (\cite{S07}) in essence, study the configuration of $Y$ at these points and also investigate the shape of $a$ at some types of closed points $P$ when $Y_P \simeq \prod F_P$.

Let $P\in S_0$ be the intersection of two distinct irreducible curves $C$ and $C'$ of  $\mathcal{H}_{\mathcal{X}}$ with generic points $\eta$ and $\eta'$ in $N'_0$ respectively. Let $\pi_P$ and $\delta_P$ be primes defining $C$ and $C'$ at $P$ be as fixed in Section \ref{subsubsection-modelparameter}. 

\subsection{Classification of marked points}
\label{subsectionclosedpointsclassification}
We use the following notations: $u_P, v_P$ will denote units in $A_P$, $D_{00}$, the Brauer class of an algebra of $\Br\left(F\right)$ unramified at $A_P$, i.e. $D_{00}\in \Br\left(A_P\right)$. Superscripts $s$ and $ns$ on $D_P$ are used to denote that the algebra $D_P$ is split and non-split respectively. We sometimes refer to the irreducible curve with generic point $\eta\in N'_0$ as $\overline{\eta}$. We begin with a lemma (similar to Lemma \ref{lemmaatetasplit-1}) relating the shapes of $D_{\eta}$ and $D_P$.

\begin{lemma}
\label{lemmaatetasplit}
If $D_{\eta} = 0 \in \Br(F_{\eta})$, then $D_P = 0\in \Br(F_P)$. \end{lemma}
\begin{proof} Since $D$ is unramified at $\eta$, $D_P = \left(v_P, \delta_P\right)$. Further as $D_{\eta}$ is split, so is $D\otimes F_{P,\eta}$. This implies $\left(v_P, \delta_P\right)=0 \in \Br\left(F_{P,\eta}\right)$. That is $\left(\overline{v_P}, \overline{\delta_P}\right) = 0 \in \Br\left(k_{P,\eta}\right)$. Recall that $k_{P,\eta}$ is a complete discretely valued field with $\overline{\delta_P}$ as a parameter. Since $v_P\in \widehat{A_P}^*$, $\overline{v_P}\in k_P^{*\ell}$ and hence $v_P\in \widehat{A_P}^{*\ell}$. This immediately implies that $D_P = 0$ in $\Br\left(F_P\right)$. 
\end{proof}

\begin{remark}\label{remark-ram-split-not-intersect} Lemma \ref{lemmaatetasplit-1} implies that if $\eta, \eta' \in N'_0$ are such that $Y_{\eta}$ is of Type RAM and $Y_{\eta'}$ is of Type SPLIT, then $\overline{\eta}$ and $\overline{\eta'}$ cannot intersect. 
\end{remark}

We now list\footnote{The order of the subscripts in the types of points do not matter. So for instance we will use both $C_{12}^{Cold}$ and $C_{21}^{Cold}$ to mean the same type of point.} the types\footnote{It will be shown in Proposition \ref{propnospecialpoints} following the classification that the starred ones can be eliminated by blowing up our model repeatedly.} of closed points in $S_0$ possible. 

\textbf{Type A}: $P$ is of Type A if both $C$ and $C'$ do not lie in the ramification locus of $D$. Further $D$ is unramified at $P$ and because the residue field is finite, $D_P$ is split. Type $A$ points are further subdivided as follows:
\begin{itemize}
\item[-]
\textbf{Type $A_{00}^{s}$}: Both $\eta$ and $\eta'$ are of Type 0. Thus $D_{\eta}$ and $D_{\eta'}$ are split.
\item[-]
\textbf{Type $A_{10}^{s}$}: Exactly one of $\eta, \eta'$ is of Type 0. Thus the other, say $\eta$, is of Type 1a.  So $D_{\eta'}$ is split whereas $D_{\eta}$ is an unramified index $\ell$ CSA.
\item[-]
*Type $A_{11}^{s}$: Both  $\eta$ and  $\eta'$ are of Type 1a.
\end{itemize} 

\textbf{Type B}: $P$ is of Type B if exactly one of $C$ and $C'$ lies in the ramification locus of $D$ (say $C$). Thus $\eta$ is of Type 1b or 2 and $\eta'$, of Type 0 or 1a. Further $D=D_{00}+ (u_P, \pi_P)$ in $\Br(F)$ and because the residue field is finite,  $D_P =\left(u_P, \pi_P\right)$ in $\Br\left(F_P\right)$. Type $B$ points are further subdivided as follows:

\begin{itemize}
\item[-]
\textbf{Type $B_{10}^{s}$}: $\eta$ is of Type 1b and $\eta'$ is of Type 0. Note that by Lemma \ref{lemmaatetasplit}, $D_P$ is split.
\item[-]
*Type $B_{11}^{s}$: $\eta$ is of Type 1b, $\eta'$ is of Type 1a and $D_P$ is split.
\item[-]
\textbf{Type $B_{11}^{ns}$}: $\eta$ is of Type 1b, $\eta'$ is of Type 1a and $D_P$ is non-split.
\item[-] 
\textbf{Type $B_{20}^{s}$}: $\eta$ is of Type 2 and $\eta'$ is of Type 0. Note that by Lemma \ref{lemmaatetasplit}, $D_P$ is split.
\item[-]
*Type $B_{21}^{s}$: $\eta$ is of Type 2, $\eta'$ is of Type 1a and $D_P$ is split.
\item[-]
\textbf{Type $B_{21}^{ns}$}: $\eta$ is of Type 2, $\eta'$ is of Type 1a and $D_P$ is non-split.
\end{itemize}

\textbf{Type C}: $P$ is of Type C if both $C$ and $C'$ lie in the ramification locus of $D$. Thus $\eta$ and $\eta'$ are of Type 1b or 2. Further $D=D_{00}+ (u_P, \pi_P) + (v_P, \delta_P)$ or $D_{00} + \left(u_P\pi_P^m, v_P\delta_P \right)$ in $\Br\left(F\right)$ for an integer $m$ coprime to $\ell$ in $\Br(F)$. 

Points $P$ where $D=D_{00} + \left(u_P\pi_P^m, v_P\delta_P \right)$ were labelled \textit{cold} points in (\cite{S07}). Thus $D_P =\left(u_P\pi_P^m, v_P\delta_P \right)$ at a cold point $P$ and the ramification data at $C$, $\partial_C([D])$ is given by $\left(\brac{v_P\delta_P}^{-m}\right)^{\frac{1}{\ell}}$ at $P$. Points $P$ where $D=D_{00}+ (u_P, \pi_P) + (v_P, \delta_P)$  were further subdivided depending on the shape of the finite subgroups $x= \langle \overline{u_P}\rangle$ and $y =\langle \overline{v_P}\rangle$ in $k_P^{*}/k_P^{*\ell}$ into \textit{chilly} points (when $x = y\neq \{1\}$), \textit{cool} points (when $x=y=\{1\}$), and \textit{hot} points (when $x\neq y$). Since $k_P^*/k_P^{*\ell}$ is a cyclic group of order $\ell$, the subgroups $x$, $y$ have to be either trivial or all of $k_P^{*}/k_P^{*\ell}$. 

If $P$ is a chilly point, without loss of generality assume $ \overline{u_P}=\overline{v_P}^j$ for some $j$ coprime to $\ell$. Thus $D_P = \left(v_P, \pi_P^j\delta_P\right)\in \Br(F_P)$ and the ramification data at $C$, $\partial_C([D])$ is given by $\left(v_P^{j}\right)^{\frac{1}{\ell}}$ at $P$. If $P$ is a cool point, $D_P=\{0\}\in \Br(F_P)$. If $P$ is a hot point, assume without loss of generality that $y=\{1\}$. Thus $D_P = \left(u_P, \pi_P\right)\in Br(F_P)$ and  the ramification data at $C$, $\partial_C([D])$ is given by $\left(u_P\right)^{\frac{1}{\ell}}$ at $P$. We also recall that in this case, $D\otimes F_{\eta'}$ has index $\ell^2$ (\cite{S07}, Proposition 0.5, Theorem 2.5) and hence $\eta'$ is of Type $2$. We continue to follow Saltman's convention while refining the classification as follows:

\begin{itemize}
 \item[-]
\textbf{Type $C_{11}^{Cold}$} :  $\eta$ is of Type 1b, $\eta'$ is of Type 1b and $P$ is cold.
\item[-]
\textbf{Type $C_{11}^{Chilly}$} :  $\eta$ is of Type 1b, $\eta'$ is of Type 1b and $P$ is chilly.
\item[-]
*Type $C_{11}^{Cool}$ :  $\eta$ is of Type 1b, $\eta'$ is of Type 1b and $P$ is cool.
\item[-]
\textbf{Type $C_{12}^{Cold}$} :  $\eta$ is of Type 1b, $\eta'$ is of Type 2 and $P$ is cold.
\item[-]
*Type $C_{12}^{Chilly}$ :  $\eta$ is of Type 1b, $\eta'$ is of Type 2 and $P$ is chilly.
  \item[-]
*Type $C_{12}^{Cool}$ :  $\eta$ is of Type 1b, $\eta'$ is of Type 2 and $P$ is cool.
\item[-]
\textbf{Type $C_{12}^{Hot}$} :  $\eta$ is of Type 1b, $\eta'$ is of Type 2 and $P$ is hot.
\item[-]
*Type $C_{22}^{-}$ :  $\eta$ is of Type 2 and $\eta'$ is of Type 2.
\end{itemize}

\subsection{Shape of $a$ when $Y_P$ is split}
\label{subsectionaatP}
We investigate the shape of $a$ at some types of closed points $P\in S_0$ when $Y_P \simeq \prod F_P$. By Proposition \ref{propaisingoodshape}, $a = \left(a'_{i,P}\right)_i\in \prod F_P$ where $a'_{i,P} = z_{i,P}\pi_P^{m_i}\delta_P^{n_i}$, $z_{i,P}\in \widehat{A_P}^*$ and $m_i, n_i\in \mathbb{Z}$ with $\sum m_i = \sum n_i = 0$.

\begin{proposition}
\label{propositionaatchillypoint}
\label{propositionaatcoldpoint}
\label{propositionaathotpoint}
\label{propositionaatB11nspoint} 
\label{propositionaatB21nspoint} 
 Let $P\in S_0$ such that $Y_P\simeq \prod F_P$ and let $a = \left(a'_{i,P}\right)_i\in \prod F_P$ where $a'_{i,P} = z_{i,P}\pi_P^{m_i}\delta_P^{n_i}$ as above. 
 
\begin{enumerate}
\item
If $P$ is a cold point with $D_P = \left(u_P\pi_P^m, v_P\delta_P\right)$ where $0 < m < \ell$, then $a'_{i,P} = \left(u_P\pi_P^m\right)^{sm_i}\left(v_P\delta_P\right)^{n_i}{\brac{{w'}_{i,P}}^{\ell}}{\brac{\pi_P^{-rm_i}}^{\ell}}$ for some ${w'}_{i,P}\in \widehat{A_P}^*$ and $s, r\in \mathbb{Z}$ such that $sm = r\ell +1$.
\item
If $P$ is a chilly point with  $D_P = \left(v_P, \pi_P^j\delta_P\right)$ where $0<j<\ell$, then $m_i = r_i\ell + jn_i$ where $r_i\in \mathbb{Z}$. Thus $\sum r_i = 0$ and $a'_{i,P} = z_{i,P}\left(\pi_P^j\delta_P\right)^{n_i}{\brac{\pi_P^{ r_i}}^{\ell}}$.
\item
If $P$ is a hot point\footnote{Here $\eta$ is of Type 1b whereas $\eta'$ is of Type 2 (\cite{S07}, Proposition 0.5, Theorem 2.5).} or a Type $B_{11}^{ns}$  point\footnote{Here $\eta$ is of Type 1b whereas $\eta'$ is of Type 1a}  with $D_P = \left(u_P, \pi_P\right)$, then $n_i = n'_i\ell$ where $n'_i\in \mathbb{Z}$. Thus $\sum n'_i = 0$ and $a'_{i,P} = z_{i,P}\pi_P^{m_i}{\brac{\delta_P^{n'_i}}^{\ell}}$.
\item
If $P$ be a Type $B_{21}^{ns}$ point\footnote{Here $\eta$ is of Type 2 whereas $\eta'$ is of Type 1a} with $D_P = \left(u_P, \pi_P\right)$, then $m_i=0$ and $n_i = n'_i\ell$ where $n'_i\in \mathbb{Z}$. Thus $\sum n'_i= 0$ and $a'_{i,P} = z_{i,P}{\brac{\delta_P^{n'_i}}^{\ell}}$.
\end{enumerate}
\end{proposition}

\begin{proof}
Since $a$ is a reduced norm from $D\otimes Y$, for each $i$, $\brac{a'_{i,P}}[D] = 0 \in \HH^3\brac{F_P, \mu_{\ell}}$.

\textit{At a cold point}: 
\begin{align*}
& \phantom{implies} \brac{ z_{i,P}\pi_P^{m_i}\delta_P^{n_i}} \brac{u_P\pi_P^m, v_P\delta_P} =  0\\
& \implies  \brac{ z_{i,P}} \brac{\pi_P^m, \delta_P} +  \brac{ \pi_P^{m_i}} \brac{u_P, \delta_P} +  \brac{ \delta_P^{n_i}} \brac{\pi_P^m, v_P} =0\\
& \implies  \brac{ z_{i,P}^m} \brac{\pi_P, \delta_P} +  \brac{ u_P^{-m_i}} \brac{\pi_P, \delta_P} +  \brac{ v_P^{-mn_i}} \brac{\pi_P,\delta_P} =0\\
& \implies  \brac{ z_{i,P}^m u_P^{-m_i}v_P^{-mn_i}} \brac{\pi_P, \delta_P} =0\\
\end{align*}
Taking residues along $\pi_P$ and then along $\overline{\delta_P}$, we see that $z_{i,P}^m=  u_P^{m_i}v_P^{mn_i}{w''}_{i,P}^{\ell}$ for some ${w''}_{i,P}\in \widehat{A_P}^*$. Since $0< m< \ell$, let $0<s<\ell$ such that $sm = r\ell + 1$ for some $r\in \mathbb{Z}$. Taking $s^{\mrm{th}}$ powers, we have $z_{i,P}^{r\ell+1}= u_P^{sm_i}v_P^{n_ir\ell + n_i}{w''}_{i,P}^{s\ell}$. Hence for some $w'_{i,P}\in \widehat{A_P}^*$,\[a'_{i,P} = z_{i,P}\pi_P^{m_i}\delta_P^{n_i} = \brac{u_P^{s}\pi_P}^{m_i}\brac{v_P\delta_P}^{n_i}{\brac{v_P^{n_ir}{w''}_{i,P}^{s}z_{i,P}^{-r}}^{\ell}} = \brac{u_P\pi_P^m}^{sm_i}{\brac{\pi_P^{-rm_i\ell}}}\brac{v_P\delta_P}^{n_i}{{w'}_{i,P}^{\ell}} \]

\textit{At a chilly point}: 
\begin{align*}
& \phantom{implies} \brac{ z_{i,P}\pi_P^{m_i}\delta_P^{n_i}} \brac{v_P, \pi_P^j\delta_P} =  0\\
& \implies  \brac{\pi_P}\brac{v_P^{m_i}, \delta_P} +  \brac{\delta_P}\brac{v_P^{jn_i}, \pi_P} = 0\\
& \implies \brac{v_P^{-m_i}}\brac{\pi_P, \delta_P} +  \brac{v_P^{jn_i}}\brac{\pi_P, \delta_P}  = 0\\
& \implies \brac{v_P^{jn_i-m_i}}\brac{\pi_P, \delta_P}=0.
\end{align*}
Taking residues along $\pi_P$ and then along $\overline{\delta_P}$, we see that $m_i \cong jn_i \mod \ell$. Since $\sum m_i=\sum n_i=0$ and $0< j < \ell$, $\sum r_i =0$.

\textit{At a hot\ /\ $B_{11}^{ns}$ point} : $\brac{ z_{i,P}\pi_P^{m_i}\delta_P^{n_i}} \brac{u_P, \pi_P} =  0$. Hence $\brac{\delta_P^{n_i}} \brac{u_P, \pi_P} = 0$ and therefore $\brac{u_P^{n_i}}\brac{\pi_P, \delta_P} = 0$. Taking residues along $\pi_P$ and then along $\overline{\delta_P}$, we see that $n_i = n'_i\ell$ for some $n'_i\in \mathbb{Z}$. Since $\sum n_i = 0$, $\sum n'_i=0$ also. 


\textit{At a $B_{21}^{ns}$ point} : 
Since $Y_P$ is split, $Y_{\eta}$ is not of Type RAM. By Remark \ref{remarktype2Ynotsplit} and Proposition \ref{propaisaunitifYnotsplit}, $a$ is a unit along $\eta$. Since $a$ is arranged to be in good shape, we have $m_i =0$. The same proof as in the previous case shows $n_i = n'_i\ell$ and $\sum n'_i=0$. \end{proof}

\subsection{Configuration of $Y$ at marked points in $S_0$}
\label{subsectionYatP}
We record the configuration of $Y$ at some types of the marked points in $S_0$. This is possible since the $\divi_{\mathcal{X}}(y)$ is arranged to be in good shape where $Y = F(\sqrt[\ell]{y})$. We spell out the proof in the case when $P$ is a $C_{11}^{Cold}$ point. The other proofs follow in a similar fashion by using Lemma \ref{lemmaatetasplit-1} and the fact that the shape of $Y_P$ can be determined from that of $Y_{\eta}$ and $Y_{\eta'}$ by going to the branch fields $Y\otimes F_{P,eta}$ and $Y\otimes F_{P,\eta'}$ (c.f proof of Lemma \ref{lemmaatetasplit-1}) along with Remarks \ref{remarktype2Ynotsplit} and \ref{remark-ram-split-not-intersect}.

In this subsection, we use the following notations in the tables: $0 < r < \ell$ and $w, u_P, v_P \in \widehat{A_P}^*$. $L_P$ refers to the unique cyclic degree $\ell$ field extension of $F_P$ unramified at $\widehat{A_P}$.

%

\begin{proposition}[At $C_{11}^{Cold}$ points]
\label{propcoldexamine1}
Let $P$ be a $C_{11}^{Cold}$ point and let $D_P = \left(u_P\pi_P^m, v_P\delta_P\right)$ for $0 < m < \ell$. Then the following table gives the possible configurations (including some symmetric situations) of $Y$ at $P$.

\centering
$
\begin{array}{|c|c|c|}
\hline 
 Y_{\eta'} & Y_{\eta} & Y_P \\
\hline
 {\mrm{RAM}}& {\mrm{RAM}} &  F_P\brac{\sqrt[\ell]{w\pi_P\delta_P^r}} \\
\hline
 {\mrm{RAM}}& {\mrm{RES}} &  {F_P\brac{\sqrt[\ell]{v_P\delta_P}}} \\
\hline
{\mrm{RAM}}& {\mrm{NONRES}} &  F_P\brac{\sqrt[\ell]{w\delta_P}} \\
\hline
{\mrm{RES}}& {\mrm{RAM}} &  {F_P\brac{\sqrt[\ell]{u_P\pi_P^m}}} \\
\hline
 {\mrm{SPLIT}}& {\mrm{SPLIT}} &  \prod F_P \\
\hline
{\mrm{SPLIT}}& {\mrm{NONRES}} &  \prod F_P \\
\hline
 {\mrm{NONRES}}& {\mrm{RAM}} &  F_P\brac{\sqrt[\ell]{w\pi_P}} \\
\hline
 {\mrm{NONRES}}& {\mrm{SPLIT}} &  \prod F_P \\
\hline
 {\mrm{NONRES}}& {\mrm{NONRES}} &  L_P\  \mathrm{or} \  \prod F_P \\
\hline
\end{array}$
\label{TableYatcoldpt1}
\captionof{table}{Shape of $Y$ at $C_{11}^{Cold}$ point $P$}

\end{proposition}

\begin{proof}
If $Y_{\eta'}$ is RAM, by Remark \ref{remark-ram-split-not-intersect}, $Y_{\eta}$ cannot be SPLIT. If $Y_{\eta'}$ is RES, then $Y_{P,\eta'} \simeq F_{P,\eta'}\brac{\sqrt[\ell]{u_P\pi_P^m}}$, a field extension and $\overline{Y_{P,\eta'}}/k_{P,\eta'}$ is ramified. Hence $Y_{\eta}$ is RAM. If $Y_{\eta'}$ is SPLIT, then $Y_P\simeq \prod F_P$ by Lemma \ref{lemmaatetasplit-1}. Hence $Y_{\eta}$ cannot be RAM. It also cannot be RES by the same argument as above. Finally, if $Y_{\eta'}$ is NONRES, the same argument shows  $Y_{\eta}$ cannot be RES.  Since $Y/F$ is arranged to be in good shape in $\mathcal{X}$, the shape of $Y_P$ can be determined from that of $Y_{\eta}$ and $Y_{\eta'}$ in a similar manner as that in the proof of Lemma \ref{lemmaatetasplit-1}. \end{proof}

%
%
%
%

\begin{proposition}[At $C_{12}^{Cold}$ points]
\label{propcoldexamine2}
Let $P$ be a $C_{12}^{Cold}$ point and let $D_P = \left(u_P\pi_P^m, v_P\delta_P\right)$ for $0 < m < \ell$. Assume without loss of generality that $\eta'$ is of Type 2. Then the following table gives the possible configurations of $Y$ at $P$. 

\centering
$
\begin{array}{|c|c|c|}
\hline 
Y_{\eta'} & Y_{\eta} & Y_P \\
\hline
{\mrm{RAM}}& {\mrm{RAM}} &  F_P\brac{\sqrt[\ell]{w\pi_P\delta_P^r}} \\
\hline
{\mrm{RAM}}& {\mrm{RES}} &  {F_P\brac{\sqrt[\ell]{v_P\delta_P}}} \\
\hline
{\mrm{RAM}}& {\mrm{NONRES}} &  F_P\brac{\sqrt[\ell]{w\delta_P}} \\
\hline
{\mrm{RES}}& {\mrm{RAM}} &  {F_P\brac{\sqrt[\ell]{u_P\pi_P^m}}} \\
\hline
{\mrm{NONRES}}& {\mrm{RAM}} &  F_P\brac{\sqrt[\ell]{w\pi_P}} \\
\hline
{\mrm{NONRES}}& {\mrm{SPLIT}} &  \prod F_P \\
\hline
{\mrm{NONRES}}& {\mrm{NONRES}} &  L_P\  \mathrm{or} \  \prod F_P \\
\hline
\end{array}$
\label{TableYatcoldpt2}
\captionof{table}{Shape of $Y$ at $C_{12}^{Cold}$ point $P$}

\end{proposition}


\begin{proposition}[At $C_{11}^{Chilly}$ points]
\label{propchillyexamine}
Let $P$ be a $C_{11}^{Chilly}$ point and let $D_P = \left(v_P, \pi_P^j\delta_P\right)$ where $0<j<\ell$. Then the following table\footnote{It includes some symmetric situations.} gives the possible configurations of $Y$ at $P$.


\centering
$
\begin{array}{|c|c|c|}
\hline 
Y_{\eta'} & Y_{\eta} & Y_P \\
\hline
{\mrm{RAM}}& {\mrm{RAM}} &  F_P\brac{\sqrt[\ell]{w\pi_P\delta_P^r}} \\
\hline
{\mrm{RAM}}& {\mrm{NONRES}} &  F_P\brac{\sqrt[\ell]{w\delta_P}} \\
\hline
{\mrm{RES}}& {\mrm{RES}} &  {L_P}\\
\hline
{\mrm{RES}}& {\mrm{NONRES}} &  {L_P} \\
\hline
{\mrm{SPLIT}}& {\mrm{SPLIT}} &  \prod F_P \\
\hline
{\mrm{SPLIT}}& {\mrm{NONRES}} &  \prod F_P \\
\hline
{\mrm{NONRES}}& {\mrm{RAM}} &  F_P\brac{\sqrt[\ell]{w\pi_P}} \\
\hline
{\mrm{NONRES}}& {\mrm{RES}} &  {L_P} \\
\hline
{\mrm{NONRES}}& {\mrm{SPLIT}} &  \prod F_P \\
\hline
{\mrm{NONRES}}& {\mrm{NONRES}} &  L_P\  \mathrm{or} \  \prod F_P \\
\hline
\end{array}$
\label{TableYatchillypt}
\captionof{table}{Shape of $Y$ at chilly point $P$}

\end{proposition}

\begin{proposition}[At $C_{12}^{Hot}$ points]
\label{propositionYathotpoints}
Let $P$ be a $C_{12}^{Hot}$ point and let\footnote{Thus $\eta$ is of Type 1b whereas $\eta'$ is of Type 2 (\cite{S07}, Proposition 0.5, Theorem 2.5).} $D_P = \left(u_P, \pi_P\right)$. If $Y_{\eta'}/F_{\eta'}$ is an unramified extension which is not RES, then it must be of Type NONRES. Further, $Y_P$ is a non-split extension and hence $Y\otimes_ F D\otimes _ F F_P$ is split. 
\end{proposition}
\begin{proof}
By (\cite{S97}, \cite{S98}), $[D] = \left[D_{00} \right] + \left(u_P, \pi_P\right) + \left(v_P, \delta_P\right) \in \Br\left(F\right)$ where $D_{00}$ is unramified at $A_P$. By (\cite{S07}, Proposition 0.5, Theorem 2.5),  $D\otimes_F F_{\eta'}$ is a division algebra and hence $Y_{\eta'}$ has to be a non-split field extension. Thus it is of Type NONRES. 

Now $D = D_{0} + \left(v_P, \delta_P\right) \in \Br\left(F_{\eta'}\right)$ where $D_{0}=\left[D_{00}\right] + \left(u_P, \pi_P\right)$ is unramified at $\eta'$. Since $P$ is a hot point, $D$ is ramified along both $\eta$ and $\eta'$. Thus $v_P$ is not an $\ell^{\mathrm{th}}$ power in $F_{\eta'}$. Since $Y_{\eta'}$ is unramified and not RES, $\left[Y_{\eta'}\left(\sqrt[\ell]{v_P}\right) : Y_{\eta'}\right] = \ell$. 

Thus by Lemma \ref{lemmaindexformula}, $\ell  = \ind\left(D\otimes Y_{\eta'}\right) = \ind\left(D_0\otimes Y_{\eta'}\left(\sqrt[\ell]{v_P}\right)\right)\left[Y_{\eta'}\left(\sqrt[\ell]{v_P}\right) : Y_{\eta'}\right]$ = $\ell\brac{\ind\left(D_0\otimes Y_{\eta'}\left(\sqrt[\ell]{v_P}\right)\right)}$.
Thus $Y_{\eta'}\left(\sqrt[\ell]{v_P}\right)$ splits $D_0$ over $F_{\eta'}$ and hence also over the branch field $F_{P, \eta'}$. Note that $[D_0] = \left(u_P, \pi_P\right) \neq 0 \in \Br\left(F_{P,\eta'}\right)$.

Suppose that $Y_P$ is split. Since $P$ is a hot point, $\overline{v_P}$ is an $\ell^{\mathrm{th}}$ power in $k_P$ and hence $Y_P\left(\sqrt[\ell]{v_P}\right) \simeq \prod F_P$. Thus along the branch field, $Y_{P, \eta'}\left(\sqrt[\ell]{v_P}\right) \simeq \prod F_{P,\eta'}$ which cannot split the non-trivial algebra $D_0$. Thus we conclude that $Y_P$ is a non-split field extension.

Since we have assumed that $Y$ is in good shape and that $Y_{\eta'}$ is unramified at $\eta'$, there exists $j\in \{0,1\}$ such that \[Y_P = F_P\left(\sqrt[\ell]{w_P\pi_P^{j}}\right), w_P\in \widehat{A_P}^*.\]

If $j=0$, $Y_P$ is the unique non-split unramified extension at $P$ and has to be isomorphic to $F_P\left(\sqrt[\ell]{u_P}\right)$, which splits $D_P$. If $j=1$, then let $\lambda^{\ell} = w_P\pi_P$ for $\lambda\in Y_P$. Thus $D_P\otimes  Y_P = \left(u_P, \pi_P\right) = \left(u_P, w_P^{-1}\right) \in \Br\left(Y_P\right)$ and hence split. \end{proof}

\begin{proposition}[At $B_{10}^{s}$ points]
\label{propB10sexamine}
Let $P$ be a $B_{10}^{s}$ point. Assume without loss of generality that $\eta$ is of Type 1b and $\eta'$ is of Type 0. Then the following table gives the possible configurations of $Y$ at $P$. 


\centering
$
\begin{array}{|c|c|c|}
\hline 
Y_{\eta'} & Y_{\eta} & Y_P \\
\hline
{\mrm{RAM}}& {\mrm{RAM}} &  F_P\brac{\sqrt[\ell]{w\pi_P^r\delta_P}} \\
\hline
{\mrm{RAM}}& {\mrm{NONRES}} &  F_P\brac{\sqrt[\ell]{w\delta_P}} \\
\hline
{\mrm{SPLIT}}& {\mrm{RES}} &  \prod F_P \\
\hline
{\mrm{SPLIT}}& {\mrm{SPLIT}} &  \prod F_P \\
\hline
{\mrm{SPLIT}}& {\mrm{NONRES}} &  \prod F_P \\
\hline
{\mrm{NONRES}}& {\mrm{RAM}} &  F_P\brac{\sqrt[\ell]{w\pi_P}} \\
\hline
{\mrm{NONRES}}& {\mrm{RES}} &    {\prod F_P} \\
\hline
{\mrm{NONRES}}& {\mrm{SPLIT}} &  \prod F_P \\
\hline
{\mrm{NONRES}}& {\mrm{NONRES}} &  L_P\  \mathrm{or} \  \prod F_P \\
\hline
\end{array}$
\label{TableYatB10spoint}
\captionof{table}{Shape of $Y$ at $B_{10}^{s}$ point $P$}

\end{proposition}

\begin{proposition}[At $B_{11}^{ns}$ points]
\label{propB11nsexamine}
Let $P$ be a $B_{11}^{s}$ point. Assume without loss of generality that $\eta$ is of Type 1b and $\eta'$ is of Type 1a. Then the following table gives the possible configurations of $Y$ at $P$. 


\centering
$
\begin{array}{|c|c|c|}
\hline 
Y_{\eta'} & Y_{\eta} & Y_P \\
\hline
{\mrm{RAM}}& {\mrm{RAM}} &  F_P\brac{\sqrt[\ell]{w\pi_P^r\delta_P}} \\
\hline
{\mrm{RAM}}& {\mrm{NONRES}} &  F_P\brac{\sqrt[\ell]{w\delta_P}} \\
\hline
{\mrm{SPLIT}}& {\mrm{SPLIT}} &  \prod F_P \\
\hline
{\mrm{SPLIT}}& {\mrm{NONRES}} &  \prod F_P \\
\hline
{\mrm{NONRES}}& {\mrm{RAM}} &  F_P\brac{\sqrt[\ell]{w\pi_P}} \\
\hline
{\mrm{NONRES}}& {\mrm{RES}} &    {L_P}  \\
\hline
{\mrm{NONRES}}& {\mrm{SPLIT}} &  \prod F_P \\
\hline
{\mrm{NONRES}}& {\mrm{NONRES}} &  L_P\  \mathrm{or} \  \prod F_P \\
\hline
\end{array}$
\label{TableYatB11nspoint}
\captionof{table}{Shape of $Y$ at $B_{11}^{ns}$ point $P$}
\end{proposition}

\section{Blowups}
We repeatedly exploit the trick of blowing up\footnote{Note that with each blow up, the set $S_0$ for the new model is enlarged to include the intersection points of the exceptional curve and the closure of the strict transforms of $\mathcal{H}_{\mathcal{X}}$ and the set $N'_0$ is expanded to include the generic point of the exceptional curve. } our model at closed points to make the model more amenable for patching. In this section, assume $P\in C\cap C'$ where $C, C'$ are distinct irreducible curves in $\mathcal{H}_{\mathcal{X}}$ with generic points $\eta$ and $\eta'$ respectively. Let $\pi_P$ and $\delta_P$ be primes defining $C$ and $C'$ at $P$ as before. After blowing up the model at $P$ once, let $\Sigma$ denote the exceptional curve with generic point $\epsilon$ and let $\tilde{C}$ and $\tilde{C}'$ denote the strict transforms of $C$ and $C'$ respectively. Let the two new intersection points be $Q_1$ (where $\epsilon$ intersects $\tilde{C}$) and $Q_2$ (where $\epsilon$ intersects $\tilde{C}'$). 


\begin{lemma}[Blowing up a cold point]
\label{propblowupcold}
Let $P$ be a cold point and let $D_P = \brac{u_P\pi_P^m, v_P\delta_P}$ where $0<m<\ell$. Let $\phi : \mathcal{X}_{1}\to \mathcal{X}$ denote the blowup at point $P$. Then the exceptional curve $\Sigma$ obtained is of Type 1b and both $Q_1= C\cap \Sigma$ and $Q_2= C'\cap \Sigma$ are cold points.
\end{lemma}

\begin{figure}[hh]
\[
\begin{tikzcd}
\phantom{.} &   \re{C'}  \arrow[dd, dash] & \phantom{.}  \\
\re{C} \arrow[rr, dash, "\hspace*{3mm}\\ \bl{P}"] &  & \phantom{.}  \\
\phantom{.}  & \phantom{.}   &  \phantom{.} 
\end{tikzcd} \leadsto 
 \begin{tikzcd}
 \phantom{.}  & \re{\Sigma} \arrow[ddd, dash]  & \phantom{.}  \\
    \re{\tilde{C}} \arrow[rr, dash, "\hspace*{5mm}\\ \bl{Q_1}"]  & \phantom{.} & \phantom{.} \\
   \re{\tilde{C}'} \arrow[rr, dash , "\hspace*{5mm}\\ \bl{Q_2}"]  &  \phantom{.}   & \phantom{.} & \phantom{.} \\
\phantom{.}  & \phantom{.} & \phantom{.} & \phantom{.} 
\end{tikzcd}\]
\caption{Blowup of a point}
\end{figure}
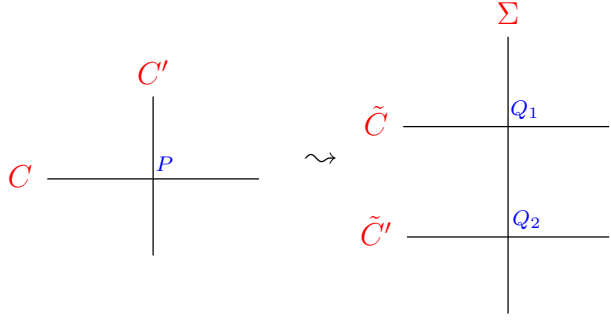

\begin{proof} Look at the local blow-up $\mathcal{Z}:= \mathrm{Proj} \brac{\frac{\widehat{A_P}[x, y]}{(x\pi_P -y \delta_P)}} \to \Spec(\widehat{A_P})$ at the maximal ideal of $\widehat{A_P}$. Setting $t = y/x$, we have $\pi_P = t\delta_P$.  Thus $\mathcal{Z}$ is the union of open affines $\Spec \frac{\widehat{A_P}[t]}{(\pi_P-t\delta_P)}$ and $\Spec \frac{\widehat{A_P}[\frac{1}{t}]}{\brac{\frac{\pi_P}{t}-\delta_P}}$ glued appropriately.

 By (\href{https://stacks.math.columbia.edu/tag/085S}{Lemma 085S}, stacks-project), we have $\widehat{\mathcal{O}_{\mathcal{X}_1, \epsilon}}=: \widehat{A_{\epsilon}} = \brac{\frac{\widehat{A_P}[t]}{\brac{\pi_P-t\delta_P}}}_{\brac{\pi_P,\delta_P}}^{\widehat{\phantom{xxxxxxx}}} $. Thus in $F_{\epsilon}$, the fraction field of $\widehat{A_{\epsilon}}$, both $\pi_P$ and $\delta_P$ are parameters.  Since $D_{\epsilon} = \brac{u_P\pi_P^m, v_P\delta_P}$, it has index at most $\ell$. The residue of $D_{\epsilon}$ is equal to $\overline{\frac{u_P\pi_P^m}{v_P^m\delta_P^m}} = \overline{u_Pv_P^{-m}t^m}$ which is non-trivial in the residue field $k_{\epsilon}=k_P(t)$. Therefore $\epsilon$ is of Type 1b.

Note that $\widehat{\mathcal{O}_{\mathcal{X}_1, Q_1}}=: \widehat{A_{Q_1}} = \brac{\frac{\widehat{A_P}[t]}{\brac{\pi_P-t\delta_P}}}_{\brac{t,\delta_P}}^{\widehat{\phantom{xxxxxxx}}} $ where $t$ defines $\tilde{C}$ and $\delta_P$ defines $\Sigma$ at $Q_1$ (cf. \cite{S07}, Pg 832, paragraph 1). Thus  over $F_{Q_1}$, the fraction field of $\widehat{A_{Q_1}}$, 
\[D_{Q_1}  = \brac{u_P{\delta}_P^mt^m, v_P{\delta}_{P}} = \brac{u_Pt^m, v_P{\delta}_{P}} + \brac{{\delta}_P^m,v_P{\delta}_{P}} = \brac{u_Pv_P^{-m}t^m, v_P{\delta}_{P}} \in \Br(F_{Q_1}).\]

Similarly $\widehat{\mathcal{O}_{\mathcal{X}_1, Q_2}}=: \widehat{A_{Q_2}} = \brac{\frac{\widehat{A_P}[1/t]}{\brac{\pi_P/t-\delta_P}}}_{\brac{1/t, \pi_P}}^{\widehat{\phantom{xxxxxxx}}} $ where $1/t$ defines $\tilde{C}'$ and $\pi_P$ defines $\Sigma$ at $Q_2$. Thus  over $F_{Q_2}$, the fraction field of $\widehat{A_{Q_2}}$ and for $s$ with $ms\cong 1\mod \ell$,
\[D_{Q_2}  = \brac{u_P{\pi}_P^m, v_P\frac{\pi_P}{t}} = \brac{u_P\pi_P^m, v_P\frac{1}{t}} + \brac{u_P{\pi}_P^m,\pi_{P}} = \brac{u_P\pi_P^m, u_P^{-s}v_P\frac{1}{t}} \in \Br(F_{Q_2}).\]

Hence both $Q_1$ and $Q_2$ are cold points. \end{proof}

We now  eliminate certain types of closed points listed in the classification in Section \ref{subsectionclosedpointsclassification}. 

\begin{proposition} \label{propnospecialpoints}
There exists a regular proper model such that $S_0$ does not contain points of Type $A_{11}^{s}, B_{11}^{s}, B_{21}^{s}, C_{11}^{Cool}, C_{12}^{Cool}, C_{12}^{Chilly}$ and  $C_{22}^{-}$.
\end{proposition}

\begin{proof}
Let $P$ denote an intersection point of one of types listed in the proposition and let $\Sigma$ and $\epsilon$ denote the exceptional curve and its generic point obtained after blowing up $P$ once. The following subtypes can be avoided by blowing up the model once at $P$. 

\textbf{Type $A_{11}^{s}$}: Since $D_P$ is split, $D\otimes F_{\epsilon}$ is split too and hence $\epsilon$ is of Type 0. Thus the two new intersection points are obtained by Type 1a curves ($\tilde{C}$ or $\tilde{C}'$) intersecting a curve of Type 0 ($\Sigma$). $\sqbrac{{A_{11}^{s}\mapsto A_{10}^{s} + A_{10}^s}}$.

\textbf{Type $B_{i1}^{s}$ ($i=1,2$)}: Since $D_P$ is split, $D\otimes F_{\epsilon}$ is split too and hence $\epsilon$ is of Type 0. Thus the two new intersection points are obtained by Type 1b/2 or 1a curves ($\tilde{C}$ or $\tilde{C}'$) intersecting a curve of Type 0 ($\Sigma$).  $\sqbrac{{B_{i1}^{s}\mapsto B_{i0}^{s} + A_{10}^s}}$.

\textbf{Type $C_{1i}^{Cool}$ ($i=1,2$)}:  $\sqbrac{{C_{11}^{Cool}\mapsto B_{10}^{s} + B_{i0}^s}}$ (cf. \cite{S07}, Theorem 2.6).

\textbf{Type $C_{12}^{Chilly}$}: This subtype can be avoided by blowing up the model $\mathcal{X}$ consecutively. Since $D_P = (v_P, \pi_P^j\delta_P)$, after one blowup at $P$, $D\otimes F_{\epsilon}$ has index at most $\ell$ and hence $\epsilon$ is of Type 0 or 1. Thus the two new intersection points are $Q_1$ (Type 11/01 : where $\Sigma$ intersects $\tilde{C}$) and $Q_2$ (Type 12/02 : where $\Sigma$ intersects $\tilde{C}'$).
 
Let us investigate the case when $\epsilon$ is of Type 1b. Then as in the proof of Lemma \ref{propblowupcold}, $D_{Q_2} = \left(v_P, \pi_P^{j+1}\frac{1}{t}\right)$ where $1/t$ defines $\tilde{C}'$ and $\pi_P$ defines $\Sigma$ at $Q_2$ with $\delta_P = \frac{\pi_P}{t}$. Hence $Q_2$ is again of Type $C_{12}^{Chilly}$. However the ramification along the Type 1b curve ($C\leadsto \Sigma$) has changed as evinced by the increase $j\leadsto j+1$. We can keep blowing up the intersection points of the strict transforms of $C'$ and the exceptional curve repeatedly till the new exceptional curve is of Type 0 or 1a and thus eliminate intersection points of the shape $C_{12}^{Chilly}$.

\textbf{Type $C_{22}^{-}$}:  This subtype can again be avoided by blowing up the model $\mathcal{X}$ an appropriate number of times at $P$.  Since $D_P$ has index at most $\ell$, after one blowup, $\epsilon$ is of Type 0, 1a or 1b. Thus the two new intersection points are obtained by Type 2 curves ($\tilde{C}$ or $\tilde{C}'$) intersecting a curve of Type 0,1a or 1b ($\Sigma$). In case $B_{21}^{s}$ or $C_{12}^{Chilly}$ points are generated, further blow up as in the previous steps to eliminate them. \end{proof}

\subsection{Limiting neighbours} 
We introduce the terminology that the closed points $P$ and $Q$ in $S_0$ are \textit{Type $x$ neighbours} if they both lie on the closure (denoted $\overline{\eta}$) of some $\eta\in N'_0$ of Type $x$ where $x\in \{0, 1a, 1b, 2\}$. Let $P, C, C', \eta, \eta', \pi_P, \delta_P, \tilde{C}, \tilde{C}', \Sigma, \epsilon$ be as before. We first begin with the following proposition that records the configuration of $Y$ when $\mathcal{X}$ is blown up at a hot point $P$ once. 

\begin{proposition}
\label{propblowuphotexamine} Let $P$ be a hot point of $\mathcal{X}$ and let $\phi : \mathcal{X}_{new}\to \mathcal{X}$ be the blowup at $P$. Without loss of generality, let $D_P = (u_P, \pi_P)$. Then $Q_2$ is a hot point in $\mathcal{X}_{new}$ while $Q_1$ is a chilly point. Further the following table records possible configurations of $Y_{\eta}$, $Y_{\eta'}$ and $Y_{\epsilon}$. In particular if $Y_{\eta'}$ is not of Type RAM, then $Y_{\epsilon}$ is not of Type NONRES. \end{proposition}

\begin{proof}
Since $D_P=\brac{u_P, \pi_P}$ where $u_P\in \widehat{A_P}^*$, $\eta$ is of Type 1b while $\eta'$ is of Type 2. Thus $D_{Q_1} = \brac{u_P, \delta_Pt}$ where $t$ defines $\tilde{C}$ and $\delta_P$ defines $\Sigma$ at $Q_1$ where $\pi_P = t\delta_P$. Similarly $D_{Q_2} = \brac{u_P, \pi_P}$ where $1/t$ defines $\tilde{C}'$ and $\pi_P$ defines $\Sigma$ at $Q_2$. Thus we have replaced $P$ with hot point $Q_2$ and chilly point $Q_1$ in $\mathcal{X}_{new}$.

Let $\epsilon\in \mathcal{X}_{new}$ denote the generic point of the exceptional curve $\Sigma$ and by abuse of notation, a parameter of $F_{\epsilon}$. Since $D_{\eta'}$ is division (\cite{S07}, Theorem 2.5, Proposition 0.5), $Y_{\eta'}$ cannot be SPLIT. If $Y_{\eta'}$ is of Type RAM, then $Y_{\eta}$ cannot be SPLIT or RES. If $Y_{\eta'}$ is of Type RES, then $Y_P \simeq \prod F_P$ and hence $Y_{\eta}$ can only be SPLIT or  NONRES.  Finally observe that if $Y_{\eta'}$ is NONRES, then $Y_{P}$ is non-split by Proposition \ref{propositionYathotpoints} and hence $Y_{\eta}$ cannot be SPLIT by Lemma \ref{lemmaatetasplit-1}. 

Thus we have the following table (in which we use the notations $v\in \widehat{A_P}^*$, $w\in \widehat{A_{\epsilon}}^*$, $0<r<\ell$ and $F_P\brac{\sqrt[\ell]{u_P}}$ to be the unique degree $\ell$ unramified field extension of $F_P$).

\centering
$
\begin{array}{|c|c|c|c|c|}
\hline 
Y_{\eta'} & Y_{\eta} & Y_P &  Y_{\epsilon}  & \mrm{Type\ of \ }Y_{\epsilon}\\
\hline
{\mrm{RAM}}& \mrm{RAM} & F_P\brac{\sqrt[\ell]{v\pi_P^r\delta_P}} & F_{\epsilon}\brac{\sqrt[\ell]{w\epsilon^{r+1}}} & {{\mrm{RAM/NONRES}}} \\
\hline
{\mrm{RAM}}& \mrm{NONRES} & F_P\brac{\sqrt[\ell]{v\delta_P}} & F_{\epsilon}\brac{\sqrt[\ell]{w\epsilon}} & {\mrm{RAM}}\\
\hline
{\mrm{RES }}& \mrm{SPLIT} & \prod F_P & \prod F_{\epsilon} & {\mrm{SPLIT}}\\
\hline
{\mrm{RES}}& \mrm{NONRES} & \prod F_P & \prod F_{\epsilon}  & {\mrm{SPLIT}}\\
\hline
{\mrm{NONRES}}& \mrm{RAM} & F_P\brac{\sqrt[\ell]{v\pi_P}} & F_{\epsilon}\brac{\sqrt[\ell]{w\epsilon}} & {\mrm{RAM}}\\
\hline
{\mrm{NONRES}}& \mrm{RES} & F_P\brac{\sqrt[\ell]{u_P}} & F_{\epsilon}\brac{\sqrt[\ell]{u_P}} & {\mrm{RES}}\\
\hline
{\mrm{NONRES}}& \mrm{NONRES} & F_P\brac{\sqrt[\ell]{u_P}}  & F_{\epsilon}\brac{\sqrt[\ell]{u_P}} & {\mrm{RES}}\\
\hline
\end{array}$
\label{TableYathotpointQ2}
\captionof{table}{Table giving shape of $Y$ at hot point $Q_2$}\end{proof}

In the following proposition, we blow up further so as to arrange for a model $\mathcal{X}$ such that its marked points do not have any `difficult' neighbours. This will be helpful when constructing $E_{1,\eta}$ and $E_{2,\eta}$ along codimension one points $\eta$ lying in the special fiber $X_0$.

\begin{proposition}
\label{propblowupchilly}
\label{propblowuphot}
\label{propblowupB10}

There exists a sequence of blowups $\phi : \mathcal{X}_{new}\to \mathcal{X}$ such that for any $\eta\in \brac{N'_0}_{\mathcal{X}_{new}}$, the following hold: 

\begin{enumerate}
\item
If $\eta$ is of Type 0 containing a $A_{10}^{s}$, $B_{10}^{s}$ or $B_{20}^{s}$ marked point $P$, then there is at most only one other marked point $Q\in \overline{\eta}$ and it is of Type $A_{00}^{s}$.
\item
If $\eta$ is of Type 1b containing a $C_{11}^{Chilly}$ marked point $P$, then there is at most only one other marked point $Q\in \overline{\eta}$ and it is of Type $B_{11}^{ns}$ or $C_{11}^{Chilly}$.
\item
If $\eta$ is of Type 1b containing a $C_{12}^{Hot}$ marked point $P$, then there is at most only one other marked point $Q\in \overline{\eta}$ and it is of Type $B_{11}^{ns}$. 
\end{enumerate}

\end{proposition}

\begin{proof} Let $P, C, C', \eta, \eta', \pi_P, \delta_P, \tilde{C}, \tilde{C}', \Sigma, \epsilon$ be as before. Recall that $Q_1 = \tilde{C}\cap \Sigma$ while $Q_2 = \tilde{C}'\cap \Sigma$ are the two new marked points obtained after blowing up at $P$. We investigate each case separately.

1. Let $\eta$ be of Type 0 with a marked point $P$ as above. Since $D_P$ is split, $\epsilon$ is of Type 0 and has exactly two marked points $Q_1$ and $Q_2$ lying on it. For $\{e,f\}=\{1,2\}$, we see that $Q_e$ replaces $P$ and has at most one Type 0 neighbour $Q_f$ which is necessarily of Type $A_{00}^{s}$.

2. Let $\eta$ be of Type 1b with a chilly point $P$. This case is reminiscent of the breaking of chilly loops in (\cite{S07}, Corollary 2.9).

If $D_P = \left(u_P, \pi_P^m\delta_P^n\right)$ for some unit $u_P$ and $0 < m, n < \ell$, we say the algebra is of the shape $[m,n]_{C, C'}$. Let $x^{-1} = \frac{1}{x}\in \brac{\mathbb{Z}/\ell\mathbb{Z}}^*$. Then $[m,n]_{C, C'} = [1, nm^{-1}]_{C, C'} = [mn^{-1}, 1]_{C, C'}$ as $\brac{u_P, \pi_P^m\delta_P^n}  = \brac{u_P, \brac{\pi_P^m\delta_P^n}^{m^{-1}m}} =  m \brac{u_P, \brac{\pi_P^m\delta_P^n}^{m^{-1}}} = \brac{u_P^m , \pi_P\delta_P^{nm^{-1}}}$.

Since $P$ is a $C_{11}^{Chilly}$ point, $D_P$ is of the shape $[1,j]_{C, C'}$ for some $0<j<\ell$. After a single blow up, as in the proof of Proposition \ref{propnospecialpoints}, $D_{Q_1} = [j, j+1]_{\tilde{C}, \Sigma}$ and $D_{Q_2}  = [1, j+1]_{\tilde{C}', \Sigma}$. Hence either $j+1\cong 0\mod \ell$ and $\epsilon$ is a Type 1a curve \footnote{Note that when $j+1\cong 0\mod \ell$, $D_{Q_2}$ is still not a split algebra and hence $\epsilon$ cannot be a Type $0$ curve by Lemma \ref{lemmaatetasplit}.} or $j+1<\ell$, $\epsilon$ is a Type 1b curve and both $Q_1$ and $Q_2$ are $C_{11}^{Chilly}$ points again.  If $j+1<\ell$, blow up the point $Q_2$ again. Repeating this process, we get a model $\mathcal{X}_{1}$ where the closure of the strict transform of $C'$ intersects an exceptional curve of Type 1a. Carry out the same procedure on $Q_1$, the other intersection point till the closure of the strict transform of $C$ also intersects an exceptional curve of Type 1a.

3. Let $\eta$ be of Type 1b with a hot point $P$. By Proposition \ref{propblowuphotexamine}, blowing up the model at $P$ yields a hot point $Q_2$ which has only $Q_1$, a chilly point, as a Type 1b neighbour. Now following the proof of the previous case and blowing up the chilly point $Q_1$ repeatedly, we see that $Q_2$ will only have a $B_{11}^{ns}$ point as a Type 1b neighbour at most. \end{proof}

\subsection{The final model $\mathcal{X}$}
\label{subsectionfinalmodel}
Recall that $\mathcal{X}$ is arranged such that the divisor $\mathcal{H}_{\mathcal{X}}$ is in good shape. We note that this property is preserved under blowups (cf. proof of Proposition \ref{propgoodshape}). Thus using Propositions \ref{propnospecialpoints}, \ref{propblowupchilly}, from now on we can and do assume that our model $\mathcal{X}$ has no marked points of Type $A_{11}^{s}$, $B_{11}^{s}$, $B_{21}^{s}$, $C_{11}^{Cool}$, $C_{12}^{Cool}$, $C_{12}^{Chilly}$ and  $C_{22}^{-}$. Further we also assume that any $C_{11}^{Chilly}$ point has only Type 1b neighbours which are either again $C_{11}^{Chilly}$ or $B_{11}^{ns}$, any $C_{12}^{Hot}$ point can be a Type 1b neighbour at most of one other point which should be of Type $B_{11}^{ns}$ and any $A_{10}^{s}$, $B_{10}^s$ or $B_{20}^s$ point has at most one Type 0 neighbour which will necessarily be of Type $A_{00}^s$. Note also that in constructing such a model (cf. the proof of Proposition \ref{propblowuphot}, hot point case), we would have blown up the original hot points exactly once and hence would have arranged for the shape of $Y$ at any hot point in the final model to be as given by Proposition \ref{propblowuphotexamine}. We finally fix parameters $\pi_{\eta}$ for each $\eta\in N'_0$ as in Section \ref{subsubsection-modelparameter}, which further determine a system of parameters for each $P\in S_0$.

\subsection{Graphs}
\subsubsection{Labelling curves with $\{$Ch, C, H, Z$\}$ labels}
\label{ChCHZlabels}
Let $\gamma\in N'_0$ be of Type 1b with $Y_{\gamma}\simeq \prod F_{\gamma}$. Using Proposition \ref{propblowupchilly}, we label it as follows:

\begin{itemize}
\item[-]{$\gamma$ is a \textbf{Ch-curve} if $\overline{\gamma}\cap S_0$ contains a chilly point.} Note that $\overline{\gamma}\cap S_0$ will consist of marked points of Types  $B_{11}^{ns}$ and $C_{11}^{Chilly}$ only.
\item[-]{$\gamma$ is a \textbf{C-curve} if $\overline{\gamma}\cap S_0$ contains a cold point.} Note that $\overline{\gamma}\cap S_0$ will consist of marked points of Types $B_{10}^{s}$, $B_{11}^{ns}$, $C_{11}^{Cold}$ and $C_{12}^{Cold}$ only. 
\item[-]{$\gamma$ is a \textbf{H-curve} if $\overline{\gamma}\cap S_0$ contains a hot point.} Note that $\overline{\gamma}\cap S_0$ will consist of marked points of Types  $B_{11}^{ns}$ and $C_{12}^{Hot}$ only.
\item[-]{$\gamma$ is a \textbf{Z-curve} if it is not a Ch, C or H-curve.} Note that $\overline{\gamma}\cap S_0$ will consist of marked points of Types $B_{10}^{s}$ or $B_{11}^{ns}$ only.
\end{itemize}

Thus the sets of Ch, C, H and Z-curves are mutually disjoint. Note also that when you blow up a cold point $P$ on a C-curve $\eta$, then the exceptional curve obtained is again a C-curve and the two new marked points obtained are again cold points (Lemmata \ref{lemmaatetasplit-1} and \ref{propblowupcold}).

\subsubsection{A partial dual graph}
In subsequent sections, we will prescribe patching data $E_{1,\eta}$ and $E_{2,\eta}$ for $\eta\in N_0:= N'_0\cap X_0$. 
Ensuring compatibility at branches can be, in part, turned into a colouring problem for \textit{a partial dual graph} built as follows: 

Construct an undirected graph $\Delta$ with vertex set $V_{\Delta}$ consisting of $\eta\in N'_0$ of Type 1b or 2.  The edge set $J_{\Delta}$ consists of cold points in $S_0$. So if $\overline{\eta}, \overline{\eta'}\in V_{\Delta}$ intersect at a cold point $P$ in our model, then they are joined by an edge labelled $P$. Note that therefore multiple edges between distinct vertices are allowed, while self loops are not. Blowing up a cold point $P$ has the effect of adding a vertex in middle of the edge $P$ in $\Delta$.

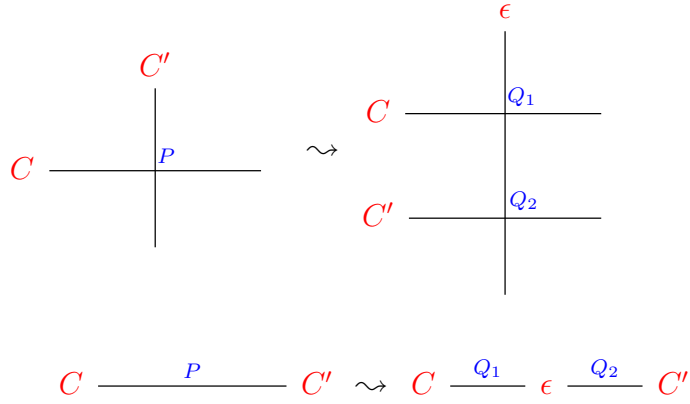
\[
\begin{tikzcd}
\phantom{.} &   \re{C'}  \arrow[dd, dash] & \phantom{.}  \\
\re{C} \arrow[rr, dash, "\hspace*{3mm}\\ \bl{P}"] &  & \phantom{.}  \\
\phantom{.}  & \phantom{.}   &  \phantom{.} 
\end{tikzcd} \leadsto 
 \begin{tikzcd}
 \phantom{.}  & \re{\epsilon} \arrow[ddd, dash]  & \phantom{.}  \\
    \re{C} \arrow[rr, dash, "\hspace*{5mm}\\ \bl{Q_1}"]  & \phantom{.} & \phantom{.} \\
   \re{C'} \arrow[rr, dash , "\hspace*{5mm}\\ \bl{Q_2}"]  &  \phantom{.}   & \phantom{.} & \phantom{.} \\
\phantom{.}  & \phantom{.} & \phantom{.} & \phantom{.} 
\end{tikzcd}\]

\[
\begin{tikzcd}
\re{C}  \arrow[rr, dash, "\bl{P}"] & \phantom{.} & \re{C'}
\end{tikzcd} \leadsto 
 \begin{tikzcd}
\re{C}  \arrow[r, dash, "\bl{Q_1}"] & \re{\epsilon}  \arrow[r, dash, "\bl{Q_2}"] & \re{C'}
\end{tikzcd}\]

\subsubsection{Primary colouring of $\Delta$}
\label{sectionfirstcolouring}
We now present a combinatorial colouring proposition, reserving for later the explanation of the precise relevance of this to the patching problem. The following guarantees that after finitely many blowups of cold points, there exists a `suitable' colouring of the vertices of $\Delta$ with the colours red (R), green (G), and blue (B). More precisely


\begin{proposition}
\label{propgraphcolour}
There exists a sequence of blowups of cold points on C-curves, $\phi : \mathcal{X}_{new}\to \mathcal{X}$, such that the vertices of the new partial graph $\Delta_{new}$ can be coloured with colours blue (B), green (G) and red (R) such that 
\begin{enumerate}
\item
$\eta\in V_{\Delta_{new}}$ is coloured green if and only if $\eta$ is not a C-curve,
\item
Any non-green vertex with an edge to a green vertex is coloured red,
\item
Any non-green vertex with an edge to a red vertex is blue,
\item
Any non-green vertex with an edge to a blue vertex is red.
\end{enumerate}
\end{proposition}

\begin{proof}
Without loss of generality, assume $\Delta$ is connected (otherwise repeat the same proof for each connected component). Let $W \subseteq V_{\Delta}$ denote the set of C-curves. Colour every $\eta\in V_{\Delta}\setminus W$ with green. Thus $\Delta$ is partially coloured. For $v\in V_{\Delta}$ which is uncoloured and $X\in \{R, G, B\}$, define the function $d(v, X)$ for any partial colouring of $\Delta$ as follows: 

\newpage

\begin{itemize}
\item
Set $d(v, X)=1$ if there is an edge between  $v$ and a vertex coloured $X$.
\item
Set $d(v, X) = 0$ if there is no edge between  $v$ and any vertex coloured $X$.
\end{itemize}

The following algorithm colours vertices in $W$ with $R$ and $B$ in a compatible fashion.

\textit{Step 1}:  Colour with red (R), all uncoloured vertices $v\in W$ such that $d(v, G) =1$. If no such vertices exist, colour an arbitrary uncoloured vertex with red (R).\\
\textit{Step 2}: The previous step might lead to a situation where two red vertices are connected by an edge. For every such edge $P : \eta-\eta'$, blow up the cold point $P$ (which note, is on two C-curves). As we have already observed, the exceptional curve obtained is again a C-curve  and the new marked points are cold points. In the new partial dual graph, this introduces a new vertex (corresponding to the exceptional curve) breaking the edge $P$ into two edges $Q_1$ and $Q_2$. Colour this new vertex with blue (B). If all vertices are coloured, terminate. \\
\textit{Step 3}: Colour with blue(B), all uncoloured vertices $v$ such that $d(v, R) =1$.  If no such vertices exist, colour an arbitrary uncoloured vertex with blue (B).\\
\textit{Step 4}: The previous step might lead to a situation where two blue vertices are connected by an edge. For every such edge $P : \eta-\eta'$, blow up the cold point $P$ (which note, is on two C-curves). As before, in the new partial dual graph, this introduces a new vertex (corresponding to the exceptional curve) breaking the edge $P$ into two edges $Q_1$ and $Q_2$. Colour this new vertex with red (R). If all vertices are coloured, terminate. \\
\textit{Step 5}: Colour with red (R), all uncoloured vertices $v\in W$ such that $d(v, B) =1$. If no such vertices exist, colour an arbitrary uncoloured vertex with red (R).\\
\textit{Step 6}: Go to Step 2.

Note that in Steps 1, 3 and 5 we colour at least one uncoloured vertex each time. In Steps 2 and 4, though we introduce new vertices, they always correspond to C-curves and we colour them with R or B in the same step. Since $|V_{\Delta}| < \infty$, the algorithm terminates after finitely many steps. Each partial colouring obtained satisfies Properties 1-4. Hence when the algorithm terminates, we will end up with a compatible colouring of $V_{\Delta}$. \end{proof}

\subsubsection{An extended rainbow colouring of $\Delta$}
\label{section-finalcolouring}
We refine the colouring of $\Delta$ by colouring \textbf{over} $\eta$ which are Ch, H or Z-curves as follows: Let $\eta\in V_{\Delta}$ be a Ch, H or Z-curve and let $a = \brac{a'_{i,\eta}}_i\in \prod F_{\eta}$.

- If each $a'_{i,\eta}$ is a unit (up to $\ell^{\mrm{th}}$ powers) in $\widehat{A_{\eta}}$, then colour $\eta$ violet (V) if it is a Ch-curve, indigo (I) if it is a H-curve, and black (Bl) if it is a Z-curve.

- If at least one $a'_{i,\eta}$ is not a unit (up to $\ell^{\mrm{th}}$ powers) in $\widehat{A_{\eta}}$, then colour $\eta$ yellow (Ye) if it is a Ch-curve, orange (O) if it is a H-curve, and white (W) if it is a Z-curve.

Thus we get a nine-colouring of $V_{\Delta}$ with colours violet (V), indigo (I), blue (B), green (G), yellow (Ye), orange (O), red (R), black (Bl) and white(W). 

\newpage

\section{Patching data at marked points in $S_0$}
Let $P\in S_0$ be the intersection of two distinct irreducible curves $C_1$ and $C_2$ of $\mathcal{H}_{\mathcal{X}}$. Let $\eta_1$ and $\eta_2$ denote the generic points of $C_1$ and $C_2$ respectively. Let $\pi_P$ and $\delta_P$ be primes defining $C_1$ and $C_2$ at $P$ fixed as in Section \ref{subsectionfinalmodel}. As before, if $Y_x \simeq \prod F_x$, we let $a = \left(a'_{i,x}\right)_i$, where $a'_{i,x}\in F_x$. We will now prescribe $E_{j,P}$ for $j=1,2$ at $P\in S_0$ in accordance with the following heuristic:
\begin{itemize}
\item[-]
If $\eta\in N'_0$ is of Type 0 or 1a, then both $E_{1,P}$ and $E_{2,P}$ should be unramified along $\eta$,
\item[-]
If $\eta\in N'_0$ is coloured G, V, I or Bl, then both $E_{1,P}$ and $E_{2,P}$ should be unramified along $\eta$,
\item[-]
If $\eta\in N'_0$ is coloured R, O, Ye or W, then $E_{1,P}$ should be ramified along $\eta$ while $E_{2,P}$ should be unramified along $\eta$,
\item[-]
If $\eta\in N'_0$ is coloured B, then $E_{1,P}$ should be unramified along $\eta$ while $E_{2,P}$ should be ramified along $\eta$.
\end{itemize}

\subsection{Points not of Type $A_{00}^s$ }
\begin{proposition}
\label{propositionEatpoints} 
Let $P\in S_0$ be such that it is not of Type $A_{00}^s$. Then for each $j=1,2$, there exist cyclic degree $\ell$ extensions $E_{j,P}/F_P$ and elements ${a}_{j,P}\in Y_P$ such that 
\begin{enumerate}
\item 
$a_{1,P}a_{2,P}= a$.

\item
$D\otimes {E}_{j,P}$ has index at most $\ell$. 

  \item
  $D\otimes Y\otimes {E}_{j,P}$ is split. 
  
\item
$a_{j,P}$ is a norm from $E_{j,P}\otimes Y_P/Y_P$. 

\item
$\N_{Y_P/F_P}\brac{a_{j,P}}=1$.
\item
 Each $E_{j,P}$ is either a split extension or $D\otimes E_{j,P}$ is split.

\end{enumerate}
\end{proposition}

\begin{proof} We investigate each type of point separately. In every case, we will determine $E_{1,P}, E_{2,P}$ and $a_{1,P}$ and set $a_{2,P} = aa^{-1}_{1,P}$, thus ensuring that Property 1 holds. Since $\N(a)=1$, Property 5 will also be satisfied provided $\N\brac{a_{1,P}}=1$. By (\cite{S97}, Proposition 1.2), Property 2 holds for any closed point.

We adopt the following notations in the proof: $u_P, v_P, w_P\in \widehat{A_P}^*$, $0<r, s, m, j<\ell$. If $Y_P$ is split, by Proposition \ref{propaisingoodshape}, $a = \brac{a'_{i,P}}$ where $a'_{i,P}=z_{i,P}\pi_P^{m_i}\delta_P^{n_i}$ where $m_i,n_i\in \mathbb{Z}$ and $z_{i,P}\in \widehat{A_P}^*$. Also since $\N(a)=1$, we have $\prod z_{i,P}=1$ and $\sum m_i = \sum n_i = 0$. $L_P$ denotes the unique non-split degree $\ell$ extension of $F_P$ unramified at $\widehat{A_P}$ and $H_P$, the extension $F_P\left(\sqrt[\ell]{u_P\pi_P^m + v_P\delta_P}\right)$.

\newpage
\textbf{Type $A_{10}^s$}: Wlog, assume $\eta_1$ is of Type 1a and $\eta_2$ is of Type 0. Note that $D_P$ is split. The following choices for $E_{j,P}$ and $a_{j,P}$ satisfy Properties 1-6.
\begin{table}[H]
\centering
\footnotesize
\arraycolsep=0.5pt
\medmuskip =0.5mu 
\[
\begin{array}{|c|c|c|c|c|c|c|}
\hline 

\mrm{Row} & {\eta_1} & {\eta_2} & {E_{1,P}} & {E_{2,P}} & {a_{1,P}} & {a_{2,P}=aa_{1,P}^{-1}} \\
\hline
0.1 & 1a  &  0 & \prod F_P   & {\prod F_P} &   {a} & {1} \\
\hline
\end{array}
\]
\captionof{table}{Patching data at points of Type $A_{10}^{s}$}
\label{TableEatA10spoint}
\end{table}

\textbf{Type $B_{10}^s$}: Wlog, assume $\eta_1$ is of Type 1b and $\eta_2$ is of Type 0. By Proposition \ref{propblowupchilly} and Section \ref{ChCHZlabels}, $\eta_1$ cannot be a Ch or H-curve and hence isn't coloured V, I, Ye or O. The following table gives the choice for $E_{j,P}$ and $a_{j,P}$.
\begin{table}[H]
\centering
\footnotesize
\arraycolsep=0.5pt
\medmuskip =0.5mu 
\[
\begin{array}{|c|c|c|c|c|c|}
\hline 

\mrm{Row} & \mrm{\eta_1}  & {E_{1,P}} & {E_{2,P}} & {a_{1,P}} & {a_{2,P}=aa_{1,P}^{-1}} \\
\hline
1.1 &  {R}, W  & F_P\brac{\sqrt[\ell]{\pi_P}} & {\prod F_P} &  \brac{\pi_P^{m_i}}_i & \brac{z_{i,P}\delta_P^{n_i}}_i \\
\hline
1.2 & {G}, Bl & \prod F_P & \prod F_P & a & 1\\
\hline
1.3 &  {B}    &  {\prod F_P} & F_P\brac{\sqrt[\ell]{\pi_P}} &  \brac{z_{i,P}\delta_P^{n_i}}_i & {\brac{\pi_P^{m_i}}_i} \\
\hline
\end{array}
\]
\captionof{table}{Patching data at points of Type $B_{10}^{s}$}
\label{TableEatB10spoint}
\end{table}
Since $D_P$ is itself split, Properties 3 and 6 hold while Property 4 holds by construction. Since $\sum m_i = \sum n_i = 0$ and $\prod z_{i,P}=1$, we have $\N\brac{a_{1,P}}=1$. Hence Property 5 holds.

\textbf{Type $B_{11}^{ns}$}: Wlog, assume $\eta_1$ is of Type 1b and $\eta_2$ is of Type 1a. Thus $D_P = \brac{u_P, \pi_P} \in \Br\brac{F_P}$. By Proposition \ref{propaisaunitifYnotsplit}, $a$ is a unit in the integral closure of $\widehat{A_P}$ in $Y_P$ if the latter is not split. By Proposition \ref{propositionaatB11nspoint}, if $Y_P$ is split, $a'_{i,P}= z_{i,P}\pi_P^{m_i}{\delta_P}^{\ell n'_i}$ with $\sum n'_i=0$. The following table\footnote{Row 2.2* is a special case when $\eta_1$ is Type 1b and green with $Y_{\eta_1}$ of Type RAM and $\eta_2$ of Type 1a with $Y_{\eta_2}$ of Type NONRES. In this situation, we choose $E_{1,P}=E_{2,P}=\prod F_P$ while $a_{1,P}=a$ and $a_{2,P}=1$.} gives the choice for $E_{j,P}$ and $a_{j,P}$.

\begin{table}[H]
\centering
\footnotesize
\arraycolsep=0.5pt
\medmuskip =0.5mu 
\[
\begin{array}{|c|c|c|c|c|c|c|}
\hline 

\mrm{Row} & \mrm{\eta_1} & {\eta_2} & {E_{1,P}} & {E_{2,P}} & {a_{1,P}} & {a_{2,P}=aa_{1,P}^{-1}} \\
\hline
2.1 &  W, {R}, {O}, {Ye}  &  & F_P\brac{\sqrt[\ell]{\pi_P}} & L_P &  \brac{\pi_P^{m_i}}_i & \brac{z_{i,P}{\delta_P}^{\ell n'_i}}_i \\
\hline
2.2^* & {G} (\mrm{RAM})& (\mrm{NONRES}) & \prod F_P & \prod F_P & a & 1\\
\hline
2.2 & Bl,{I},{G},{V} & & L_P & L_P & a & 1\\
\hline
2.3 &  {B} &  &  L_P & F_P\brac{\sqrt[\ell]{\pi_P}} &  \brac{z_{i,P}{\delta_P}^{\ell n'_i}}_i &  {\brac{\pi_P^{m_i}}_i} \\
\hline
\end{array}
\]
\captionof{table}{Patching data at points of Type $B_{11}^{ns}$}
\label{TableEatB11nspoint}
\end{table}
Since $D_P = \brac{u_P, \pi_P}$ and $u_P$ becomes an $\ell^{\mrm{th}}$ power in $L_P$, $E_{j,P}$ splits $D$ in each case except Row $2.2^*$. In this row however, $Y_P = F_P\brac{\sqrt[\ell]{w_P\pi_P}}$ and hence $D\otimes Y_P$ is split. Thus Properties 3 and 6 hold. By Lemmata \ref{lemmanormfromE-split} and \ref{lemmanormfromE-nonsplit}, Property 4 holds. Since $\sum m_i = \sum n'_i= 0$ and $\prod z_{i,P}=1$, we have $\N\brac{a_{1,P}}=1$. Hence Property 5 holds.

\newpage

\textbf{Type $B_{20}^s$}: Wlog, assume $\eta_1$ is Type 2 and $\eta_2$ is Type 0. Thus $\eta_1$ is coloured G. The following choices satisfy Properties 1-6.
\begin{table}[H]
\centering
\footnotesize
\arraycolsep=0.5pt
\medmuskip =0.5mu 
\[
\begin{array}{|c|c|c|c|c|c|}
\hline 

\mrm{Row} & \mrm{\eta_1} & {E_{1,P}} & {E_{2,P}} & {a_{1,P}} & {a_{2,P}=aa_{1,P}^{-1}} \\
\hline
3.1 &  {G}  &   {\prod F_P} & {\prod F_P} &  a &1 \\

\hline
\end{array}
\]
\captionof{table}{Patching data at points of Type $B_{20}^{s}$}
\label{TableEatB20spoint}
\end{table}

\textbf{Type $B_{21}^{ns}$}: Wlog, assume $\eta_1$ is Type 2 and $\eta_2$ is Type 1a. Thus $D_P = \brac{u_P, \pi_P} \in \Br\brac{F_P}$. By Proposition \ref{propaisaunitifYnotsplit}, $a$ is a unit in the integral closure of $\widehat{A_P}$ in $Y_P$ if the latter is not split.  By Proposition \ref{propositionaatB21nspoint}, if $Y_P$ is split, $a'_{i,P}= z_{i,P}{\delta_P}^{\ell n'_i}$ with $\sum n'_i=0$. The following table\footnote{Row 4.1* is a special case when $\eta_1$ is Type 2 with $Y_{\eta_1}$ of Type RAM and $\eta_2$ of Type 1a with $Y_{\eta_2}$ of Type NONRES. In this situation, we choose $E_{1,P}=E_{2,P}=\prod F_P$ while $a_{1,P}=a$ and $a_{2,P}=1$.} gives the choice for $E_{j,P}$ and $a_{j,P}$.

\begin{table}[H]
\centering
\footnotesize
\arraycolsep=0.5pt
\medmuskip =0.5mu 
\[
\begin{array}{|c|c|c|c|c|c|c|}
\hline 
\mrm{Row} & \mrm{\eta_1} & {\eta_2} & {E_{1,P}} & {E_{2,P}} & {a_{1,P}} & {a_{2,P}=aa_{1,P}^{-1}} \\
\hline
4.1^* & G (\mrm{RAM})& (\mrm{NONRES}) & \prod F_P & \prod F_P & a & 1\\
\hline
4.1 & {G} & & L_P & L_P & a & 1\\
\hline
\end{array}
\]
\captionof{table}{Patching data at points of Type $B_{21}^{ns}$}
\label{TableEatB21nspoint}
\end{table}

Since $D_P = \brac{u_P, \pi_P}$ and $u_P$ becomes an $\ell^{\mrm{th}}$ power in $L_P$, $E_{j,P}$ splits $D$ in Row 4.1. In Row 4.1*, $Y_P = F_P\brac{\sqrt[\ell]{w_P\pi_P}}$ and hence $D\otimes Y_P$ is split. Thus Properties 3 and 6 hold. By Lemmata \ref{lemmanormfromE-split} and \ref{lemmanormfromE-nonsplit}, Property 4 holds. Since $\N(a)=1$, so does Property 5.

\textbf{Type $C_{11}^{Chilly}$} : We assume that $D = D_{00} + \brac{u_P, \pi_P} + \brac{v_P, \delta_P} \in \Br(F)$ where $D_{00}$ is unramified at $A_P$ and $D_P = \brac{v_P, \pi_P^j\delta_P}$ where $0<j<\ell$. By Proposition \ref{propositionaatchillypoint}, if $Y_P$ is split, then $a'_{i,P}= z_{i,P}\brac{\pi_P^j\delta_P}^{n_i}\brac{\pi_P^{r_i\ell}}$ where $m_i = jn_i + r_i\ell$ and hence $\sum r_i=0$. In particular, $a$ is a unit (up to $\ell^{\mrm{th}}$ powers) in $\widehat{A_{\eta_1}}^*$ if and only if it is is a unit (up to $\ell^{\mrm{th}}$ powers) in $\widehat{A_{\eta_2}}^*$. 

Let $j=1, 2$. Recall that if $\eta_j$ is not a Ch-curve, then it is coloured G. The above discussion implies that if $\eta_1$ and $\eta_2$ are both Ch-curves, then they are both coloured Ye or both coloured V. Similarly if $\eta_1$ is coloured G and $\eta_{2}$ is a Ch-curve, then $\eta_{2}$ is coloured V. Likewise if $\eta_2$ is coloured G and $\eta_{1}$ is a Ch-curve, then $\eta_{1}$ is coloured V. Invoking Proposition \ref{propchillyexamine}, Table 13 below prescribes $E_{j,P}$ depending on the configuration of $a, Y, \eta_1, \eta_2$.

Since $D_P = \brac{v_P, \pi_P^j\delta_P}$ and $v_P$ becomes an $\ell^{\mrm{th}}$ power in $L_P$, $E_{j,P}$ splits $D$ in each case. Thus Properties 3 and 6 hold. By Lemma \ref{lemmanormfromE-nonsplit}, Property 4 holds for Rows 5.1-5.4, 5.8-5.9, 5.11 (when $Y_P$ is a field).  By Lemma \ref{lemmanormfromE-split}, Property 4 holds in the remaining cases (except for $\brac{E_{1,P}, a_{1,P}}$ in Row 5.6, where it is clear by observation). Finally since $\N(a)=1$ and for Row 5.6, $\sum n_i = \sum r_i = 0$ and $\prod z_{i,P}=1$, $\N\brac{a_{1,P}}=1$ for all rows. Hence Property 5 holds.\\

\begin{table}[H]
\centering
\footnotesize
\arraycolsep=0.5pt
\medmuskip =0.5mu 
\[
\begin{array}{|c|c|c|c|c|c|c|c|c|c|}
\hline 

\mrm{Row} & \eta_1 & {\eta_2}  &  {Y_{\eta_1}} & {Y_{\eta_2}} & {Y_P} &  {E_{1,P}} & {E_{2,P}} & {a_{1,P}} & {a_{2,P}=aa_{1,P}^{-1}} \\
\hline
5.1 &  {G} &  {G}  & \mrm{RAM} & \mrm{RAM} &  F_P\left(\sqrt[\ell]{w_P\pi_P^r\delta_P}\right)  & L_P &  L_P & {a} & {1} \\
\hline
5.2 &  {G} &  {G}  & \mrm{NONRES} & \mrm{RAM} &  F_P\left(\sqrt[\ell]{w_P\delta_P}\right)  & L_P &  L_P & {a} & {1} \\
\hline
5.3 & {G} &  {G}  & \mrm{RES} & \mrm{RES} & L_P  & L_P &  L_P & {a} & {1} \\
\hline
5.4 &  {G} &  {G}  & \mrm{NONRES} & \mrm{RES} &  L_P  & L_P &  L_P & {a} & {1} \\
\hline
5.5 &  {V} &  {V}  & \mrm{SPLIT} & \mrm{SPLIT} &  \prod F_P  & L_P &  L_P & {a} & {1} \\
\hline
5.6 & {Ye} &  {Ye}  & \mrm{SPLIT} & \mrm{SPLIT} &  \prod F_P  &  F_P\left(\sqrt[\ell]{\pi_P^j\delta_P}\right) &  L_P & {\brac{\brac{\pi_P^j\delta_P}^{n_i}\brac{\pi_P^{r_i\ell}}}_i} & {\brac{z_{i,P}}_i} \\
\hline
5.7 &  {G} &  {V}  & \mrm{NONRES} & \mrm{SPLIT} &  \prod F_P  & L_P &  L_P & {a} & {1} \\
\hline
5.8 &  {G} &  {G}  & \mrm{RAM} & \mrm{NONRES} &  F_P\left(\sqrt[\ell]{w_P\pi_P}\right)  & L_P &  L_P & {a} & {1} \\
\hline
5.9 &  {G} &  {G}  & \mrm{RES} & \mrm{NONRES} &  L_P & L_P &  L_P & {a} & {1} \\
\hline
5.10 & {V} &  {G}  & \mrm{SPLIT} & \mrm{NONRES} &  \prod F_P  & L_P &  L_P & {a} & {1} \\
\hline
5.11 &  {G} &  {G}  & \mrm{NONRES} & \mrm{NONRES} &  L_P/\prod F_P & L_P &  L_P & {a} & {1} \\
\hline
\end{array}\]

\captionof{table}{Patching data at points of Type $C_{11}^{Chilly}$}
\label{TableEatC11chillypoint}
\end{table}

\textbf{Cold points}: We assume that $D = D_{00} + \brac{u_P\pi_P^m, v_P\delta_P}\in \Br(F)$ where $D_{00}$ is unramified at $A_P$ and $D_P =\brac{u_P\pi_P^m, v_P\delta_P}$ where $0<m<\ell$. By Proposition \ref{propositionaatcoldpoint}, if $Y_P$ is split, then $a'_{i,P}= \brac{u_P\pi_P^m}^{sm_i}\brac{v_P\delta_P}^{n_i}{\brac{w'_{i,P}\pi_P^{-rm_i}}^{\ell}}$ where $sm = r\ell +  1$ and $w'_{i,P}\in \widehat{A_P}^*$ with ${w'_{i,P}}^{\ell}u_P^{sm_i}v_{P}^{n_i}=z_{i,P}$. Set $x_{i,P} = {\brac{w'_{i,P}\pi_P^{-rm_i}}}$. Since $\sum m_i = \sum n_i =  0$ and $\prod a'_{i,P}=1$, clearly $\prod x_{i,P}^{\ell} = 1$.

\textbf{Type $C_{11}^{Cold}$}: Let $j=1,2$. If $Y_{\eta_j}$ is of Type SPLIT, then since $P$ is a cold point lying on it, $\eta_j$ must be a C-curve. Thus it is coloured R or B. If $\eta_1$ is coloured G and $\eta_{2}$ is a C-curve, then by Proposition \ref{propgraphcolour}, $\eta_2$ will be coloured R. Similarly, if $\eta_2$ is coloured G and $\eta_1$ is a C-curve, then $\eta_1$ will be coloured R. Finally if both $\eta_1$ and $\eta_2$ are C-curves, then both of them cannot be of the same colour. Invoking Proposition \ref{propcoldexamine1}, we prescribe the choices for $E_{j,P}$ and $a_{j,P}$ in the following table:
\begin{table}[H]
\centering
\footnotesize
\arraycolsep=0.5pt
\medmuskip =0.5mu 
\[
\begin{array}{|c|c|c|c|c|c|c|c|c|c|}
\hline 

\mrm{Row} & \eta_1 & {\eta_2}  &  {Y_{\eta_1}} & {Y_{\eta_2}} & {Y_P} &  {E_{1,P}} & {E_{2,P}} & {a_{1,P}} & {a_{2,P}=aa_{1,P}^{-1}} \\
\hline
6.1 &  {G} &  {G}  & \mrm{RAM} & \mrm{RAM} &  F_P\left(\sqrt[\ell]{w_P\pi_P^r\delta_P}\right)  & H_P &  H_P & {a} & {1} \\
\hline
6.2 &  {G} &  {G}  & \mrm{RES} & \mrm{RAM} &  F_P\left(\sqrt[\ell]{v_P\delta_P}\right)  & H_P &  H_P & {a} & {1} \\
\hline
6.3 &  {G} &  {G}  & \mrm{NONRES} & \mrm{RAM} &  F_P\left(\sqrt[\ell]{w_P\delta_P}\right)  & H_P &  H_P & {a} & {1} \\
\hline
6.4 &  {G} &  {G}  & \mrm{RAM} & \mrm{RES} &  F_P\left(\sqrt[\ell]{u_P\pi_P^m}\right)  & H_P &  H_P & {a} & {1} \\
\hline
6.5 &  {R} &  {B}  & \mrm{SPLIT} & \mrm{SPLIT} &  \prod F_P  &  F_P\left(\sqrt[\ell]{u_P\pi_P^m}\right) &  F_P\left(\sqrt[\ell]{v_P\delta_P}\right) & {\brac{\brac{u_P\pi_P^m}^{sm_i}{\brac{x_{i,P}}^{\ell}}}_i} &  {\brac{\brac{v_P\delta_P}^{n_i}}_i} \\
\hline
6.6 & {B} &  {R}  & \mrm{SPLIT} & \mrm{SPLIT} &  \prod F_P  &  F_P\left(\sqrt[\ell]{v_P\delta_P}\right)  &   F_P\left(\sqrt[\ell]{u_P\pi_P^m}\right)  & {\brac{\brac{v_P\delta_P}^{n_i}}_i} & {\brac{\brac{u_P\pi_P^m}^{sm_i}{\brac{x_{i,P}}^{\ell}}}_i} \\
\hline
6.7 &  {G} &  {R}  & \mrm{NONRES} & \mrm{SPLIT} &  \prod F_P  &  F_P\left(\sqrt[\ell]{v_P\delta_P}\right) &  H_P & {a} & {1} \\
\hline
6.8 &  {G} &  {G}  & \mrm{RAM} & \mrm{NONRES} &  F_P\left(\sqrt[\ell]{w_P\pi_P}\right)  & H_P &  H_P & {a} & {1} \\
\hline
6.9 &  {R} &  {G}  & \mrm{SPLIT} & \mrm{NONRES} &  \prod F_P  &  F_P\left(\sqrt[\ell]{u_P\pi_P^m}\right)  &  H_P & {a} & {1} \\
\hline
6.10 &  {G} &  {G}  & \mrm{NONRES} & \mrm{NONRES} &  L_P/\prod F_P & H_P &  H_P & {a} & {1} \\
\hline
\end{array}\]
\captionof{table}{Patching data at points of Type $C_{11}^{Cold}$}
\label{TableEatC11coldpoint}
\end{table}
Since $D_P = \brac{u_P\pi_P^m, v_P\delta_P}$, clearly $F_P\left(\sqrt[\ell]{u_P\pi_P^m}\right)$ and $F_P\left(\sqrt[\ell]{v_P\delta_P}\right)$ split it. Since the symbol algebra $(x,y)=(x+y, -yx^{-1})$, so does $H_P$. Thus Properties 3 and 6 hold. By Lemma \ref{lemmanormfromE-nonsplit}, Property 4 holds for Rows 6.1-6.4, 6.8 (and for Row 6.10, if $Y_P=L_P$).  By construction it also holds for Rows 6.5-6.6. In Row 6.7, $a$ is a unit along $\eta_1$. Thus $a'_{i,P} = \brac{v_P\delta_P}^{n_i}w_{i,P}^{\ell}$ which is a norm from $F_P\left(\sqrt[\ell]{v_P\delta_P}\right)$ . A similar argument works for Row 6.9. In Row 6.10, if $Y_P =\prod F_P$, then since $a$ is a unit along both $\eta_1$ and $\eta_2$, we have that $a'_{i,P}=w_{i,P}^{\ell}$. So each $a'_{i,P}$ is a norm from $E_{1,P}$. Thus Property 4 holds for $C_{11}^{Cold}$ points. Finally since $\N(a)=1$ and for Rows 6.6-6.7, $\sum m_i  = \sum n_i = 0$ and $\prod x_{i,P}^{\ell} =1$, we have $\N\brac{a_{1,P}}=1$ for all rows. Hence Property 5 holds.

\textbf{Type $C_{12}^{Cold}$} : Wlog, assume $\eta_2$ is of Type 2.  Hence it is coloured G. If $Y_{\eta_1}$ is of Type SPLIT, then since $P$ is a cold point lying on it, $\eta_1$ must be a C-curve. Thus it is coloured R or B. Since $\eta_2$ is coloured G, then by Proposition \ref{propgraphcolour}, $\eta_1$ will be coloured R in this case. Invoking Proposition \ref{propcoldexamine2}, we prescribe the choices for $E_{j,P}$ and $a_{j,P}$ in the following table. The proof that Properties 1-6 hold is exactly similar to the Type  $C_{11}^{Cold}$ point case.

\begin{table}[H]
\centering
\footnotesize
\arraycolsep=0.5pt
\medmuskip =0.5mu 
\[
\begin{array}{|c|c|c|c|c|c|c|c|c|c|}
\hline 

\mrm{Row} & \eta_1 & {\eta_2}  &  {Y_{\eta_1}} & {Y_{\eta_2}} & {Y_P} &  {E_{1,P}} & {E_{2,P}} & {a_{1,P}} & {a_{2,P}=aa_{1,P}^{-1}} \\
\hline
7.1 &  {G} &  {G}  & \mrm{RAM} & \mrm{RAM} &  F_P\left(\sqrt[\ell]{w_P\pi_P\delta_P^r}\right)  & H_P &  H_P & {a} & {1} \\
\hline
7.2 &  {G} &  {G}  & \mrm{RES} & \mrm{RAM} &  F_P\left(\sqrt[\ell]{v_P\delta_P}\right)  & H_P &  H_P & {a} & {1} \\
\hline
7.3 &  {G} &  {G}  & \mrm{NONRES} & \mrm{RAM} &  F_P\left(\sqrt[\ell]{w_P\delta_P}\right)  & H_P &  H_P & {a} & {1} \\
\hline
7.4 &  {G} &  {G}  & \mrm{RAM} & \mrm{RES} &  F_P\left(\sqrt[\ell]{u_P\pi_P^m}\right)  & H_P &  H_P & {a} & {1} \\
\hline
7.5 &  {G} &  {G}  & \mrm{RAM} & \mrm{NONRES} &  F_P\left(\sqrt[\ell]{w_P\pi_P}\right)  & H_P &  H_P & {a} & {1} \\
\hline
7.6 &  {R} &  {G}  & \mrm{SPLIT} & \mrm{NONRES} &  \prod F_P  &  F_P\left(\sqrt[\ell]{u_P\pi_P^m}\right)  &  H_P & {a} & {1} \\
\hline
7.7 & {G} &  {G}  & \mrm{NONRES} & \mrm{NONRES} &  L_P/\prod F_P & H_P &  H_P & {a} & {1} \\
\hline
\end{array}\]
\captionof{table}{Patching data at points of Type $C_{12}^{Cold}$}
\label{TableEatC21coldpoint}
\end{table}

\textbf{Type $C_{12}^{Hot}$} : Wlog, assume $\eta_2$ is of Type 2 and coloured G. If $Y_{\eta_1}$ is of Type SPLIT, then since $P$ is a hot point lying on it, $\eta_1$ must be a H-curve and coloured I or O. By Proposition \ref{propositionaathotpoint}, if $Y_P$ is split, then $a'_{i,P}= z_{i,P}\pi_P^{m_i}\delta_P^{\ell n'_i}$ where $\sum n'_i=0$. Invoking the table in Proposition \ref{propblowuphotexamine}, we prescribe the choices for $E_{j,P}$ and $a_{j,P}$ in the following table:
\begin{table}[H]
\centering
\footnotesize
\arraycolsep=0.5pt
\medmuskip =0.5mu 
\[
\begin{array}{|c|c|c|c|c|c|c|c|c|c|}
\hline 

\mrm{Row} & \eta_1 & {\eta_2}  &  {Y_{\eta_1}} & {Y_{\eta_2}} & {Y_P} &  {E_{1,P}} & {E_{2,P}} & {a_{1,P}} & {a_{2,P}=aa_{1,P}^{-1}} \\
\hline
8.1 &  {G} &  {G}  & \mrm{RAM} & \mrm{RAM} &  F_P\left(\sqrt[\ell]{w_P\pi_P^r\delta_P}\right)  & L_P &  L_P & {a} & {1} \\
\hline
8.2 & {G} &  {G}  & \mrm{NONRES} & \mrm{RAM} &   F_P\left(\sqrt[\ell]{w_P\delta_P}\right) & L_P  &  L_P & {a} & {1} \\
\hline
8.3 &  {I} & {G}  & \mrm{SPLIT} & \mrm{RES} &  \prod F_P & L_P &  L_P & {a} & {1} \\
\hline
8.4 &  {O} &  {G}  & \mrm{SPLIT} & \mrm{RES} &  \prod F_P & F_P\left(\sqrt[\ell]{\pi_P}\right) &  L_P & {\brac{\pi_P^{m_i}}_i} & {\brac{z_{i,P}\delta_P^{\ell n'_i}}_i}\\
\hline
8.5 &  {G} &  {G}  & \mrm{RAM} & \mrm{NONRES} &  F_P\left(\sqrt[\ell]{w_P\pi_P}\right)  & \prod F_P &  \prod F_P & {a} & {1} \\
\hline
8.6 &  {G} &  {G}  & \mrm{RES} & \mrm{NONRES} &  L_P  & \prod F_P  &  \prod F_P & {a} & {1} \\
\hline
\end{array}\]
\captionof{table}{Patching data at points of Type $C_{12}^{Hot}$}
\label{TableEatC21hotpoint}
\end{table}
Since $D_P = \brac{u_P,\pi_P}$, clearly $F_P\left(\sqrt[\ell]{\pi_P}\right)$ and $L_P$ splits it. Thus Properties 3 and 6 hold for Rows 8.1-8.4. For Rows 8.5-8.6, we observe that $D\otimes Y_P=0$ and $E_{j,P}$ are split. Thus Properties 3 and 6 hold for all cases. By Lemma \ref{lemmanormfromE-nonsplit}, Property 4 holds for Rows 8.1-8.2 and 8.5-8.6. In Row 8.3, the colours of $\eta_1$ and $\eta_2$ imply that $a$ is a unit along $\eta_2$ and $\eta_1$ (up to $\ell^{\mrm{th}}$ powers). Hence each $a'_{i,P}=z_{i,P}\pi_P^{\ell m'_i}$ where $m_i = \ell m'_i$ is a unit in $\widehat{A_P}^*$ up to $\ell^{\mrm{th}}$ powers. Thus by Lemma \ref{lemmanormfromE-split}, Property 4 holds here.  For Row 8.4, clearly $\pi_P^{m_i}$ is a norm from $F_P\brac{\sqrt[\ell]{\pi_P}}$. Appealing again to Lemma \ref{lemmanormfromE-split}, we see that $z_{i,P}\delta_P^{\ell n'_i}$ are norms from $L_P$. Finally Property 5 holds because $\N(a)=1$ and $\sum m_i= \sum n'_i = 0$. 
\end{proof}

\subsection{Points of Type $A_{00}^s$}
\label{sectionsubtypesA00}
Let $P$ be of Type $A_{00}^{s}$. Thus $\eta_1$ and $\eta_2$ are both of Type 0. For $j=1,2$, let $\mathcal{C}_j := \brac{\overline{\eta_j}\cap S_0}\setminus \{P\}$ denote the set of marked points on $\eta_j$ apart from $P$. By Proposition \ref{propblowupB10}, it is clear that if $Q_j\in \mathcal{C}_j$ is not of Type $A_{00}^{s}$, then $\mathcal{C}_j=\{Q_j\}$. In such a case, let $\gamma_j$ denote the Type 1a/1b/2 curve such that $Q_j \in \overline{\eta_j}\cap \overline{\gamma_j}\cap S_0$. Note $\gamma_j$ can only be coloured R, G, B, Bl, or W. We subdivide Type $A_{00}^{s}$ points into three sub-types:

\textbf{D1}: $\mathcal{C}_j=\{Q_j\}$ where $Q_j$ is not of Type $A_{00}^{s}$ for $j=1,2$. \\
\textbf{D2}: $\mathcal{C}_j=\{Q_j\}$ where $Q_j$ is not of Type $A_{00}^{s}$ and $\mathcal{C}_{j'}$ is either empty or consists only of Type $A_{00}^{s}$ points for $\{j,j'\} = \{1,2\}$.\\
\textbf{D3}: $\mathcal{C}_j$ is either empty or consists only of Type $A_{00}^{s}$ points for $j=1,2$.

\begin{proposition}
\label{propositionEatpoints-A00} 
Let $P\in S_0$ be such that it is of Type $A_{00}^s$. Set $E_{1,P}=E_{2,P}=\prod F_P$. Then there exist ${a}_{1,P}$, $a_{2,P} \in Y_P$ such that for $j=1,2$,  
\begin{enumerate}
\item 
$a_{1,P}a_{2,P}= a$.
\item
$\N_{Y_P/F_P}\brac{a_{j,P}}=1$.
\item
$a_{j,P}$ is a norm from $E_{j,P}\otimes Y_P/Y_P$. 
\end{enumerate}
\end{proposition}

\begin{proof}
Note that since we have chosen the split extension for each $E_{j,P}$, Property 3 holds for any choice of $a_{j,P}$. By Remark \ref{remark-ram-split-not-intersect}, note that if $Y_{\eta_1}$ is of Type RAM, then $Y_{\eta_2}$ cannot be of Type SPLIT and vice-versa. For the same reason, if $\gamma_j$ is coloured red/blue/white/black, then $Y_{\eta_j}$ cannot be of Type RAM. Finally if $Y_{\eta_j}$ is of Type RAM, then by Proposition \ref{propaisaunitifYnotsplit} and Lemmata \ref{lemmanormoneramified-dim1} and \ref{lemmanormoneramified}, $a\in \mathcal{O}_{Y_{\eta_j}}^{*\ell}$ and $\mathcal{O}_{Y_P}^{*\ell}$. We prescribe $a_{j,P}$ as in the tables below\footnote{If $Y_P$ is not split,  set $m_i=n_i=0$ and read the entry $\brac{\pi_P^{m_i}}_i$ as $1$ and $\brac{z_{i,P}\delta_P^{n_i}}_i$ as $a$ etc.} depending on the subtype and neighbours of $P$. 

\textbf{Subtype D3}: Let $P$ be of subtype D3. Thus $\mathcal{C}_j$ is empty or consists only of Type $A_{00}^{s}$ points for $j=1,2$.
\begin{table}[H]
\centering
\footnotesize
\arraycolsep=0.5pt
\medmuskip =0.5mu 
\[
\begin{array}{|c|c|c|}
\hline 
\mrm{Row}    & {a_{1,P}} & {a_{2,P}}=aa_{1,P}^{-1} \\
\hline
11.1 &     {a} & {1} \\
\hline
\end{array}
\]
\captionof{table}{Patching data at points of subtype D3}
\label{TableEatA00spoint-D3}
\end{table}

\textbf{Subtype D2}: Let $P$ be of subtype D2. Wlog assume $\mathcal{C}_1 = \{Q_1\}$ where $Q_1$ is not of Type $A_{00}^{s}$ and that $\mathcal{C}_2$ is empty or consists only of Type $A_{00}^{s}$ points. 
\begin{table}[H]
\centering
\footnotesize
\arraycolsep=0.5pt
\medmuskip =0.5mu 
\[
\begin{array}{|c|c|c|c|c|c|c|c|}
\hline 
\mrm{Row} & Y_{\eta_1} & Y_{\eta_2} & Q_{1}  &  \gamma_1 & \mrm{Colour\ of\ } \gamma_1    & {a_{1,P}} & {a_{2,P}}=aa_{1,P}^{-1} \\
\hline
10.1 & - & - &    A_{10}^{s}   & 1a &  &   {a} & {1} \\
\hline
10.2' & - & \mrm{RAM} &   B_{10}^{s}   & 1b & {R}, W  &   {a} & {1} \\
\hline
10.2 & - & \mrm{Not\ RAM}&   B_{10}^{s}   & 1b & {R}, W  &   {\brac{\delta_P^{n_i}}_i} & {\brac{z_{i,P}\pi_P^{m_i}}_i} \\
\hline
10.3 & - & - &   B_{10}^{s}    & 1b & {G}, {B}, Bl  &  {a} & {1}\\
\hline
10.4 &  - & - &  B_{20}^{s}   & 2 & {G}  &  {a} & {1}\\
\hline
\end{array}
\]
\captionof{table}{Patching data at points of subtype D2}
\label{TableEatA00spoint-D2}
\end{table}

\textbf{Subtype D1}: Let $P$ be of subtype D1. For $j=1,2$, let $\mathcal{C}_j = \{Q_j\}$ where $Q_j$ is not of Type $A_{00}^s$. 
\begin{table}[H]
\centering
\footnotesize
\arraycolsep=0.5pt
\medmuskip =0.5mu 
\[
\begin{array}{|c|c|c|c|c|c|c|c|c|c|c|}
\hline 
\mrm{Row} & Y_{\eta_1} & Y_{\eta_2} & Q_{1} & Q_{2} &  \gamma_1 & \mrm{Colour\ of\ } \gamma_1 & \gamma_2 & \mrm{Colour\ of\ } \gamma_2  & {a_{1,P}} & {a_{2,P}}=aa_{1,P}^{-1} \\
\hline
9.1 &  -  & - & A_{10}^{s} & A_{10}^{s}   & 1a & &  1a & & {a} & {1} \\
\hline
9.2' &  \mrm{RAM} &  - & A_{10}^{s} & B_{10}^{s}   & 1a & &  1b & {R}, W & {a} & {1} \\
\hline
9.2 &   \mrm{Not \ RAM} & - & A_{10}^{s} & B_{10}^{s}   & 1a & &  1b & {R}, W & {\brac{\pi_P^{m_i}}_i} & {\brac{z_{i,P}\delta_P^{n_i}}_i} \\
\hline
9.3 & - & -  &   A_{10}^{s} & B_{10}^{s}   & 1a & &  1b & {G}, {B}, Bl & {a} & {1} \\
\hline
9.4 &  - & -  &  A_{10}^{s} & B_{20}^{s}   & 1a & &  2  & {G}  & {a} & {1} \\
\hline
\hline
9.5' & - & \mrm{RAM} &  B_{10}^{s} & A_{10}^{s}   & 1b & {R}, W &  1a & &  {a} & {1}\\
\hline
9.5 & - &\mrm{Not\ RAM} &  B_{10}^{s} & A_{10}^{s}   & 1b & {R}, W &  1a & &  {\brac{\delta_P^{n_i}}_i} & {\brac{z_{i,P}\pi_P^{m_i}}_i}\\
\hline
9.6 & - & - & B_{10}^{s} & B_{10}^{s}   & 1b & {R}, W  &  1b & {R}, W & {1} & {a} \\
\hline
9.7' & - &  \mrm{RAM} & B_{10}^{s} & B_{10}^{s}   & 1b & {R}, W &  1b & {G} & {a} & {1} \\
\hline
9.7 & -  & \mrm{Not\ RAM} & B_{10}^{s} & B_{10}^{s}   & 1b & {R}, W &  1b & {G}, {B}, Bl & {\brac{\delta_P^{n_i}}_i} & {\brac{z_{i,P}\pi_P^{m_i}}_i} \\
\hline
9.8' & - & \mrm{RAM} &   B_{10}^{s} & B_{20}^{s}   & 1b & {R}, W  &  2  & {G}  &  {a} & {1} \\
\hline
9.8 & -  & \mrm{Not\ RAM} &   B_{10}^{s} & B_{20}^{s}   & 1b & {R}, W  &  2  & {G}  &  {\brac{\delta_P^{n_i}}_i} & {\brac{z_{i,P}\pi_P^{m_i}}_i} \\
\hline
\hline
9.9 &  - & - &  B_{10}^{s} & A_{10}^{s}   & 1b & {G}, {B}, Bl   &  1a & &  {a} & {1}\\
\hline
9.10' &\mrm{RAM} & - &   B_{10}^{s} & B_{10}^{s}   & 1b & {G}  &  1b & {R}, W &  {a} & {1} \\
\hline
9.10 & \mrm{Not\ RAM} & - &   B_{10}^{s} & B_{10}^{s}   & 1b & {G}, {B}, Bl  &  1b & {R}, W &  {\brac{\pi_P^{m_i}}_i} & {\brac{z_{i,P}\delta_P^{n_i}}_i} \\
\hline
9.11 &  - & - &  B_{10}^{s} & B_{10}^{s}   & 1b & {G}, {B}, Bl  &  1b & {G}, {B}, Bl & {a} & {1} \\
\hline
9.12 &-  &- &  B_{10}^{s} & B_{20}^{s}   & 1b & {G}, {B}, Bl   &  2  & {G}  &  {a} & {1} \\
\hline
\hline
9.13 & - & - &  B_{20}^{s} & A_{10}^{s}   & 2 & {G}  &  1a & &  {a} & {1}\\
\hline
9.14' &\mrm{RAM}  & -  &  B_{20}^{s} & B_{10}^{s}   & 2 & {G}  &  1b & {R}, W &  {a} & {1} \\
\hline
9.14 &  \mrm{Not\ RAM} & - &  B_{20}^{s} & B_{10}^{s}   & 2 & {G}  &  1b & {R}, W &  {\brac{\pi_P^{m_i}}_i} & {\brac{z_{i,P}\delta_P^{n_i}}_i} \\
\hline
9.15 &  - & - & B_{20}^{s} & B_{10}^{s}   & 2 & {G} &  1b & {G}, {B}, Bl & {a} & {1} \\
\hline
9.16 & - & - &  B_{20}^{s} & B_{20}^{s}   & 2 & {G}   &  2  & {G}  &  {a} & {1} \\
\hline
\end{array}
\]
\captionof{table}{Patching data at points of subtype D1}
\label{TableEatA00spoint-D1}
\end{table} \end{proof}

\section{Structure of $E_{j,P}$ and $a_{j,P}$ along branch fields}
Let $P\in S_0$. Recall the choice of parameter, $\pi_{\eta}$, of $F_{\eta}$ for each $\eta \in N'_0$ as in Section \ref{subsectionfinalmodel}, which defines $\overline{\eta}$ at $P$ if $P\in \overline{\eta}$. In this case, $\pi_P :=\pi_{\eta}$ is part of the chosen system of parameters $(pi_P, \delta_P)$ of $A_P$. In this section, we study the ramification and splitting properties of $E_{j,P}$ and the shape of $a_{j,P}$ for $j=1,2$ with respect to the colour and type of curves on which $P$ lies. This will be useful when we construct extensions $E_{j,\eta}$ and elements $a_{j,\eta}$ for codimension one points $\eta\in N_0$.

We first begin by calculating how the lift of residues looks like along the residue fields $k_{P,\eta}$ of branch fields $F_{P,\eta}$.
\begin{lemma}
\label{lemma-liftofresidues}
 Let $P\in S_0$ lie on the intersection of two distinct irreducible curves of $\mathcal{H}_{\mathcal{X}}$ with generic points $\eta_1$ and $\eta_2$. Let $(\pi_P, \delta_P)$ be the system of parameters at $P$ chosen as in Section \ref{subsectionfinalmodel} such that $\pi_P$ cuts out $\overline{\eta_1}$ and $\delta_P$ cuts out $\overline{\eta_2}$ at $P$. Let $\eta = \eta_1$ or $\eta_2$ and let $H_{\eta}$ denote the lift of residues along $\eta$. Set $H'_P = \overline{H_{\eta}\otimes F_{P,\eta}}/k_{P,\eta}$ and $u'\in k_{P,\eta}^*/k_{P,\eta}^{*\ell}$ to be the residue of $D$ over $F_{P,\eta}$. Then the following table gives the shape of $H'_P/k_{P,\eta}$ and $u'\in k_{P,\eta}$.
\begin{table}[H]
\centering
\footnotesize
\arraycolsep=0.5pt
\medmuskip =0.5mu 
\[
\begin{array}{|c|c|c|c|c|c|c|c|}
\hline 
\mrm{Row} & \mrm{Location} & P & {D_P\in \Br(F_P)} & {(\eta, \mathrm{Type})} & u' & {H'_{P}/k_{P,\eta}} & \mrm{Description\ of\ }H'_P  \\
\hline
a. & \mrm{Table} \ \ref{TableEatB10spoint} & B_{10}^{s} & 0  & \brac{\eta_1, 1b} & 1 & \prod k_{P,\eta} & \mrm{Split} \\
\hline
b. &  \mrm{Table} \ \ref{TableEatB11nspoint} & B_{11}^{ns} & \brac{u_P,\pi_P}  & \brac{\eta_1, 1b} & \overline{u_P} & k_{P,\eta}\brac{\sqrt[\ell]{\overline{u_P}}} & \mrm{Unramified \, nonsplit} \\
\hline
c. & \mrm{Table} \ \ref{TableEatB20spoint} & B_{20}^{s} & 0  & \brac{\eta_1, 2} & 1 &\prod k_{P,\eta} & \mrm{Split} \\
\hline
d. & \mrm{Table} \ \ref{TableEatB21nspoint} & B_{21}^{ns} & \brac{u_P,\pi_P}  & \brac{\eta_1, 2}  & \overline{u_P} & k_{P,\eta}\brac{\sqrt[\ell]{\overline{u_P}}} & \mrm{Unramified \, nonsplit} \\
\hline
e. & \mrm{Table} \ \ref{TableEatC11chillypoint} & C_{11}^{Chilly} & \brac{v_P,\pi_P^j\delta_P}  & \brac{\eta_1, 1b}  & \overline{v_P^j} & k_{P,\eta}\brac{\sqrt[\ell]{\overline{v_P}}} & \mrm{Unramified \, nonsplit} \\
\hline
f. & \mrm{Table} \ \ref{TableEatC11chillypoint} & C_{11}^{Chilly} & \brac{v_P,\pi_P^j\delta_P}  & \brac{\eta_2, 1b}  & \overline{v_P} & k_{P,\eta}\brac{\sqrt[\ell]{\overline{v_P}}} & \mrm{Unramified \, nonsplit} \\
\hline
g. & \mrm{Table} \ \ref{TableEatC11coldpoint} & C_{11}^{Cold} & \brac{u_P\pi_P^m, v_P\delta_P}  & \brac{\eta_1, 1b}  & \overline{v_P^{-m}\delta_P^{-m}} &  k_{P,\eta}\brac{\sqrt[\ell]{\overline{{v_P\delta_P}}}} & \mrm{Ramified \, nonsplit} \\
\hline
h. & \mrm{Table} \ \ref{TableEatC11coldpoint} & C_{11}^{Cold} & \brac{u_P\pi_P^m, v_P\delta_P} & \brac{\eta_2, 1b}  & \overline{u_P\pi_P^m} & k_{P,\eta}\brac{\sqrt[\ell]{\overline{u_P\pi_P^m}}} & \mrm{Ramified \, nonsplit} \\
\hline
i. & \mrm{Table} \ \ref{TableEatC21coldpoint} & C_{12}^{Cold} & \brac{u_P\pi_P^m, v_P\delta_P}  & \brac{\eta_1, 1b}  & \overline{v_P^{-m}\delta_P^{-m}}  & k_{P,\eta}\brac{\sqrt[\ell]{\overline{{v_P\delta_P}}}} & \mrm{Ramified \, nonsplit} \\
\hline
j. & \mrm{Table} \ \ref{TableEatC21coldpoint} & C_{12}^{Cold} & \brac{u_P\pi_P^m, v_P\delta_P} & \brac{\eta_2, 2}  &\overline{u_P\pi_P^m} &  k_{P,\eta}\brac{\sqrt[\ell]{\overline{u_P\pi_P^m}}} & \mrm{Ramified \, nonsplit} \\
\hline
l. & \mrm{Table} \ \ref{TableEatC21hotpoint} & C_{12}^{Hot} & \brac{u_P, \pi_P}  & \brac{\eta_1, 1b}  &\overline{u_P} & k_{P,\eta}\brac{\sqrt[\ell]{\overline{u_P}}} & \mrm{Unramified \, nonsplit} \\
\hline
m. & \mrm{Table} \ \ref{TableEatC21hotpoint} & C_{12}^{Hot} & \brac{u_P, \pi_P}  & \brac{\eta_2, 2}  & 1 &  \prod k_{P,\eta}& \mrm{Split} \\
\hline

\end{array}
\]
\captionof{table}{The shape of the lift of residues}
\label{Tableliftofresidues}
\end{table}
\end{lemma}

In the following, we let $\pi_{\eta}=\pi_P$ be the prime defining $\eta$ at $P$ and let $\delta_P$ denote the other prime completing the system of parameters at $P$. We also let $L_P$ denote the the unique degree $\ell$ extension of $F_P$ unramified at $\widehat{A_P}$. 
\begin{proposition}[Violet/Indigo/Black]
\label{otherproperties-violet-indigo}
\label{otherproperties-black}
Let $\eta \in N'_0$ and $P\in \overline{\eta}\cap S_0$. Assume further that $\eta$ is coloured violet, indigo or black. Let $j=1$ or $2$ and let $E_{j,P}$ be as prescribed in Proposition \ref{propositionEatpoints}. Then
\begin{enumerate}
\item
$E_{j,P}=L_P$ if $\eta$ is coloured violet or indigo.
\item
$a_{1,P}=a$ and $a_{2,P}=1$.
\item
$E_{j,P} \otimes F_{P,\eta}$ is an unramified extension of $F_{P,\eta}$ and matches with the lift of residues at $\eta$ as etale algebras over $F_{P,\eta}$.
\item
$E_{j,P}\otimes F_{P,\eta}$ splits $D$ in $\Br\brac{F_{P,\eta}}$.
\end{enumerate} \end{proposition}

\begin{proof} An inspection of Row 1.2, 2.2, 5.5, 5.7, 5.10 and 8.3 of the tables in Proposition \ref{propositionEatpoints} immediately shows that Properties 1-4 hold (Lemma \ref{lemma-liftofresidues}). 

 \end{proof}

 \begin{proposition}[Blue]
\label{otherproperties-blue}
Let $\eta \in N'_0$ and $P\in \overline{\eta}\cap S_0$. Assume further that $\eta$ is coloured blue. Let $D_{\eta} \simeq M_{\ell}\brac{u_{\eta},w_{\eta}\pi_{\eta}}$ for units $w_{\eta}, u_{\eta}\in \widehat{A_{\eta}}^*$. Let $j=1$ or $2$ and let $E_{j,P}$ be as prescribed in Proposition \ref{propositionEatpoints}. Then 
\begin{enumerate}
\item
$E_{1,P}\otimes F_{P,\eta}$ is an unramified extension of $F_{P,\eta}$ and matches with the lift of residues at $\eta$ as etale algebras over $F_{P,\eta}$.
\item
$E_{2,P}\otimes F_{P,\eta}$ is a ramified extension of $F_{P,\eta}$.
\item
$a_{1,P}$ is a unit along $\eta$.
\item
$E_{j,P}\otimes F_{P,\eta}$ splits $D$ in $\Br\brac{F_{P,\eta}}$.
\item
There exist $w_P, x_{i,P}\in \widehat{A_P}$ for $i\leq \ell$ which are units along $\eta$ such that 
\begin{enumerate}
\item
$E_{2,P} \simeq F_P[t]/(t^{\ell}-w_P\pi_P)$ and $a_{2,P} = \brac{\brac{w_P\pi_P}^{m_{i,P}}{x_{i,P}^{\ell}}}_i$ for $m_{i,P}\in \mathbb{Z}$.
\item
$\brac{w_Pw_{\eta}^{-1}, u_{\eta}}=0\in \Br\brac{F_{P,\eta}}$.
\end{enumerate}
\end{enumerate} \end{proposition}

\begin{proof} 
Since $\eta$ is coloured blue, it has to be Type 1b C-curve. Thus $P$ can either be a point of Type $B_{10}^{s}, B_{11}^{ns}$ or $C_{11}^{Cold}$. (It cannot be a $C_{12}^{Cold}$ point because then the other curve will be of Type 2 and hence green in colour. And therefore $\eta$ would have to be red). We mention the relevant rows of the tables in Proposition \ref{propositionEatpoints} below (whence Properties 1-4 become clear) and give a proof of Property 5a \& 5b in each case. 

\textit{Row 1.3 of Table \ref{TableEatB10spoint}}: Here $w_P=1=x_{i,P}$. Since $P$ is a $B_{10}^{s}$ point, by Lemma \ref{lemma-liftofresidues}, $u_{\eta}\in F_{P,\eta}^{*\ell}$ and hence Property 5a \& 5b is satisfied.

\textit{Row 2.3 of Table \ref{TableEatB11nspoint}}: Here $w_P=1=x_{i,P}$. Since $P$ is a $B_{11}^{ns}$ point, by Lemma \ref{lemma-liftofresidues}, the lift of residues along $\eta$ matches with $L_P$ along $F_{P,\eta}$. Writing $D = D_{00} + \brac{u_P, \pi_P}\in \Br(F)$ where $D_00$ is unramified at $A_P$, we have $\brac{u_P, \pi_P} = \brac{u_{\eta}, w_{\eta}\pi_{\eta}}\in \Br\brac{F_{P,\eta}}$. 
Since $\pi_P = \pi_{\eta}$, comparing residues we have that $u_P \cong u_{\eta}$ up to $\ell^{\mrm{th}}$ powers in $F_{P,\eta}$, and hence $ \brac{u_{\eta}, w_{\eta}}=0$. As $w_P = 1$ here, Property 5b is proved.

\textit{Rows 6.5 and 6.6 of Table \ref{TableEatC11coldpoint}}: We only investigate Row 6.6 (as the proof for Row 6.5 is similar in nature). Since $P$ is a $C_{11}^{Cold}$ point, by Lemma \ref{lemma-liftofresidues}, the lift of residues along $\eta_1$ matches with $F_P\brac{\sqrt[\ell]{v_P\delta_P}}$ along $F_{P,\eta_1}$.

Unravelling the expression for $a_{2,P}$ from Row 6.6, we see it is $\brac{\brac{u_P^{s}\pi_P}^{m_i}{{w'_{i,P}}^{\ell}}}_{i}$ where $sm = r\ell + 1$ and $w'_{i,P}\in \widehat{A_P}^*$. Also $E_{2,P} = F_P\brac{\sqrt[\ell]{u_P\pi_P^m}} = F_P\brac{\sqrt[\ell]{u_P^{s}\pi_P}}$. Thus  $w_P = u_P^{s}$ and $x_{i,P}=w'_{i,P}$. Writing $D = D_{00} + \brac{u_P\pi_P^m, v_P\delta_P}\in \Br(F)$ for $D_00$ unramified at $A_P$, we have $\brac{u_P\pi_P^m, v_P\delta_P} = \brac{u_{\eta}, w_{\eta}\pi_{\eta}}$ in $\Br\brac{F_{P,\eta}}$. As $\pi_{\eta} = \pi_P$, comparing residues as before, we have that $\brac{v_P\delta_P}^{-m} \cong u_{\eta}$ up to $\ell^{\mrm{th}}$ powers in $F_{P,\eta}$. 

Hence $\brac{u_P\pi_P^m, v_P\delta_P} = \brac{w_{\eta}^{m}\pi_{\eta}^{m}, v_P\delta_P}$. This implies $\brac{u_P, v_P\delta_P} = \brac{w_{\eta}^{m}, v_P\delta_P}$. Hence $\brac{u_Pw_{\eta}^{-m}, v_P\delta_P} = 0$ and so $\brac{u_P^s w_{\eta}^{-1}, \brac{v_P\delta_P}^m} = 0$. Thus  $\brac{w_Pw_{\eta}^{-1}, u_{\eta}} = 0$.  \end{proof}

\begin{proposition}[Green(1)]
\label{otherproperties-green-NONRES}
Let $\eta \in N'_0$ be of Type 1b or 2 and let $P\in \overline{\eta}\cap S_0$. Assume further that $Y_{\eta}$ is of Type NONRES.  Let $j=1$ or $2$ and let $E_{j,P}$ be as prescribed in Proposition \ref{propositionEatpoints}. Then 
\begin{enumerate}
\item
$\eta$ is coloured green.
\item
$a_{1,P}=a$ and $a_{2,P}=1$.
\item
$E_{j,\eta}\otimes F_{P,\eta}$ is an unramified extension of $F_{P,\eta}$ and matches with the lift of residues at $\eta$ as etale algebras over $F_{P,\eta}$.
\end{enumerate}
\end{proposition}

\begin{proof} 
Property 1 is obvious (Section \ref{sectionfirstcolouring}). An inspection of Rows 1.2, 2.2, 3.1, 4.1, 5.2, 5.4, 5.7-5.11, 6.3, 6.7-6.10, 7.3, {7.5-7.7}, 8.2 and {8.5-8.6} of the Tables in Proposition \ref{propositionEatpoints} shows that Properties 2 and 3 also hold (Lemma \ref{lemma-liftofresidues}). \end{proof}

\begin{proposition}[Green(2)]
\label{otherproperties-green-notnonres}
Let $\eta \in N'_0$ and $P\in \overline{\eta}\cap S_0$. Assume further that one of the following holds:
\begin{enumerate}
  \item
 $\eta$ is of Type 1b and coloured green and $Y_{\eta}$ is not of Type NONRES,
 \item
 $\eta$ is of Type 2 (and hence coloured green) and $Y_{\eta}$ is not of Type NONRES.
 \end{enumerate}
Let $j=1$ or $2$ and let $E_{j,P}$ be as prescribed in Proposition \ref{propositionEatpoints}. Then,
\begin{enumerate}
\item
$E_{j,P}\otimes F_{P,\eta}$ is an unramified (possibly split) extension of $F_{P,\eta}$.
\item
$E_{j,P}\otimes \beta_{rbc, \eta} = 0$ where $\beta_{rbc,\eta}$ is as defined in Section \ref{section-fixingresidualbrauerclass}.
\item
If $P$ is not a hot point, $a_{1,P}=a$ and $a_{2,P}=1$.
\item
If $P$ is a hot point and $Y_P$ is not split, then $a_{1,P}=a$ and $a_{2,P}=1$.
\item
If $P$ is a hot point and $Y_P$ is split, then $a_{2,P}$ is a unit in $\prod \widehat{A_P}$ up to $\ell^{\mrm{th}}$ powers.
\end{enumerate}
 \end{proposition}

\begin{proof}
The hypothesis implies $Y_{\eta}$ is of Type RES or RAM. We investigate each case separately.  Note that Properties 1, 3-5 will be clear from inspection of the relevant rows in the tables in Proposition \ref{propositionEatpoints} (which we will mention subsequently).

\textbf{{RAM}}: Let $Y_{\eta}$ be of Type RAM. Then  $Y_{\eta} = F_{\eta}\brac{\sqrt[\ell]{w_{\eta}\pi_{\eta}}}$ for some unit $w_{\eta}\in \widehat{A_{\eta}}^*$ and $\beta_{rbc,\eta} \in \Br\brac{F_{\eta}}$ with $D\otimes F_{\eta} = \beta_{rbc,\eta} + \brac{u_{\eta}, w_{\eta}\pi_{\eta}}$. If $P$ is not a cold point write $D = D_{00} + \brac{u_P, \pi_P} + \brac{v_P, \delta_P}$ where $D_{00}$ is unramified at $A_P$ and $u_P, v_P \in A_P^*$. Thus $D\otimes F_{P,\eta} = \brac{u_P, \pi_{\eta}} + \brac{v_P, \delta_P} $. 
Hence in $\Br\brac{F_{P,\eta}}$, we have\begin{align*}
& \brac{u_P, \pi_{\eta}} + \brac{v_P, \delta_P} = \beta_{rbc,\eta}  +  \brac{u_{\eta}, w_{\eta}\pi_{\eta}}.\\
& \implies \beta_{rbc,\eta}  =  \brac{v_P, \delta_P}  + \brac{u_P, \pi_{\eta}} - \brac{u_{\eta}, w_{\eta}\pi_{\eta}} =  \brac{v_P, \delta_P} + \brac{w_{\eta}, u_{\eta}} + \brac{u_Pu_{\eta}^{-1},\pi_{\eta}}
\end{align*}

Comparing residues, we get $\beta_{rbc,\eta}\otimes F_{P,\eta}  = \brac{\brac{v_P,\delta_P} + \brac{w_{\eta}, u_{\eta}}} \otimes F_{P,\eta}$, where $\overline{u_{\eta}}$ is the residue of $D\otimes F_{\eta}$. We investigate the relevant rows of the Tables in Proposition \ref{propositionEatpoints}. 

\textbf{$\eta$ of Type 1b}:\\ 
\textit{Row 1.2 of Table \ref{TableEatB10spoint}}: Here $(v_P,\delta_P)=0 \in \Br(F)$ itself. Computing residues along $\eta$, we see that $\overline{u_{\eta}}\in k_{P,\eta}$ is an $\ell^{\mrm{th}}$ power. So $\beta_{rbc,\eta}$ is already split over $F_{P,\eta}$.

\textit{Row $2.2^*$ of Table \ref{TableEatB11nspoint}}: Here $(v_P,\delta_P)=0 \in \Br(F)$ itself. Computing residues along $\eta$, we see that $\overline{u_{\eta}} \cong \overline{u_P} \in k_{P,\eta}$ upto $\ell^{\mrm{th}}$ powers. Also since $Y_{\eta_2}$ is of Type NONRES, $\overline{w_{\eta}}\in \mathcal{O}_{k_{P,\eta}}^*$. So $\beta_{rbc,\eta}$ is already split over $F_{P,\eta}$.

\textit{Row 2.2 of Table \ref{TableEatB11nspoint}}: Here $(v_P,\delta_P)=0 \in \Br(F)$ itself. Computing residues along $\eta$, we see that $\overline{u_{\eta}} \cong \overline{u_P} \in k_{P,\eta}$ upto $\ell^{\mrm{th}}$ powers. Thus the choice of $E_{j,P} = L_P$ splits $\beta_{rbc,\eta}$ over $F_{P,\eta}$.

\textit{Rows 5.1, {5.2} and 5.8 of Table \ref{TableEatC11chillypoint}}: Since $P$ is a chilly point, we have $\overline{u_P}= \overline{v_P}^j$. Here $E_{j,P} = L_P$. Note that $(v_P, \delta_P)$ is split by $L_P$ as $v_P$ becomes an $\ell^{\mrm{th}}$ power in $L_P$. Computing residues along $\eta$, we see $\overline{u_{\eta}} \cong \overline{v_P}^{j} \in k_{P,\eta}$ upto $\ell^{\mrm{th}}$ powers. Thus $L_P$ splits $\brac{u_{\eta}, w_{\eta}}$ over $F_{P,\eta}$ also.

\textit{Row 8.1 of Table \ref{TableEatC21hotpoint}}: Since $P$ is a hot point and $\eta$ is of Type $1b$, $(v_P,\delta_P) = 0 \in \Br(F_P)$ as $v_P\in \widehat{A_P}^{*\ell}$. Computing residues along $\eta$, we see $\overline{u_{\eta}} \cong \overline{u_P} \in k_{P,\eta}$ upto $\ell^{\mrm{th}}$ powers. Here $E_{j,P} = L_P$ which splits $\brac{w_{\eta}, u_{\eta}} = \beta_{rbc,\eta}$ over $F_{P,\eta}$.

\textit{Row 8.5 of Table \ref{TableEatC21hotpoint}}: Since $P$ is a hot point and $\eta$ is of Type 1b, $(v_P,\delta_P) = 0 \in \Br(F_P)$ as $v_P\in \widehat{A_P}^{*\ell}$. Computing residues along $\eta$, we see $\overline{u_{\eta}} \cong \overline{u_P} \in k_{P,\eta}$ upto $\ell^{\mrm{th}}$ powers. Since $Y_{\eta_2}$ has to be of Type NONRES in this configuration, we see that $\overline{w_{\eta}}$ has to have valuation $\cong 0 \mod \ell$ in the complete discretely valued field $k_{P,\eta}$ with parameter $\overline{\delta_P}$.

Putting this together, we see that $\overline{\brac{u_{\eta}, w_{\eta}}}$ is unramified over local field $k_{P,\eta}$, hence trivial. Therefore $\brac{u_{\eta}, w_{\eta}}$ is trivial over $F_{P,\eta}$. So $\beta_{rbc,\eta}$ is already split over $F_{P,\eta}$.

\textbf{$\eta$ of Type 2}: \\
\textit{Row 3.1 of Table \ref{TableEatB20spoint}}: Here $(v_P,\delta_P)=0 \in \Br(F)$ itself. Computing residues along $\eta$, we see that $\overline{u_{\eta}}\in k_{P,\eta}$ is an $\ell^{\mrm{th}}$ power. So $\beta_{rbc,\eta}$ is already split over $F_{P,\eta}$.

\textit{Row $4.1^*$ of Table \ref{TableEatB21nspoint}}: Here $(v_P,\delta_P)=0 \in \Br(F)$ itself. Computing residues along $\eta$, we see that $\overline{u_{\eta}} \cong \overline{u_P} \in k_{P,\eta}$ upto $\ell^{\mrm{th}}$ powers. Also since $Y_{\eta_2}$ is of Type NONRES, $\overline{w_{\eta}}\in \mathcal{O}_{k_{P,\eta}}^*$. So $\beta_{rbc,\eta}$ is already split over $F_{P,\eta}$.

\textit{Row 4.1 of Table \ref{TableEatB21nspoint}}: Here $(v_P,\delta_P)=0 \in \Br(F)$ itself. Computing residues along $\eta$, we see that $\overline{u_{\eta}} \cong \overline{u_P} \in k_{P,\eta}$ upto $\ell^{\mrm{th}}$ powers. Thus the choice of $E_{j,P} = L_P$ splits $\beta_{rbc,\eta}$ over $F_{P,\eta}$.

\textit{Rows {8.1-8.2} of Table \ref{TableEatC21hotpoint}}: Since $P$ is a hot point and $\eta$ is of Type $2$, computing residues along $\eta$, we see $\overline{u_{\eta}}  \in k_{P,\eta}^{*\ell}$ and therefore $\beta_{rbc,\eta} = \brac{v_P,\delta_P}$. Since $E_{j,P} = L_P$, it splits $\beta_{rbc,\eta}$ over $F_{P,\eta}$.

Now let $P$ be a cold point and write $D = D_{00} + \brac{u_P\pi_P^m, v_P\delta_P}$ where $D_{00}$ is unramified at $A_P$ and $u_P, v_P \in A_P^*$. Thus $D\otimes F_{P,\eta} = \brac{u_P\pi_{\eta}^m, v_P\delta_P} $. 

Hence in $\Br\brac{F_{P,\eta}}$, we have
\begin{align*}
& \brac{u_P\pi_{\eta}^m, v_P\delta_P} = \beta_{rbc,\eta}  +  \brac{u_{\eta}, w_{\eta}\pi_{\eta}} \\
& \implies \beta_{rbc,\eta}  = \brac{u_P\pi_{\eta}^m, v_P\delta_P}  - \brac{u_{\eta}, w_{\eta}\pi_{\eta}}  = \brac{u_P, v_P\delta_P} + \brac{w_{\eta}, u_{\eta}} +  \brac{\pi_{\eta}, v_P^m\delta_P^mu_{\eta}} 
\end{align*}

Comparing residues, we get $u_{\eta} \cong v_P^{-m}\delta_P^{-m}$ up to $\ell^{\mrm{th}}$ powers and $\beta_{rbc,\eta}\otimes F_{P,\eta}$ equals $\brac{\brac{u_P, v_P\delta_P} + \brac{w_{\eta}, u_{\eta}}} \otimes F_{P,\eta}$ which is $\brac{u_Pw_{\eta}^{-m}, v_P\delta_P} \otimes F_{P,\eta}$ in $\Br\brac{F_{P,\eta}}$.


\textit{Rows 6.1-6.4 and 6.8 of Table \ref{TableEatC11coldpoint} and Rows {7.1-7.5} of Table \ref{TableEatC21coldpoint} are relevant here}. In each case, $E_{j,P} = F_P\brac{\sqrt[\ell]{u_P\pi_P^m + v_P\delta_P}}$ . Thus $E_{j,P}\otimes F_{P,\eta} =  F_{P,\eta}\brac{\sqrt[\ell]{v_P\delta_P}}$ is unramified and clearly splits $\beta_{rbc,\eta}$ over $F_{P,\eta}$. 

\textbf{{RES}}: Let $Y_{\eta}$ be of Type RES. Then $Y_{\eta} = F_{\eta}\brac{\sqrt[\ell]{u_{\eta}}}$ and $\beta_{rbc,\eta}\in \Br\brac{F_{\eta}}$ with $D\otimes F_{\eta} = \beta_{rbc,\eta} + \brac{u_{\eta}, \pi_{\eta}}$. If $P$ is not a cold point, write $D = D_{00} + \brac{u_P, \pi_P} + \brac{v_P, \delta_P}\in Br(F)$ where $D_{00}$ is unramified at $A_P$ and $u_P, v_P \in A_P^*$. Thus $D\otimes F_{P,\eta} = \brac{u_P, \pi_{\eta}} + \brac{v_P, \delta_P} $. Hence in $\Br\brac{F_{P,\eta}}$, we have $\brac{u_P, \pi_{\eta}} + \brac{v_P, \delta_P} = \beta_{rbc,\eta}  +  \brac{u_{\eta}, \pi_{\eta}}$ which implies $\beta_{rbc,\eta} - \brac{v_P, \delta_P} =  \brac{u_{\eta}^{-1}u_P, \pi_{\eta}}$. Comparing residues, we see that $u_{\eta}\cong u_P$ up to $\ell^{\mrm{th}}$ powers and $\beta_{rbc,\eta}\otimes F_{P,\eta}  = \brac{v_P,\delta_P}\otimes F_{P,\eta}$.

If $P$ is a cold point, $D = D_{00} + \brac{u_P\pi_{P}^m, v_P\delta_P}\in \Br(F)$. Following a similar argument, we get $\beta_{rbc,\eta}$ equals  $\brac{u_P\pi_{\eta}^m, v_P\delta_P} - \brac{u_{\eta}, \pi_{\eta}}$ which equals $\brac{u_P,  v_P\delta_P} + \brac{\pi_{\eta} , v_P^m\delta_P^mu_{\eta}}$ in $\Br\brac{F_{P,\eta}}$. Comparing residues, we see $\beta_{rbc,\eta} \otimes F_{P,\eta }=  \brac{u_P,  v_P\delta_P} \otimes F_{P,\eta }$. We investigate the relevant rows of the Tables in Proposition \ref{propositionEatpoints}. 

\textbf{$\eta$ of Type 1b}:\\
\textit{Row 1.2 of Table \ref{TableEatB10spoint}}: Here $(v_P,\delta_P)=0 \in \Br(F)$ itself.\\
\textit{Row 2.2 of Table \ref{TableEatB11nspoint}}:  Here $(v_P,\delta_P)=0 \in \Br(F)$ itself.\\
\textit{Rows 5.3-5.4 and 5.9 of Table \ref{TableEatC11chillypoint}}: Since $P$ is a chilly point, we have $\overline{u_P}= \overline{v_P}^j$. In any case $(v_P, \delta_P)$ is split by $L_P$ as $v_P$ becomes an $\ell^{\mrm{th}}$ power in $L_P$.\\
\textit{Rows 6.2 and 6.4 of Table \ref{TableEatC11coldpoint} and Row 7.2 of Table \ref{TableEatC21coldpoint}}: Since $P$ is a cold point, we are interested in splitting $\brac{u_P, v_P\delta_P}$. Here $E_{j,P} $ is $F_P\brac{\sqrt[\ell]{u_P\pi_P^m + v_P\delta_P}}$ and hence $E_{j,P} \otimes F_{P,\eta} = F_{P,\eta}\brac{\sqrt[\ell]{v_P\delta_P}}$ which clearly splits $\beta_{rbc,\eta}$.\\
\textit{Row 8.6 of Table \ref{TableEatC21hotpoint}}: Since $P$ is a hot point and $\eta$ is of Type 1b, $\beta_{rbc,\eta} = (v_P,\delta_P)$ where $v_P\in \widehat{A_P}^{*\ell}$. Thus $\beta_{rbc,\eta}$ is already split over $F_{P,\eta}$.\\

\textbf{$\eta$ of Type 2}:\\
\textit{Row 3.1 of Table \ref{TableEatB20spoint}}: Here $(v_P,\delta_P)=0 \in \Br(F)$ itself.\\
\textit{Row 4.1 of Table \ref{TableEatB21nspoint}}:  Here $(v_P,\delta_P)=0 \in \Br(F)$ itself.\\
\textit{Row 7.4 of Table \ref{TableEatC21coldpoint}}: Since $P$ is a cold point, we are interested in splitting $\brac{u_P, v_P\delta_P}$. Here $E_{j,P} $ is $F_P\brac{\sqrt[\ell]{u_P\pi_P^m + v_P\delta_P}}$ and hence $E_{j,P} \otimes F_{P,\eta} = F_{P,\eta}\brac{\sqrt[\ell]{v_P\delta_P}}$ which clearly splits $\beta_{rbc,\eta}$.\\
\textit{Rows {8.3- 8.4} of Table \ref{TableEatC21hotpoint}}:  Since $P$ is a hot point and $\eta$ is of Type $2$, $\beta_{rbc,\eta} = (v_P,\delta_P)$ where $v_P\in \widehat{A_P}^*$ but not an $\ell^{th}$ power.  Here $E_{j,P} $ is either $L_P$ or $F_P\brac{\sqrt[\ell]{\delta_P}}$. In either case, it splits $\beta_{rbc,\eta}$ over $F_{P,\eta}$.  \end{proof}

\begin{proposition}[Yellow/Orange/Red/White]
\label{otherproperties-yellow-orange}
\label{otherproperties-red}
\label{otherproperties-white}
Let $\eta \in N'_0$ and $P\in \overline{\eta}\cap S_0$. Assume further that $\eta$ is coloured yellow, orange, red or white. Let $D\simeq M_{\ell}\brac{u_{\eta},w_{\eta}\pi_{\eta}}$ for units $w_{\eta}, u_{\eta}\in \widehat{A_{\eta}}^*$.  Let $j=1$ or $2$ and let $E_{j,P}$ be as prescribed in Proposition \ref{propositionEatpoints}. Then,
\begin{enumerate}
\item
$E_{1,P}\otimes F_{P,\eta}$ is a ramified extension of $F_{P,\eta}$.
\item
$E_{2,P} = L_P$ if $\eta$ is coloured yellow or orange. 
\item
$a_{2,P}$ is a unit along $\eta$.
\item
$E_{2,P}\otimes F_{P,\eta}$ is an unramified extension of $F_{P,\eta}$ and matches with the lift of residues at $\eta$ as etale algebras over $F_{P,\eta}$.
\item
$E_{j,P}\otimes F_{P,\eta}$ splits $D$ in $\Br\brac{F_{P,\eta}}$.
\item
There exist $w_P, x_{i,P}\in \widehat{A_P}$ for $i\leq \ell$ which are units along $\eta$ such that 
\begin{enumerate}
\item
$E_{1,P} \simeq F_P[t]/(t^{\ell}-w_P\pi_P)$ and $a_{1,P} = \brac{\brac{w_P\pi_P}^{m_{i,P}}{x_{i,P}^{\ell}}}_i$ for $m_{i,P}\in \mathbb{Z}$.
\item
$\brac{w_Pw_{\eta}^{-1}, u_{\eta}}=0\in \Br\brac{F_{P,\eta}}$.
\end{enumerate}
\end{enumerate}
\end{proposition}

\begin{proof} We will mention the relevant rows of the tables in Proposition \ref{propositionEatpoints} below, whence Properties 1-5 become clear (Lemma \ref{lemma-liftofresidues} for Property 4). We will give a proof of Property 6a \& 6b in each case.

\textit{Row 1.1 of Table \ref{TableEatB10spoint}}: $\eta$ is coloured red/white. Here $w_P=1=x_{i,P}$. The proof is similar to that of Proposition \ref{otherproperties-blue} 5(b) for Row 1.3. 

\textit{Row 2.1 of Table \ref{TableEatB11nspoint}}: $\eta$ is coloured yellow/orange/red/white. Here $w_P=1=x_{i,P}$. Since $P$ is a $B_{11}^{ns}$ point, $D_P = \brac{u_P, \pi_P}$ is nonsplit. The proof is similar to that of Proposition \ref{otherproperties-blue} 5(b) for Row 2.3. 
 
\textit{Row 5.6 of Table \ref{TableEatC11chillypoint}}: $\eta$ is coloured yellow. Unravelling the expression for $a_{1,P}$ from Row 5.6, we see it is $\brac{\pi_P^{m_i}\delta_P^{n_i}}_i$ where $m_i = r_i\ell + jn_i$. Let $sj \cong 1 \mod \ell$. Thus $n_i = r'_i\ell + sm_i$ for some $r'_i$. Hence $a_{1,P} = \brac{\brac{\pi_P\delta_P^s}^{m_i}{\delta_P^{\ell r'_i}}}_i$. As $E_{1,P} = F_P\brac{\sqrt[\ell]{\pi_P^j\delta_P}} = F_P\brac{\sqrt[\ell]{\pi_P\delta_P^s}} $, we have  $w_P = \delta_P^s$ and $x_{i,P} = \delta_P^{r'_i}$ here. Writing $D = D_{00} + \brac{u_P, \pi_P} + \brac{v_P, \delta_P}\in \Br(F)$ for $D_{00}$ unramified at $A_P$, we have $\brac{v_P, \pi_P^j\delta_P} = \brac{u_{\eta}, w_{\eta}\pi_{\eta}}$  in $\Br\brac{F_{P,\eta}}$.  Since $\pi_P = \pi_{\eta}$, comparing residues we have $v_P^j \cong u_{\eta}$ up to $\ell^{\mrm{th}}$ powers, and hence $ \brac{u_{\eta}, w_{\eta}}= \brac{v_P, \delta_P}$. Since $sj \cong 1 \mod \ell$, we have $v_P \cong u_{\eta}^s$ up to $\ell^{\mrm{th}}$ powers. Therefore $\brac{u_{\eta}, w_{\eta}} = \brac{u_{\eta}, \delta_P^s}$ which implies $\brac{u_{\eta}, \delta_P^sw_{\eta}^{-1}} = 0$. Hence Properties 6a \& 6b hold.

\textit{Rows 6.5, {6.6}, {6.7} and 6.9 of Table \ref{TableEatC11coldpoint} and Row 7.6 of Table \ref{TableEatC21coldpoint}}: $\eta$ is coloured red and $D_P = \brac{u_P\pi_P^m, v_P\delta_P}$. The proof is similar to that of Proposition \ref{otherproperties-blue} 5(b) of Rows 6.5-6.6. 

\textit{Row 8.4 of Table \ref{TableEatC21hotpoint}}: $\eta$ is coloured orange. Here $w_P=1=x_{i,P}$. Writing $D = D_{00} + \brac{u_P, \pi_P} + \brac{v_P, \delta_P} \in \Br(F)$ for $D_{00}$ unramified at $A_P$, we have $\brac{u_P, \pi_P} = \brac{u_{\eta}, w_{\eta}\pi_{\eta}}$ in $\Br\brac{F_{P,\eta}}$. Since $\pi_P = \pi_{\eta}$, comparing residues we have that $u_P \cong u_{\eta}$ up to $\ell^{\mrm{th}}$ powers, and hence $ \brac{u_{\eta}, w_{\eta}}=0$. Hence Properties 6a \& 6b hold. \end{proof}

 \begin{proposition}[0/1a]
\label{otherproperties-1a/0}
Let $\eta \in N'_0$ be of Type 1a or 0 and $P\in \overline{\eta}\cap S_0$.  Let $j=1$ or $2$ and let $E_{j,P}$ be as prescribed in Propositions \ref{propositionEatpoints} and \ref{propositionEatpoints-A00}. Then
\begin{enumerate}
\item
$E_{j,P}\otimes F_{P,\eta}$ is an unramified (possibly split) extension of $F_{P,\eta}$.
\item
If $Y_{\eta}$ is of Type RAM, then $a_{1,P}=a$ and $a_{2,P}=1$.
\end{enumerate}

Further if $\eta$ is of Type 0, then there exist $j, j'$ such that $\{j,j'\} = \{1,2\}$ such that for every $P\in \overline{\eta}\cap S_0$, the element $a_{j',P}$ is a unit in $\mathcal{O}_{Y\otimes F_{P,\eta}}$ and $E_{j,P}\simeq \prod F_P$.

\end{proposition}

\begin{proof} By an inspection of Tables \ref{TableEatA10spoint}, \ref{TableEatB10spoint}, \ref{TableEatB11nspoint}, \ref{TableEatB20spoint}, \ref{TableEatB21nspoint} and the choice of $E_{j,P}$ in Proposition \ref{propositionEatpoints-A00}, it is clear that they are unramified along $\eta$. If $Y_{\eta}$ is of Type RAM, by Remark \ref{remark-ram-split-not-intersect} it cannot intersect $\eta'\in N'_0$ where $Y_{\eta'}$ is of Type SPLIT. Hence it cannot intersect $\eta'\in N'_0$ which is coloured V, I, B, Ye, O, R, W or Bl. Inspecting Tables \ref{TableEatA10spoint}, \ref{TableEatB10spoint}, \ref{TableEatB11nspoint}, \ref{TableEatB20spoint}, \ref{TableEatB21nspoint}, \ref{TableEatA00spoint-D3}, \ref{TableEatA00spoint-D2} and \ref{TableEatA00spoint-D1} shows that in this case, $a_{1,P}=a$ and $a_{2,P}=1$.

Now let us assume $\eta$ is of Type 0 and let $\mathcal{P}'_{\eta} = \overline{\eta}\cap S_0$. Then one of the following holds (Proposition \ref{propblowupB10} and proof of Proposition \ref{propositionEatpoints-A00}): Case A) $\mathcal{P}'_{\eta} = \{Q\}$ where $Q$ is not of Type $A_{00}^{s}$, Case B) $\mathcal{P}'_{\eta} = \{Q, Q'\}$ where $Q$ is not of Type $A_{00}^{s}$ and $Q'$ is of Type $A_{00}^{s}$. Case C) $\mathcal{P}'_{\eta}$ consists of only Type $A_{00}^{s}$ points.

\textit{Case A/B}: Let $Q\in \overline{\eta}\cap \overline{\eta'}$ (and $Q'\in \overline{\eta}\cap \overline{\gamma}$ in Case B). Note that $Q$ can be of Type $A_{10}^{s}, B_{10}^{s}$ or $B_{20}^{s}$. Thus $\eta'$ can be Type 1a or coloured red, green, blue, white or black (and $\gamma$ is of Type 0 in Case B while $Q'$ has to be of subtype D1 or D2 as defined in Section \ref{sectionsubtypesA00}).

Set $j=1$ and $j'=2$ if $\eta'$ is of Type 1a or coloured green, blue or black. Set $j=2$ and $j'=1$ if $\eta'$ is coloured red or white. An inspection of Tables \ref{TableEatA10spoint}, \ref{TableEatB10spoint} and \ref{TableEatB20spoint} for Case A and Tables \ref{TableEatA00spoint-D2} and \ref{TableEatA00spoint-D1} for Case B\footnote{Appeal to Lemma \ref{lemmanormoneramified} and Proposition \ref{propaisaunitifYnotsplit} for Rows 9.5', 9.7', 9.8' and 10.2'} verifies that our choice of $j$ and $j'$ is compatible.

\textit{Case C}: In this case each $P_i\in \mathcal{P}'_{\eta}$ is of Type $A_{00}^{s}$. Thus it has to be of subtype D2 or D3. Set $j=1$ and $j'=2$. An inspection of Tables \ref{TableEatA00spoint-D3} and \ref{TableEatA00spoint-D2} verifies the compatibility of this choice. \end{proof}

\section{Understanding $E_{j,P}$ in terms of norms from some extensions}
In this section, we continue to assume that $\eta\in N'_0$ with $P\in S_0\cap \overline{\eta}$ being the intersection of two distinct irreducible curves with generic points $\eta_1$ and $\eta_2$. Let $\pi_P, \delta_P, \pi_{\eta}$ be as before. We study $E_{j,P}$ (as prescribed in Propositions \ref{propositionEatpoints} and \ref{propositionEatpoints-A00}) vis-\'a-vis norms from some related extensions. 
\subsection{When $\eta$ is 1b or 2 and $Y_{\eta}$ is RAM}
\label{section-GaloisclosureRAM}
Let $\eta = \eta_1$ or $\eta_2$ be of Type 1b or 2 with $Y_{\eta}$ of Type RAM. By Proposition \ref{otherproperties-green-notnonres}, $E_{j,P}$ is unramified along $\eta$ for $j=1,2$ and splits $\beta_{rbc,\eta}$ over $F_{P,\eta}$. Let $Y_{\eta} = F_{\eta}\brac{\sqrt[\ell]{w_{\eta}\pi_{\eta}}}$ where $w_{\eta}\in \widehat{A_{\eta}}^*$. Thus $D = \beta_{rbc,\eta} + \brac{u_{\eta}, w_{\eta}\pi_{\eta}}$ where $u'=\overline{u_{\eta}} \in k_{\eta}/k_{\eta}^{*\ell}$ is the residue of $D_{\eta}$. In this subsection, we show that $u'$ is locally a norm from $\overline{E_{j,P}}$. This will be useful in the final part of this paper (Section \ref{section-f}) where we show that the constructed $E_j$s are \textit{good}. 

 \begin{proposition}
\label{prop-points-normsfromGaloisclosures-RAM} 
Let $\eta$, $u'$ and $P$ be as above. Then there exist $w_{1,P}, w_{2,P}\in k_{P,\eta}$  such that for $j=1$ or $2$, 
\begin{enumerate}
\item
 $\overline{E_{j,P}\otimes F_{P,\eta}}=k_{P,\eta}[t]/(t^{\ell}- w_{j,P})$
 \item
 $\brac{w_{j,P}, u'}=0\in \Br\brac{k_{P,\eta}}$.
 \end{enumerate}
\end{proposition}

\begin{proof}
Since $\eta$ is Type 1b/2 and $Y_{\eta}$ is RAM, $\eta$ is coloured green. We investigate the relevant rows in the tables given in the proof of Proposition \ref{propositionEatpoints}.

For the following situations, choose $w_{1,P} = w_{2,P} =1$.

- {$B_{10}^{s}$ point, $\eta$ of Type 1b} (cf. Row 1.2 of Table \ref{TableEatB10spoint} and Row a in Table \ref{Tableliftofresidues}) \\
- some {$B_{11}^{ns}$ points(*), $\eta$ of Type 1b} (cf. Row $2.2^*$ of Table \ref{TableEatB11nspoint} and Row b in Table \ref{Tableliftofresidues})\\
- {$B_{20}^{s}$ point, $\eta$ of Type 2}  (cf. Row 3.1 of Table \ref{TableEatB20spoint} and Row c in Table \ref{Tableliftofresidues}) \\
- some {$B_{21}^{ns}$ points(*), $\eta$ of Type 2} (cf. Row $4.1^*$ of Table \ref{TableEatB21nspoint} and Row d in Table \ref{Tableliftofresidues}) \\
- {$C_{12}^{Hot}$ point, $\eta$ of Type 1b} (cf. Row 8.5 of Table \ref{TableEatC21hotpoint} and Row l in Table \ref{Tableliftofresidues}).

For the following situations, choose $w_{1,P} = w_{2,P} = u'$.

- {$B_{11}^{ns}$ point, $\eta$ of Type 1b}  : (cf. Row 2.2 of Table \ref{TableEatB11nspoint} and Row b in Table \ref{Tableliftofresidues})\\
- {$B_{21}^{ns}$ point, $\eta$ of Type 2}  : (cf. Row 4.1 of Table \ref{TableEatB21nspoint} and Row d in Table \ref{Tableliftofresidues})\\
- {$C_{11}^{Chilly}$ point, $\eta$ of Type 1b} : (cf. Rows 5.1, 5.2 and 5.8 of Table \ref{TableEatC11chillypoint} and Rows e,f in Table \ref{Tableliftofresidues})\\
- {$C_{11}^{Cold}$ point, $\eta$ of Type 1b} : (cf. Rows 6.1-6.4 and 6.8 of Table \ref{TableEatC11coldpoint} and Rows g,h in Table \ref{Tableliftofresidues}) \\
- {$C_{12}^{Cold}$ point, $\eta$ of Type 1b} : (cf. Rows 7.1, 7.4 and 7.5 of Table \ref{TableEatC21coldpoint} and Row i in Table \ref{Tableliftofresidues}) \\
- {$C_{12}^{Cold}$ point, $\eta$ of Type 2} : (cf. Rows 7.1-7.3 of Table \ref{TableEatC21coldpoint} and Row j in Table \ref{Tableliftofresidues}) \\
- {$C_{12}^{Hot}$ point, $\eta$ of Type 1b} : (cf. Row 8.1 of Table \ref{TableEatC21hotpoint} and Row l in Table \ref{Tableliftofresidues}).

For {$C_{12}^{Hot}$ points, $\eta$ of Type 2} (cf. Rows 8.1-8.2 of Table \ref{TableEatC21hotpoint} and Row m in Table \ref{Tableliftofresidues}), choose $w_{1,P} = w_{2,P} \in  \mathcal{O}_{k_{P,\eta}}^*\setminus \mathcal{O}_{k_{P,\eta}}^{*\ell}$. Since $u'=1$ here, $\brac{w_{j,P},u'}=0$. \end{proof}

\subsection{When $\eta$ is 1b or 2 and $Y_{\eta}$ is RES}
\label{section-GaloisclosureRES}
Let $\eta=\eta_1$ or $\eta_2$ be of Type 1b or 2 with $Y_{\eta}$ of Type RES. In this subsection, we will define certain extensions  $\tL_P$ and $\tL'_P$ of $k_{P,\eta}$ and understand $E_{j,P}$ in terms of norms from these extensions. This will be helpful in constructing $E_{j,\eta}$ over $F_{\eta}$ by approximating local data.

Since $Y_{\eta}$ is of Type RES, we have $Y_{\eta}\simeq F_{\eta}\brac{\sqrt[\ell]{u_{\eta}}}$ where $u'=\overline{u_{\eta}} \in k_{\eta}/k_{\eta}^{*\ell}$ is the residue of $D\otimes F_{\eta}$. Recall that $\Gal\brac{Y_{\eta}/F_{\eta}} = \langle \psi\rangle$. Let $Y'$ be $\overline{Y\otimes F_{\eta}}$ over $k_{\eta}$ and by abuse of notation, let $\Gal\brac{Y'/k_{\eta}} = \langle \psi \rangle$ also. Finally let $Y'_P$ be $\overline{Y\otimes F_{P,\eta}}$ over $k_{P,\eta}$ with an induced action of $\psi$.

Note that if $Y'_P$ is split, then $Y'_P \simeq \prod k_{P,\eta}$ where $x\in Y'$ is identified with the tuple $\brac{x, \psi(x), \ldots, \psi^{\ell-1}(x)}$. Note that $\psi$ acts on $\prod k_{P,\eta}$ by permutations. That is, for $x_i\in k_{P,\eta}$, \[\psi(x_1, x_2, \ldots, x_{\ell}) = (x_2, x_3, \ldots, x_{\ell}, x_1).\]

Let $a_{j,P}$ and $E_{j,P}$ be as prescribed in Propositions \ref{propositionEatpoints} and \ref{propositionEatpoints-A00}. By Proposition \ref{otherproperties-green-notnonres}, $E_{j,P}$ is unramified along $\eta$ for $j=1,2$ and splits $\beta_{rbc,\eta}$ over $F_{P,\eta}$. Also $a_{j,P}$ are units along $\eta$ (Proposition \ref{propaisaunitifYnotsplit}). Set $b_P := \overline{a_{1,P}}, b'_P := \overline{a_{2,P}}$ in $Y'_P$. If $Y'_P$ is split, then set $b_P = \brac{b_{i,P}}_i \in \prod k_{P,\eta}$ and $b'_P = \brac{b'_{i,P}}_i \in \prod k_{P,\eta}$.

When $Y'_P$ is not split, set 
\begin{align*}
\tL_P &= Y'_P \brac{\sqrt[\ell]{b_P}, \sqrt[\ell]{\psi(b_P)}, \ldots, \sqrt[\ell]{\psi^{\ell-1}(b_P)}}\\
\tL'_P &= Y'_P \brac{\sqrt[\ell]{b'_P}, \sqrt[\ell]{\psi(b'_P)}, \ldots, \sqrt[\ell]{\psi^{\ell-1}(b'_P)}}.
\end{align*}

Since $Y'_P$ is a nonsplit extension of $k_{P,\eta}$ and $\N\brac{b_P}=1= \N\brac{b'_P}$, first of all $b_P$ and $b'_P$ are units in $\mathcal{O}_{Y',P}$. Also by Lemmata \ref{lemmanormoneramified-dim1} and \ref{lemmanormoneunramified}, we have $b_P, b'_P\in {Y'_P}^{*\ell}$. Hence $\psi^j\brac{b_P}, \psi^j\brac{b'_P}$ are all $\ell^{\mrm{th}}$ powers in $Y'_P$ also. Therefore $\tL_P = Y'_P =\tL'_P$.

When $Y'_P$ is split, set 
\begin{align*}
\tL_P &= k_{P,\eta} \brac{\sqrt[\ell]{b_{1,P}},\sqrt[\ell]{b_{2,P}}, \ldots, \sqrt[\ell]{b_{\ell,P}} }\\
\tL'_P &=k_{P,\eta} \brac{\sqrt[\ell]{b'_{1,P}},\sqrt[\ell]{b'_{2,P}}, \ldots, \sqrt[\ell]{b'_{\ell,P}} }\\
\end{align*}
Note that in either case $\tL_P/k_{P,\eta}$ and $\tL'_P/k_{P,\eta}$ are Galois extensions.

\begin{proposition}
\label{prop-points-normsfromGaloisclosures-RES} 
Let $\eta$, $P$, $u'$, $\tL_P$ and $\tL'_P$ be as above. There exist $w_P, w'_{P}\in k_{P,\eta}$, $z_P\in \tL_P$ and $z'_P\in \tL'_P$ such that 
\begin{enumerate}
\item
 $\overline{E_{1,P}\otimes F_{P,\eta}}=k_{P,\eta}[t]/(t^{\ell}- w_{P})$, $\N_{\tL_P/k_{P,\eta}}(z_P)=w_{P}$ and $\brac{w_{P}, u'}=0\in \Br\brac{k_{P,\eta}}$.
 \item
 $\overline{E_{2,P}\otimes F_{P,\eta}}=k_{P,\eta}[t]/(t^{\ell}- w'_{P})$, $\N_{\tL'_P/k_{P,\eta}}(z'_P)=w'_{P}$ and $\brac{w'_{P}, u'}=0\in \Br\brac{k_{P,\eta}}$.
 \end{enumerate}
\end{proposition}

\newpage

\begin{proof}
Since $\eta$ is of Type 1b or 2 and $Y_{\eta}$ is RES, $\eta$ is coloured green. We investigate the relevant rows in the tables given in the proof of Proposition \ref{propositionEatpoints}.

\textit{$\eta$ of Type 1b}: 

Choose $w_P = w'_P = z_P = z'_P = 1$ for {$B_{10}^{s}$ points} (cf. Row 1.2 of Table \ref{TableEatB10spoint} and Row  a in Table \ref{Tableliftofresidues}) and {$C_{12}^{Hot}$ points} (cf. Row 8.6 of Table \ref{TableEatC21hotpoint} and Row l in Table \ref{Tableliftofresidues}).

For the following situations, $Y'_P=k_{P,\eta}$ is a nonsplit (unramified/ramified extension). Thus $\tL_P= \tL'_P = Y'_P = k_{P,\eta}\brac{\sqrt[\ell]{u'}}$. Choose $w_{P} =w'_P = u'$ and $z_P = z'_P = \sqrt[\ell]{u'}$. 

- {$B_{11}^{ns}$ point}  : (cf Row 2.2 of Table \ref{TableEatB11nspoint} and Row b in Table \ref{Tableliftofresidues})\\
-{$C_{11}^{Chilly}$ point} : (cf. Rows 5.3, 5.4 and 5.9 of Table \ref{TableEatC11chillypoint} and Rows e,f in Table \ref{Tableliftofresidues})\\
-{$C_{11}^{Cold}$ point} : (cf. Rows 6.2 and 6.4 of Table \ref{TableEatC11coldpoint} and Row g,h in Table \ref{Tableliftofresidues})\\ 
-{$C_{12}^{Cold}$ point} : (cf Row 7.2 of Table \ref{TableEatC21coldpoint} and Row i in Table \ref{Tableliftofresidues})

\textit{$\eta$ of Type 2}:

Choose $w_P = w'_P = z_P = z'_P = 1$ for {$B_{20}^s$ points} (cf Row 3.1 of Table \ref{TableEatB20spoint} and Row c in Table \ref{Tableliftofresidues}).

For the following situations, $Y'_P/k_{P,\eta}$ is a nonsplit (unramified/ramified extension). Thus $\tL_P= \tL'_P = Y'_P = k_{P,\eta}\brac{\sqrt[\ell]{u'}}$. Choose $w_{P} =w'_P = u'$ and $z_P = z'_P = \sqrt[\ell]{u'}$. 

- {$B_{21}^{ns}$ point} : (cf. Row 4.1 of Table \ref{TableEatB21nspoint} and Row d in Table \ref{Tableliftofresidues})\\
- {$C_{12}^{Cold}$ point} : (cf Row 7.4 of Table \ref{TableEatC21coldpoint} and Row j in Table \ref{Tableliftofresidues}).

\item We are left with the case of hot points. Rows 8.3-8.4 of Table \ref{TableEatC21hotpoint} are relevant here (with $\eta=\eta_2$). Note that by Row m of Table \ref{Tableliftofresidues}, we know $Y'_P$ is split. Hence $u'\in k_{P,\eta}^{*\ell}$ and therefore $\brac{w_P,u'} = \brac{w'_P,u'}=0$ for whatever be the choice of $w_P$ and $w'_P$. However we need to be more careful in making our choice to ensure the existence of $z_P$ and $z'_P$.

\textit{Row 8.3 of Table \ref{TableEatC21hotpoint}}: Observe $\eta_1$ is coloured indigo. This implies that the $b_{i,P}$ are all units in the local field $k_{P,\eta}$. Hence $\tL_P$ is an unramified extension of $k_{P,\eta}$. By construction, $b'_{i,P}=1$ and hence $\tL'_P = k_{P,\eta}$. Thus all units of $k_{P,\eta}$ are norms from $\tL_P$ and $\tL'_P$. Choose $w_P=w'_P$ to be a unit  in $k_{P,\eta}$ which is not an $\ell^{\mrm{th}}$ power and $z_P\in \tL_P$ such that $\N_{L_P/k_{P,\eta}}(z_P)=w_P$. Also set $z'_P = w'_P$.

\textit{Row 8.4 of Table \ref{TableEatC21hotpoint}}: By choice $b_{i,P} = \overline{\pi_P^{m_i}}$. Hence we see that $\tL_P = k_{P,\eta}\brac{\sqrt[\ell]{\overline{\pi_P}}}$. Set $w_P = \overline{\pi_P}$ which is clearly a norm from $\tL_P$. Note that $b'_{i,P}=\overline{z_{i,P}}$. Since $z_{i,P}$ are units in $\widehat{A_P}^*$, we see that $\tL'_P$ is an unramified extension of $k_{P,\eta}$. Thus all units of $k_{P,\eta}$ are norms from $\tL'_P$. As before choose $w'_P$ to be a unit in $\mathcal{O}_{k_{P,\eta}}$ which is not an $\ell^{\mrm{th}}$ power and $z'_P\in \tL'_P$ such that $\N_{\tL'_P/k_{P,\eta}}(z'_P)=w'_P$. \end{proof}

\subsection{When $\eta$ is 1a}
Let $\eta=\eta_1\in N_0$ be of Type 1a. For convenience, we again summarize of the choice of $E_{j,P}$ at points $P\in \overline{\eta}\cap S_0$ for $j=1,2$ while also tabulating the shape of $Y$. Note that both extensions are unramified along $\eta$.. Note also that by Proposition \ref{propaisaunitifYnotsplit} and Lemma \ref{lemmaalong1a}, $a$ is a unit along $\eta$ up to $\ell^{\mathrm{th}}$ powers.
\begin{figure}[H]
\centering
\footnotesize
\arraycolsep=0.5pt
\medmuskip =0.5mu 
\[
\begin{array}{|c|c|c|c|c|c|c|c|c|c|c|}
\hline 

\mrm{Row} & \mrm{Colour\ of\ }\eta_2 & {\eta_2} &  {P} &  {Y_{\eta_1}} & {Y_{\eta_2}} & {Y_P} &  {E_{1,P}} & {E_{2,P}} & {a_{1,P}} & {a_{2,P}=aa_{1,P}^{-1}} \\
\hline
R.1 &  & 0 & A_{10}^{s} & \mrm{RAM} & \mrm{RAM} &  F_P\left(\sqrt[\ell]{w_P\pi_P^r\delta_P}\right)  & {\prod F_P} &  {\prod F_P} &  {a} & {1}\\
\hline
R.2 &  & 0 & A_{10}^{s} & \mrm{RAM} & \mrm{NONRES} &  F_P\left(\sqrt[\ell]{w_P\pi_P}\right)  & {\prod F_P} &  {\prod F_P} &  {a} & {1} \\
\hline
R.3 &  {G} &  \mrm{1b} & B_{11}^{ns} & \mrm{RAM} &  \mrm{RAM}  &  F_P\left(\sqrt[\ell]{w_P\pi_P^r\delta_P}\right) &  L_P & {L_P}&  {a} & {1}\\
\hline
R.4 &  {G} &  \mrm{1b}   & B_{11}^{ns} & \mrm{RAM} & \mrm{NONRES}  & F_P\left(\sqrt[\ell]{w_P\pi_P}\right)  &  L_P & L_P &  {a} & {1}\\
\hline
R.5 & {G} & 2 & B_{12}^{ns} & \mrm{RAM} & \mrm{RAM} & F_P\left(\sqrt[\ell]{w_P\pi_P^r\delta_P}\right) & L_P &  L_P &  {a} & {1}\\
\hline
R.6 & {G} & 2 & B_{12}^{ns} & \mrm{RAM} & \mrm{NONRES} & F_P\left(\sqrt[\ell]{w_P\pi_P}\right) & L_P &  L_P &  {a} & {1}\\
\hline
\hline
S.1 &  & 0 & A_{10}^{s} & \mrm{SPLIT} & \mrm{SPLIT} &  {\prod F_P} & {\prod F_P} &  {\prod F_P} &   {a} & {1} \\
\hline
S.2 &  & 0 & A_{10}^{s} & \mrm{SPLIT} & \mrm{NONRES} &  {\prod F_P} & {\prod F_P} &  {\prod F_P} &  {a} & {1} \\
\hline
S.3 & {Bl, I, V} & \mrm{1b} & B_{11}^{ns} & \mrm{SPLIT} & \mrm{SPLIT} &  {\prod F_P} &
 L_P &  L_P &  {a} & {1} \\
\hline
S.4 & {W,R,O,Ye} & \mrm{1b} & B_{11}^{ns} & \mrm{SPLIT} & \mrm{SPLIT} &  {\prod F_P} &  F_P\left(\sqrt[\ell]{\delta_P}\right) & L_P   & \brac{\delta_P^{n_i}} &  \brac{z_{i,P}\pi_P^{\ell m'_i}}\\
\hline
S.5 & {B} & \mrm{1b} & B_{11}^{ns} & \mrm{SPLIT} & \mrm{SPLIT} &  {\prod F_P} &  L_P & F_P\left(\sqrt[\ell]{\delta_P}\right)    &   \brac{z_{i,P}\pi_P^{\ell m'_i}} & \brac{\delta_P^{n_i}} \\
\hline
S.6 & {G} & \mrm{1b} & B_{11}^{ns} & \mrm{SPLIT} & \mrm{NONRES} &  {\prod F_P} &
 L_P &  L_P &  {a} & {1} \\
\hline
S.7 & {G}  & 2 & B_{12}^{ns} & \mrm{SPLIT} & \mrm{NONRES} &  {\prod F_P} & L_P &  L_P & {a} & {1} \\
\hline
\hline
NR.1 &  & 0 & A_{10}^{s} & \mrm{NONRES} & \mrm{RAM} &  F_P\brac{\sqrt[\ell]{w_P\delta_P}} & {\prod F_P} &  {\prod F_P} &  {a} & {1}  \\
\hline
NR.2 &  &0 & A_{10}^{s} & \mrm{NONRES} & \mrm{SPLIT} &  {\prod F_P} & {\prod F_P} &  {\prod F_P} &  {a} & {1} \\
\hline
NR.3 &  &  0 & A_{10}^{s} & \mrm{NONRES} & \mrm{NONRES} &  F_P\brac{\sqrt[\ell]{w_P}} & {\prod F_P} &  {\prod F_P} &  {a} & {1}\\
\hline


NR.4 & {G} &  \mrm{1b} & B_{11}^{ns} & \mrm{NONRES} & \mrm{RAM} &  F_P\brac{\sqrt[\ell]{w_P\delta_P}} & \prod F_P &  \prod F_P & {a} & {1} \\
\hline

NR.5 & {G}  &  \mrm{1b} & B_{11}^{ns} & \mrm{NONRES} & \mrm{RES} &  L_P & L_P  &  L_P  & {a} & {1} \\
\hline

NR.6 & {Bl, I, V} &  \mrm{1b} & B_{11}^{ns} & \mrm{NONRES} & \mrm{SPLIT} &  {\prod F_P} & L_P & L_P   & {a} & {1}\\
\hline

NR.7  & {W, R, O, Ye} &  \mrm{1b} & B_{11}^{ns} & \mrm{NONRES} & \mrm{SPLIT} &  {\prod F_P}  & F_P\left(\sqrt[\ell]{\delta_P}\right) &  L_P    & \brac{\delta_P^{n_i}} &  \brac{z_{i,P}\pi_P^{\ell m'_i}}   \\
\hline

NR.8  & {B} &  \mrm{1b} & B_{11}^{ns} & \mrm{NONRES} & \mrm{SPLIT} &  {\prod F_P} &  L_P & F_P\left(\sqrt[\ell]{\delta_P}\right)    &  \brac{z_{i,P}\pi_P^{\ell m'_i}} & \brac{\delta_P^{n_i}}    \\
\hline

NR. 9&  {G} &  \mrm{1b} & B_{11}^{ns} & \mrm{NONRES} & \mrm{NONRES} &  F_P\brac{\sqrt[\ell]{w_P}}  & L_P  &  L_P  & {a} & {1}\\
\hline
NR.10 & {G}  &  2 & B_{12}^{ns} & \mrm{NONRES} & \mrm{RAM} &  F_P\brac{\sqrt[\ell]{w_P\delta_P}} &  {\prod F_P}  &  {\prod F_P}  & {a} & {1} \\
\hline
NR.11 & {G}  &  2 & B_{12}^{ns} & \mrm{NONRES} & \mrm{RES} &  {\prod F_P} &  L_P  &  L_P  & {a} & {1} \\
\hline
NR.12 & {G}  &  2 & B_{12}^{ns} & \mrm{NONRES} & \mrm{NONRES} &  F_P\brac{\sqrt[\ell]{w_P}} &  L_P & L_P  & {a} & {1} \\
\hline

\end{array}\]

\captionof{table}{Patching data at closed points when $\eta_1$ is of Type 1a}
\label{TableEatclosedpoints-Type1a}
\end{figure}

\subsubsection{When $Y_{\eta}$ is NONRES}
\begin{proposition}
\label{prop-points-normsfromGaloisclosures-1a} 
Let $\eta = {\eta_1}$ be of Type 1a and let $P\in \overline{\eta}\cap S_0$. Further assume $Y_{\eta}$ is of Type NONRES. Let $Y' =\overline{Y_{\eta}} = k_{\eta}\brac{\sqrt[\ell]{u'}}$ and $Y'_P = Y'\otimes k_{P,\eta}$. If $Y'_P$ is nonsplit, set $\tL_P = \tL'_P = Y'_P$. If $Y'_P = \prod k_{P,\eta}$, let $\overline{a_{1,P}} = \brac{b_{i,P}}_i$ and $\overline{a_{2,P}} = \brac{b'_{i,P}}_i$ in $\prod k_{P,\eta}$ and set $\tL_P = k_{P,\eta}\brac{\sqrt[\ell]{b_{1,P}}, \sqrt[\ell]{b_{2,P}}, \ldots, \sqrt[\ell]{b_{\ell,P}}}$ and $\tL'_P  = k_{P,\eta}\brac{\sqrt[\ell]{b'_{1,P}}, \sqrt[\ell]{b'_{2,P}}, \ldots, \sqrt[\ell]{b'_{\ell,P}}}$. 

Then there exist $w_P, w'_{P}\in k_{P,\eta}$, $z_P\in \tL_P$ and $z'_P\in \tL'_P$ such that 
\begin{enumerate}
\item
 $\overline{E_{1,P}\otimes F_{P,\eta}}=k_{P,\eta}[t]/(t^{\ell}- w_{P})$ and $\N_{\tL_P/k_{P,\eta}}(z_P)=w_{P}$.
 \item
 $\overline{E_{2,P}\otimes F_{P,\eta}}=k_{P,\eta}[t]/(t^{\ell}- w'_{P})$ and $\N_{\tL'_P/k_{P,\eta}}(z'_P)=w'_{P}$.
 \end{enumerate}
\end{proposition}

\begin{proof}
In Table \ref{TableEatclosedpoints-Type1a}, the Rows NR.1-NR.12 are relevant. We only give the choices of $w_P$ and $w'_P$ from which existence of $z_P, z'_P$ become clear. For Rows NR.1-NR.4 and NR.10, choose $w_P = w'_P  = 1$. For Row NR.7, $\tilde{L}_P = k_{P,\eta}\brac{\sqrt[\ell]{\overline{\delta_P}}}$ is a ramified extension and $\tilde{L}'_P/k_{P,\eta}$ is an unramified extension. Choose $w_P = \overline{\delta_P}$ and $w'_P$ to be any unit in $\mathcal{O}_{k_{P,\eta}}$ which is not an $\ell^{\mrm{th}}$ power. For Row NR.8, a similar proof as in Row NR.7 works. For the remaining rows,  $\tilde{L}_P$ and $\tilde{L}'_P$ are unramified extensions of $k_{P,\eta}$. So again choose $w_P = w'_P$ to be any unit in $\mathcal{O}_{k_{P,\eta}}$ which is not an $\ell^{\mrm{th}}$ power. \end{proof}

\subsubsection{When $Y_{\eta}$ is SPLIT}
\begin{proposition}
\label{prop-points-normsfromGaloisclosures-1a-split} 
Let $\eta=\eta_1$ be of Type 1a and let $P\in \overline{\eta}\cap S_0$. Further assume $Y_{\eta}$ is of Type SPLIT. Let\footnote{Modify $a_{j,P}$ by $\ell$th powers of the parameter $\pi_{\eta}$ if needed to define $(b_{i,P})$ and $(b'_{i,P})$.} ${\overline{a_{1,P}}}=\brac{b_{i,P}}$ (resp. ${\overline{a_{2,P}}}=\brac{b'_{i,P}}$). Set ${X_P} = k_{P,\eta}\brac{\sqrt[\ell]{b_{1,P}}, \ldots, \sqrt[\ell]{b_{\ell,P} }}$ and ${X'_P} = k_{P,\eta}\brac{\sqrt[\ell]{b'_{1,P}}, \ldots, \sqrt[\ell]{b'_{\ell,P} }}$.

Then there exist $w_P, w'_{P}\in k_{P,\eta}$, $z_P\in {X_P}$ and $z'_P\in {X'_P}$ such that 
\begin{enumerate}
\item
 $\overline{E_{1,P}\otimes F_{P,\eta}}=k_{P,\eta}[t]/(t^{\ell}- w_{P})$ and $w_P = \N_{{X_P}/k_{P,\eta}}\brac{z_P}$. 
\item
 $\overline{E_{2,P}\otimes F_{P,\eta}}=k_{P,\eta}[t]/(t^{\ell}- w'_{P})$ and $w'_P = \N_{{X'_P}/k_{P,\eta}}\brac{z'_P}$. 
\end{enumerate}
\end{proposition}

\begin{proof}
In Table \ref{TableEatclosedpoints-Type1a}, the Rows S.1-S.7 are relevant. We only give the choices of $w_P$ and $w'_P$ from which the existence of $z_P, z'_P$ become clear.

For Rows S.1-S.2, choose $w_P = w'_P  = 1$. For Row S.4, $X_P = k_{P,\eta}\brac{\sqrt[\ell]{\overline{\delta_P}}}$ is a ramified extension and $X'_P/k_{P,\eta}$ is an unramified extension. Choose $w_P = \overline{\delta_P}$ and $w'_P$ to be any unit in $\mathcal{O}_{k_{P,\eta}}$ which is not an $\ell^{\mrm{th}}$ power. For Row S.5, a similar proof as in Row S.4 works. For Rows S.3, S.6 and S.7,  $X_P$ and $X'_P$ are unramified extensions of $k_{P,\eta}$. So choose $w_P = w'_P$ to be any unit in $\mathcal{O}_{k_{P,\eta}}$ which is not an $\ell^{\mrm{th}}$ power. \end{proof}

\section{Patching data at coloured points of $N_0$}
\label{sectionpatchingdataatN0}

 Let $\eta\in N_0$ be of Type 1b or 2 where ${N_0}$ denotes the subset $N'_0\cap X_0$. Let $\pi_{\eta}$ denote the parameter of $F_{\eta}$ fixed in Section \ref{subsectionfinalmodel} and let $\beta_{rbc,\eta}$ be as defined in Section \ref{section-fixingresidualbrauerclass}. For $j=1,2$ and any $P\in S_0$, let $E_{j,P}$ and $a_{j,P}$ be as prescribed in Propositions \ref{propositionEatpoints}. We now prescribe the choices for $E_{j,\eta}$ and ${a}_{j,\eta}$.
 
\begin{proposition}[Violet/Indigo/Black]
\label{propositionE-choosingliftofresidues-violetandindigo}
Let $\eta\in N_0$ be coloured violet, indigo or black. Set $E_{1,\eta}$ and $E_{2,\eta}$ to be the lift of residues at $\widehat{A_{\eta}}$. Further, set $a_{1,\eta} = a$ and $a_{2,\eta}=1$. Then for $j=1,2$, we have 
\begin{enumerate}
\item 
$a_{1,\eta}a_{2,\eta} = a$. 

\item
 $D\otimes E_{j,\eta}$ is split.
  
\item
$a_{j,\eta}$ is a norm from $E_{j,\eta}\otimes Y_{\eta}/Y_{\eta}$. 

\item
$\N_{Y_{\eta}/F_{\eta}}\brac{a_{j,\eta}}=1$.
\item
$E_{j,\eta}\otimes F_{P,\eta} \simeq E_{j,P} \otimes F_{P,\eta}$ for each point $P\in S_0 \cap \overline{\eta}$.
\item
$a_{j,P} = a_{j, \eta} \in Y\otimes F_{P,\eta}$ for each point $P\in S_0 \cap \overline{\eta}$.
\end{enumerate}
\end{proposition}

\begin{proof} Recall that we have $D = {\beta}_{rbc,\eta} + \left(u_{\eta}, \pi_{\eta} \right)\ \mrm{in} \ \Br\left(F_{\eta}\right)$ where ${\beta}_{rbc,\eta}$ is an unramified algebra with index at most $\ell$. And by construction, $E_{j,\eta}\simeq F_{\eta}\brac{\sqrt[\ell]{u_{\eta}}}$. Properties  1 \& 4 are immediate, while Properties 5 \& 6 follow from Proposition \ref{otherproperties-violet-indigo}. Since $D$ is ramified at $\eta$, we have $\sqbrac{E_{j,\eta}:F_{\eta}}=\ell$. Since $\eta$ is of Type 1b, $D_{\eta}\simeq  \M_{\ell}\brac{u_{\eta}, w_{\eta}\pi_{\eta}}$ for some unit $w_{\eta}\in \widehat{A_{\eta}}^*$. Thus $D\otimes E_{j,\eta}$ is split which shows Property 2 holds.

Since $\eta$ is coloured violet/indigo/black, then it has to be a Ch/H/Z curve respectively. In particular $\eta$ is a  Type 1b curve with $Y_{\eta}$ of Type SPLIT and $a = \brac{a'_{i,\eta}}_i$ where each $a'_{i,\eta}$ is a unit up to $\ell^{\mrm{th}}$ powers in $\widehat{A_{\eta}}$. Therefore the fact that $a\in \Nrd\left(D\otimes Y_{\eta}\right)$ translates to $\left(a'_{i,\eta}\right)\left(u_{\eta},\pi_{\eta}\right) = 0$ in  $\HH^3\left(F_{\eta}, \mu_{\ell}\right)$ for all $i$ (Lemma \ref{lemmareducednormsunramified}). This implies by taking residues that $\overline{\left(u_{\eta}, a'_{i,\eta}\right)}=0 \ \mathrm{in} \ \HH^2\left({k_{\eta}}, \mu_{\ell}\right)$ for all $i$. Hence each $a'_{i,\eta}$ is a norm from $E_{1,\eta}$. Since clearly $1$ is a norm from $E_{2,\eta}$, Property 3 holds. \end{proof}

\begin{proposition}[Blue]
\label{propositionE-1bSPLIT-Blue}
Let $\eta\in N_0$ be coloured blue. Set $E_{1,\eta}$ to be the lift of residues at $\widehat{A_{\eta}}$. Then there exists a ramified cyclic extension $E_{2,\eta}/F_{\eta}$ of degree $\ell$ and elements $a_{1,\eta} = \brac{\tilde{a}_{1,i,\eta}}_i$ and  $a_{2,\eta} = \brac{\tilde{a}_{2,i,\eta}}_i \in \prod F_{\eta}$ such that for $j=1,2$, the following holds:
\begin{enumerate}
\item 
$a_{1,\eta}a_{2,\eta} = a \in Y_{\eta}$, i.e. $\tilde{a}_{1,i,\eta}\tilde{a}_{2,i,\eta}= a'_{i,\eta}$ for each $i$.

\item
$D\otimes E_{j, \eta}$ is split.

\item
$a_{j,\eta}$ is a norm from $E_{j,\eta}\otimes Y_{\eta}/Y_{\eta}$, i.e . $\tilde{a}_{j,i,\eta}$ is a norm from $E_{j,\eta}$ for each $i$.

\item
$\N_{Y_{\eta}/F_{\eta}}\brac{a_{j,\eta}}=1$, i.e. $\prod_i \tilde{a}_{j,i,\eta} = 1$.
\item
$E_{j,\eta}\otimes F_{P,\eta} \simeq E_{j,P} \otimes F_{P,\eta}$ for each point $P\in S_0 \cap \overline{\eta}$.
\item
 $\tilde{a}_{j,i,\eta}{\mu_{j,i,P,\eta}} = \tilde{a}_{j,i,P} \in F_{P,\eta}$ for all $i$ at each point $P\in S_0 \cap \overline{\eta}$ for some $\mu_{j,i,P,\eta}\in F_{P,\eta}^{*\ell}$ such that $\prod_i \mu_{j,i,P,\eta}=1$.
 \item
 $E_{1,\eta}/F_{\eta}$ is unramified and cyclic of degree $\ell$.
\end{enumerate}
\end{proposition}

\begin{proof} Since $\eta$ is coloured blue, it is a C curve of Type 1b with $Y_{\eta}$ of Type SPLIT. Write $D \simeq \M_{\ell}\brac{u_{\eta}, w_{\eta}\pi_{\eta}}$ where $w_{\eta}, u_{\eta}\in \widehat{A_{\eta}}^*$ and $a = \brac{a'_{i,\eta}}_i$  where $a'_{i,\eta} = x'_i\pi_{\eta}^{m_i}$ with $x'_i\in \widehat{A_{\eta}}^*$. Because $\N(a)=1$, we have $\sum m_i=0$ and $\prod x'_i = 1$.

Let $P\in \overline{\eta}\cap S_0$. By Proposition \ref{otherproperties-blue},  $E_{2,P} = \frac{F_P[t]}{(t^{\ell}-w_P\pi_P)}$ and $a_{2,P} = \brac{\tilde{a}_{2,i,P}}$ where $\tilde{a}_{2,i,P} = \brac{w_P\pi_P}^{m_{i,P}}{x_{i,P}^{\ell}}$ for some $w_P, x_{i,P}\in \widehat{A_P}$ which are units along $\eta$. Further $E_{1,P}$ matches the lift of residues along $F_{P,\eta}$ and $a_{1,P} = \brac{\tilde{a}_{1,i,P}}$ where $\tilde{a}_{1,i,P}$ are all units along $\eta$. Thus, since $a$ is arranged to be in good shape, we have $m_i = m_{i,P}$.

Let $\tilde{X}_{\eta} = F_{\eta}\brac{\sqrt[\ell]{u_{\eta}}}$ and let $X_{\eta} = k_{\eta}\brac{\sqrt[\ell]{\overline{u_{\eta}}}}$. Our goal is to find a $\theta_{\eta}\in \widehat{A_{\eta}}^*$ which is a norm from $\tilde{X}_{\eta}$ so that $\overline{w_{\eta}\theta_{\eta}}$ is close to $\overline{w_P}$ in $k_{P,\eta}$ for each $P\in \overline{\eta}\cap S_0$. 

If such a $\theta_{\eta}$ exists, then we set 
\begin{align*}
E_{2,\eta} &= F_{\eta}\brac{\sqrt[\ell]{w_{\eta}\theta_{\eta}\pi_{\eta}}}, \\
E_{1,\eta} & = \tilde{X}_{\eta} = F_{\eta}\brac{\sqrt[\ell]{u_{\eta}}},\\
\tilde{a}_{2,i,\eta} & = \brac{w_{\eta}\theta_{\eta}\pi_{\eta}}^{m_i} \forall \  i\leq \ell-1,\\
\tilde{a}_{2,\ell,\eta} & = \brac{\tilde{a}_{2,1,\eta}\ldots \tilde{a}_{2,\ell-1,\eta}}^{-1},\\
a_{1,\eta} & = aa_{2,\eta}^{-1}.
\end{align*}

Thus clearly Properties 1, 4, 5, 6 and 7 hold. Since $\theta_{\eta}$ is assumed to be a norm from $\tilde{X}_{\eta}$, we have $\brac{u_{\eta}, \theta_{\eta}} = 0 \in \Br\brac{F_{\eta}}$. Hence $D = \brac{u_{\eta}, w_{\eta}\theta_{\eta}\pi_{\eta}} \in \Br\brac{F_{\eta}}$, which therefore implies Property 2 holds.

Let us check that Property 3 holds. Clearly $\tilde{a}_{2,i,\eta} = \brac{w_{\eta}\theta_{\eta}\pi_{\eta}}^{m_i}$ is a norm from $E_{2,\eta}$ for each $i\leq \ell-1$. Since $1$ is a norm always, so is $\tilde{a}_{2,\ell,\eta}$. It is left to show that $\tilde{a}_{1,i,\eta} = x'_i \brac{w_{\eta}\theta_{\eta}}^{-m_i}$ is a norm from $\tilde{X}_{\eta}$ for each $i\leq \ell-1$ (which will automatically imply $\tilde{a}_{1,i,\eta}$ is a norm from $\tilde{X}_{\eta}$ also as $\N(a_{1,\eta})=1$).

Since each $a'_{i,\eta}$ is a reduced norm of $D$, we have $\brac{u_{\eta}, w_{\eta}\theta_{\eta}\pi_{\eta}} {\brac{a'_{i,\eta}}} = 0$. This implies $\brac{u_{\eta}, w_{\eta}\theta_{\eta}\pi_{\eta}} \brac{{\brac{w_{\eta}\theta_{\eta}\pi_{\eta}}^{m_i}}{x'_i \brac{w_{\eta}\theta_{\eta}}^{-m_i}}}= 0$ and hence $\brac{u_{\eta}, w_{\eta}\theta_{\eta}\pi_{\eta}} \brac{{x'_i \brac{w_{\eta}\theta_{\eta}}^{-m_i}}}= 0$. Taking residues, we see $\overline{\brac{u_{\eta}, x'_i \brac{w_{\eta}\theta_{\eta}}^{-m_i} }} = 0$ and thus each $\tilde{a}_{1,i,\eta}$ is a norm from $\tilde{X}_{\eta}$.

Now let us find $\theta_{\eta}$. Recall that $X_{\eta}$ is the residue of $D_{\eta}$. For each $P\in \overline{\eta}\cap S_0$, by Proposition \ref{otherproperties-blue} 5(b), we know that $\overline{w_Pw_{\eta}^{-1}}$ is a norm from $X_{\eta}\otimes k_{P,\eta}$. Thus for each $P\in \overline{\eta}\cap S_0$, let $z_{P,\eta}\in X_{\eta}\otimes k_{P,\eta}$ such that $\N\brac{z_{P,\eta}} = \overline{w_Pw_{\eta}^{-1}}$. By weak approximation, find $z\in X_{\eta}$ which is close to each $z_{P,\eta}$. Set $\theta = \N_{X_{\eta}/k_{\eta}}\brac{z} \in k_{\eta}$ and let $\theta_{\eta}$ denote its lift in $F_{\eta}$. This $\theta_{\eta}$ satisfies the required properties. \end{proof}

\begin{proposition}[Green(1)]
\label{propositionE-choosingliftofresidues}
Let $\eta\in N_0$ be of Type 1b/2 with $Y_{\eta}$ of Type NONRES. Set $E_{1,\eta}$ and $E_{2,\eta}$ to be the lift of residues at $\widehat{A_{\eta}}$. Further, set $a_{1,\eta} = a$ and $a_{2,\eta}=1$. Then for $j=1,2$, we have 
\begin{enumerate}
\item 
$a_{1,\eta}a_{2,\eta} = a$ 

\item
$D\otimes {E}_{j,\eta}$ has index at most $\ell$. 

\item
$D\otimes Y\otimes {E}_{j,\eta}$ is split.
 
\item
$a_{j,\eta}$ is a norm from $E_{j,\eta}\otimes Y_{\eta}/Y_{\eta}$. 

\item
$\N_{Y_{\eta}/F_{\eta}}\brac{a_{j,\eta}}=1$.
\item
$E_{j,\eta}\otimes F_{P,\eta} \simeq E_{j,P} \otimes F_{P,\eta}$ for each point $P\in S_0 \cap \overline{\eta}$.
\item
$a_{j,P} = a_{j, \eta} \in Y\otimes F_{P,\eta}$ for each point $P\in S_0 \cap \overline{\eta}$.
\end{enumerate}
\end{proposition}

\begin{proof}
By hypothesis, $\eta$ is coloured green. Recall that we have $D = {\beta}_{rbc,\eta} + \left(u_{\eta}, \pi_{\eta} \right)\ \mrm{in} \ \Br\left(F_{\eta}\right)$ where ${\beta}_{rbc,\eta}$ is unramified with index at most $\ell$. By construction, $E_{j,\eta}\simeq F_{\eta}\brac{\sqrt[\ell]{u_{\eta}}}$. Properties 1 \& 5 are immediate while Properties 6 \& 7 follow from Proposition \ref{otherproperties-green-NONRES}.  Note that $D \otimes E_{j,\eta} = {\beta}_{rbc,\eta}\otimes E_{j,\eta} \in \Br\brac{F_{\eta}}$ and hence Property 2 holds.

Since $D$ is ramified at $\eta$, we have $\sqbrac{E_{j,\eta}:F_{\eta}}=\ell$. As $Y_{\eta}$ is NONRES, we have {$E_{j,\eta}\otimes Y_{\eta}/Y_{\eta}$ is a field extension of degree $\ell$}. Then by Lemma \ref{lemmaindexformula}, the index of $\left(D\otimes Y_{\eta}\right)$ equals $\ind\left({\beta}_{rbc,\eta}\otimes_{F_{\eta}} Y_{\eta}\otimes_{F_{\eta}} E_{j,\eta} \right)\left[E_{j,\eta}\otimes Y_{\eta} : Y_{\eta}\right]$ which is $\ind\left({\beta}_{rbc,\eta}\otimes_{F_{\eta}} Y_{\eta}\otimes_{F_{\eta}} E_{j,\eta} \right)\times \ell$. Since $\ind\left(D\otimes Y_{\eta}\right) \leq \ell$, we see that $E_{j,\eta}\otimes_{F_{\eta}} Y_{\eta}$ splits ${\beta}_{rbc,\eta}$ and hence also $D$, implying Property 3.

By hypothesis, $a\in \Nrd\left(D\otimes Y_{\eta}\right)$. That is, $\left(a\right)\left[{\beta}_{rbc,\eta}\right] + \left(a\right)\left(u_{\eta},\pi_{\eta}\right) = 0 \ \mathrm{in} \ \HH^3\left(Y_{\eta}, \mu_{\ell}\right)$. By Proposition \ref{propaisaunitifYnotsplit}, $a$ is a unit at $\eta$. By Lemma \ref{lemmareducednormsunramified}, $a$ is a reduced norm of ${\beta}_{rbc,\eta}\otimes Y_{\eta}$ and therefore we have that $\left(a\right)\left(u_{\eta},\pi_{\eta}\right) = 0 \ \mathrm{in} \ \HH^3\left(Y_{\eta}, \mu_{\ell}\right)$. Thus by taking residues, we see that $\overline{\left(u_{\eta}, a\right)}=0 \ \mathrm{in} \ \HH^2\left(\overline{Y_{\eta}}, \mu_{\ell}\right)$ which would imply that $a$ is a norm from $E_{j,\eta}\otimes_{F_{\eta}}Y_{\eta}/Y_{\eta}$. Hence Property 4 holds.\end{proof}

\begin{proposition}[Green(2)-RAM]
\label{propositionE-1bRAM}
Let $\eta\in N_0$ be of Type 1b or 2 and let $Y_{\eta}$ be of Type RAM. Set $a_{1,\eta}=a$ and $a_{2,\eta}=1$. Then for $j=1,2$, there exist $E_{j,\eta}/F_{\eta}$, unramified cyclic extensions of degree $\ell$ such that 
\begin{enumerate}
\item 
$a_{1,\eta}a_{2,\eta}=a$
\item
$E_{j, \eta}$ splits the residual Brauer class ${\beta}_{rbc,\eta}$. 
\item
$D\otimes {E}_{j,\eta}$ has index at most $\ell$ and $D\otimes Y\otimes {E}_{j,\eta}$ is split.
\item
$a_{j,\eta}$ is a norm from $E_{j,\eta}\otimes Y_{\eta}/Y_{\eta}$. 
\item
$\N_{Y_{\eta}/F_{\eta}}\brac{a_{j,\eta}}=1$.
\item
$E_{j,\eta}\otimes F_{P,\eta} \simeq E_{j,P} \otimes F_{P,\eta}$ for each point $P\in S_0 \cap \overline{\eta}$.
\item
 $a_{j,\eta} = a_{j,P} \in Y\otimes F_{P,\eta}$ for each point $P\in S_0 \cap \overline{\eta}$.
\item
The residue of $D_{\eta}$ is a norm from $\overline{E_{j,\eta}}/k_{\eta}$.
\end{enumerate}
\end{proposition}

\begin{proof} By hypothesis, $\eta$ is coloured green. Properties 1 \& 5 are immediate. By Lemma \ref{lemmanormoneramified-dim1}, $a = a_{1,\eta}\in Y_{\eta}^{*\ell}$. Since $a_{2,\eta}=1$, Property 4 holds for whatever degree $\ell$ extensions $E_{j,\eta}$ we choose. We first construct $\overline{E_{j,\eta}}/k_{\eta}$ and then set $E_{j,\eta}/k_{\eta}$ to be the unramified lift of $\overline{E_{j,\eta}}/k_{\eta}$. We would like to apply Lemma \ref{lemmauglyinitial} to construct $\overline{E_{j,\eta}}$. 

Recall that we have $D = \beta_{rbc,\eta} + (u_{\eta} , w_{\eta}\pi_{\eta})$ in $\Br(F_{\eta})$ where $Y_{\eta} \simeq F_{\eta}(\sqrt[\ell]{w_{\eta}\pi_{\eta}})$ for $w_{\eta}\in \widehat{A_{\eta}}^*$.  Let $D' = \overline{\beta_{rbc,\eta}}$, the residual Brauer class considered over the residue field $k_{\eta}$. Thus $D'$ is a central simple algebra of exponent and index at most $\ell$ over global field $k_{\eta}$. Let $u' := \overline{u_{\eta}}\in k_{\eta}$. Let $\mathcal{P}'_{\eta} :=  \overline{\eta} \cap S_0$. Let $\mathcal{Q}'_{\eta}$ denote the set of closed points $Q\in \overline{\eta}$ not in $\mathcal{P}'_{\eta}$ such that $D'\otimes k_{Q,\eta} \neq 0$. 

For $P\in \mathcal{P}'_{\eta}$, set $E'_{j,P} := \overline{E_{j,P}\otimes F_{P,\eta}} / k_{P,\eta}$ and let $w_{1,P}, w_{2,P}\in k_{P,\eta}^*$ be the ones obtained from Proposition \ref{prop-points-normsfromGaloisclosures-RAM}. So $\overline{E_{j,P}\otimes F_{P,\eta}} = k_{P,\eta}[t]/(t^{\ell}-w_{j,P})$ and $\brac{u',w_{j,P}}=0$. For $Q\in \mathcal{Q}'_{\eta}$, set $E'_{1,Q}$ and $E'_{2,Q}$ to be the unique unramified field extension of $k_{Q,\eta}$ and set $w_{1,Q} = w_{2,Q}$ be any unit in $\mathcal{O}_{k_{Q,\eta}}$ which is not an $\ell^{\mrm{th}}$ power.

Let us verify that the hypotheses of Lemma \ref{lemmauglyinitial} hold now. Let $P\in \mathcal{P}'_{\eta}\cup \mathcal{Q}'_{\eta}$. We first check that $\brac{w_{j,P}, u'}=0\in \Br\brac{k_{P,\eta}}$. For $P\in \mathcal{P}'_{\eta}$, this is assured by Proposition \ref{prop-points-normsfromGaloisclosures-RAM}. For $P\in \mathcal{Q}'_{\eta}$, $u'$ is a unit in $\mathcal{O}_{k_{P,\eta}}$. Since $w_{j,P}$ is a unit, by local class field theory, $\brac{w_{j,P},u'} = 0$.

Next we verify that $D'\otimes E'_{j,P}$ is trivial. For $P\in \mathcal{P}'_{\eta}$, this is assured by Proposition \ref{otherproperties-green-notnonres}. For $P\in \mathcal{Q}'_{\eta}$, since each $E'_{j,P}$ is a nonsplit unramified extension of degree $\ell$, local class field theory guarantees that it will split any index $\ell$ algebra over $k_{P, \eta}$.

Clearly for each $Q\not\in \brac{ \mathcal{P}'_{\eta}\cup \mathcal{Q}'_{\eta}}$, $D'\otimes k_{Q,\eta}$ is split.

Thus Lemma \ref{lemmauglyinitial} can be used to construct $\overline{E_{1,\eta}}$ and $\overline{E_{2,\eta}}$ over $k_{\eta}$. Setting $E_{1,\eta}$ and $E_{2,\eta}$ to be their respective unramified lifts over $F_{\eta}$, it is immediate that Properties 2, 6, and 8 are satisfied. Property 7 is guaranteed by again using Proposition \ref{otherproperties-green-notnonres}. 

To complete the proof of Property 3, note that as $E_{j,\eta}$ splits $\beta_{rbc,\eta}$ and $Y_{\eta} = F_{\eta}\brac{\sqrt[\ell]{w_{\eta}\pi_{\eta}}}$ and $D = \beta_{rbc,\eta} + \left(u_{\eta}, w_{\eta}\pi_{\eta}\right) \in \Br\left(F_{\eta}\right)$, it is immediate that $\ind\left(D\otimes_F E_{i,\eta}\right) \leq \ell$ and that $D\otimes Y_{\eta} \otimes E_{j,\eta}$ is split. \end{proof}

\begin{proposition}[Green(2)-RES]
\label{propositionE-1bRESY}
\label{propositionE-2RESY}
Let $\eta\in N_0$ be of Type 1b or 2 and let $Y_{\eta}$ be of Type RES. Then for $j=1,2$, there exist $E_{j,\eta}/F_{\eta}$, unramified cyclic extensions of degree $\ell$ and elements $a_{1,\eta}, a_{2,\eta}\in \mathcal{O}_{Y_{\eta}}$ such that 
\begin{enumerate}
\item 
$a_{1,\eta}a_{2,\eta}=a$.
\item
$E_{j, \eta}$ splits $\beta_{rbc,\eta}$. 
\item
$D\otimes {E}_{j,\eta}$ has index at most $\ell$ and $D\otimes Y\otimes {E}_{j,\eta}$ is split.
\item
$a_{j,\eta}$ is a norm from $E_{j,\eta}\otimes Y_{\eta}/Y_{\eta}$. 
\item
$\N_{Y_{\eta}/F_{\eta}}\brac{a_{j,\eta}}=1$.
\item
$E_{j,\eta}\otimes F_{P,\eta} \simeq E_{j,P} \otimes F_{P,\eta}$ for each point $P\in S_0 \cap \overline{\eta}$.
\item
 $a_{j,\eta}{\mu_{j,P,\eta}} = a_{j,P}$ in $Y\otimes F_{P,\eta}$ for each point $P\in S_0 \cap \overline{\eta}$ for some $\mu_{j,P,\eta}\in  \mathcal{O}_{Y\otimes F_{P,\eta}}$ such that $\mu_{j,P,\eta} \cong 1$ mod  $\brac{\pi_{\eta}}$ and $\N\brac{\mu_{j,P,\eta}}=1$. 
\item
The residue of $D\otimes F_{\eta}$ is a norm from $\overline{E_{j,\eta}}/k_{\eta}$.
\end{enumerate}
\end{proposition}

\begin{proof} By hypothesis, $\eta$ is coloured green. Since $Y_{\eta}$ is of Type RES, by Lemma \ref{propaisaunitifYnotsplit} we have that $a\in \mathcal{O}_{Y,\eta}^*$. Recall that we have $D = \beta_{rbc,\eta} + (u_{\eta} , \pi_{\eta})$ in $\Br(F_{\eta})$.  Set $u' := \overline{u_{\eta}}\in k_{\eta}$, $Y' := \overline{Y_{\eta}} = k_{\eta}\brac{\sqrt[\ell]{\overline{u'}}}$ and $a'=\overline{a} \in Y'$ and let $\Gal\brac{Y'/k_{\eta}} = \langle \psi \rangle$.

Let $\mathcal{P}'_{\eta} :=  \overline{\eta} \cap S_0$. By Proposition \ref{otherproperties-green-notnonres}, $a_{j,P}$ is a unit along $\eta$. First let's construct $a'_1\in Y'$ approximating $\overline{a_{1,P}}\in \overline{Y\otimes F_P \otimes F_{P,\eta}}$ for each $P\in \mathcal{P}'_{\eta}$. Since $\N\brac{\overline{a_{1,P}}}=1$, by Hilbert 90 there exists $c_{P}\in Y'\otimes k_{P,\eta}$ such that $c_P^{-1}\psi\brac{c_P}=\overline{a_{1,P}}$. Using weak approximation, find $c\in Y'$ which is close to $c_P$ for each $P\in \mathcal{P}'_{\eta}$. Set $a'_1 = c^{-1}\psi(c)$ and set $a'_2 = a'{a'_1}^{-1}$. Let $a_{1,\eta}$ denote a lift of $a'_1$ and let $a_{2,\eta} = aa_{1,\eta}^{-1}$.

Therefore Properties 1, 5 \& 7 are immediate. We will first construct $\overline{E_{j,\eta}}/k_{\eta}$ and then set $E_{j,\eta}/k_{\eta}$ to be the unramified lift of $\overline{E_{j,\eta}}/k_{\eta}$. We appeal to Lemma \ref{lemmaugly} to construct $\overline{E_{j,\eta}}$.

Let $D' = \overline{\beta_{rbc,\eta}}$, the residual Brauer class considered over the residue field $k_{\eta}$. Thus $D'$ is a central simple algebra of exponent and index at most $\ell$ over global field $k_{\eta}$.  Let $\mathcal{Q}'_{\eta}$ denote the set of closed points $Q\in \overline{\eta}$ not in $\mathcal{P}'_{\eta}$ such that $D'\otimes k_{Q,\eta} \neq 0$. 

Let $j=1$ or $2$. For $P\in \mathcal{P}'_{\eta}$, set $E'_{j,P} := \overline{E_{j,P}\otimes F_{P,\eta}} / k_{P,\eta}$. For $P\in \mathcal{Q}'_{\eta}$, set $E'_{j,P}$ to be the unique unramified field extension of $k_{P,\eta}$. Let $\tL$ denote the Galois closure of $Y'(\sqrt[\ell]{a_1'})$ and let $\tL'$ denote the Galois closure of $Y'(\sqrt[\ell]{a_2'})$. 

Whenever $Y'\otimes k_{Q,\eta}$ is not split, since $a'_1$ and $a'_2$ have norm 1, they are also units and in fact $\ell^{\mrm{th}}$ powers in the complete discretely valued field $Y'\otimes k_{Q,\eta}$ for every $Q\in \overline{\eta}$ (Lemmata \ref{lemmanormoneramified-dim1} and \ref{lemmanormoneunramified}). Further, since $Y_{\eta}$ is RES, we note that $Y'$ is unramified except at points $P \in S_0\cap \overline{\eta}$ of Type $C_{11}^{Cold}$ or $C_{12}^{Cold}$. For each $P\in \mathcal{P}'_{\eta} $,  recall the extensions $\tL_P, \tL'_P$ defined in Section \ref{section-GaloisclosureRES}. Note that the extension $\tL\otimes_{k_{\eta}} k_{P,\eta}\simeq \prod_{i=1}^{g} \tL_P$ and extension $\tL'\otimes_{k_{\eta}} k_{P,\eta}\simeq \prod_{j=1}^{h} \tL'_P$ for some $g, h\geq 1$.

Proposition \ref{prop-points-normsfromGaloisclosures-RES} says that there exist $w_P, w'_P \in k_{P,\eta}$, $z_P\in \tL_P$ and $z'_P\in \tL'_P$ so that $E'_{1,P} = k_{P,\eta}[t]/\brac{t^{\ell}-w_P}$, $\N_{\tL_P/k_{P,\eta}}\brac{z_P} = w_P$,  $E'_{2,P} = k_{P,\eta}[t]/\brac{t^{\ell}-w'_P}$ and $\N_{\tL'_P/k_{P,\eta}}\brac{z'_P} = w'_P$. For each $P\in \mathcal{Q}'_{\eta}$, let $w_P = w'_P$ be any unit in $\mathcal{O}_{k_{P,\eta}}$ which is not an $\ell^{\mrm{th}}$ power. 

We now construct\footnote{If $a'_j$ is an $\ell^{\mrm{th}}$ power in $Y'$, Property 4 is automatically satisfied for $a_{j,\eta}$. A check of the relevant rows mentioned in Proposition \ref{otherproperties-green-notnonres} show $u'$ is a norm from $\overline{E_{j,P}}\otimes k_{P,\eta}$. We can then use Lemma \ref{lemmauglyinitial} to construct $\overline{E_{j,\eta}}$.} the extensions $\overline{E_{j,\eta}}$ using Lemma \ref{lemmaugly} by verifying that the hypotheses of the same hold. Let $P\in \mathcal{P}'_{\eta}\cup \mathcal{Q}'_{\eta}$.

We need to find $\tilde{z}_P\in \tL\otimes k_{P,\eta} =\prod \tL_P$ (respectively $\tilde{z}'_P\in \tL'\otimes k_{P,\eta} =\prod \tL'_P$) such that its norm to $k_{P,\eta}$ is $w_P$ (resp $w'_P$). For $P\in \mathcal{P}'_{\eta}$, {set $\tilde{z_P} = \brac{z_P, 1, 1, \ldots, 1}$ and $\tilde{z'_P} = \brac{z'_P, 1, 1, \ldots, 1}$} and use Proposition \ref{prop-points-normsfromGaloisclosures-RES} to conclude the proof in this case. For $P\in \mathcal{Q}'_{\eta}$, we claim that $Y'_P$ is a nonsplit unramified extension of $k_{P,\eta}$. This is because of the following:

Write $D = D_{00} + \brac{u_P, \pi_{P}}\in \Br(F)$ where $u_P\in \widehat{A_P}^*$ and $\pi_P$ is a prime corresponding to the curve $\overline{\eta}$ (\cite{S97}).  We also have $D = \beta_{rbc,\eta} + \brac{u_{\eta}, \pi_{\eta}}\in \Br\brac{F_{\eta}}$. Note that $\pi_P = \pi_{\eta}w_{\eta}\in F_{\eta}$ where $w_{\eta}\in \widehat{A_{\eta}}^*$. Comparing these two expressions in $\Br\brac{F_{P,\eta}}$, we see $\brac{u_P, w_{\eta}} +   \brac{u_P, \pi_{\eta}}  = \beta_{rbc,\eta} + \brac{u_{\eta}, \pi_{\eta}}$. Taking residues, we see that $\overline{u_P}\overline{u_{\eta}}^{-1} = 1$ up to $\ell^{\mrm{th}}$ powers in $k_{P,\eta}$. And hence $\beta_{rbc,\eta} = \brac{u_\eta, w_{\eta}} \in \Br\brac{F_{P,\eta}}$. Now we are looking at a place $P\not\in \mathcal{P}'_{\eta}$ such that this algebra is not trivial. In particular, this implies $\overline{u_{\eta}}$ is not an $\ell^{\mrm{th}}$ power. Therefore $Y'_P$ is not split.

As observed before, this further implies $a'_1$ and $a'_2$ are units (and in fact $\ell^{\mrm{th}}$ powers) in $\mathcal{O}_{Y'_P}$. Thus $\tL\otimes k_{P,\eta} = \prod \tL_P$ and $\tL'\otimes k_{P,\eta} = \prod \tL'_P $ where $\tL_P =\tL'_P = Y'_P$. Since $Y'_P$ is unramified nonsplit extension of $k_{P,\eta}$, every unit of $k_{P,\eta}$ is a norm from it and hence from $\tilde{L}_P$ and $\tilde{L}'_P$, which finishes the proof of this case.

We need to verify that $D'\otimes E'_{j,P}$ is trivial for all $P\in \overline{\eta}$. For $P\in \mathcal{P}'_{\eta}$, use Proposition \ref{otherproperties-green-notnonres} to conclude the proof in this case. For $P\in \mathcal{Q}'_{\eta}$, since each $E'_{j,P}$ is a nonsplit unramified extension of degree $\ell$, local class field theory guarantees that it will split any index $\ell$ algebra over $k_{P, \eta}$. Also clearly for each $Q\not\in \brac{ \mathcal{P}'_{\eta}\cup \mathcal{Q}'_{\eta}}$, $D'\otimes k_{Q,\eta}$ is split already. 

We need to verify that $\brac{w_P, a'_1}_y = \brac{w'_P, a'_2}_y = 0$ for every valuation $y\in \Omega_{Y'}$ lying over $P$. For $P\in \mathcal{P}'_{\eta}$, this is assured by Proposition \ref{propositionEatpoints} (4). For $P\in \mathcal{Q}'_{\eta}$, we have already noted that $Y'_P$ is unramified and nonsplit over $k_{P,\eta}$. By Lemma \ref{lemmanormoneunramified}, $a'_1$ and $a'_2$ are $\ell^{\mrm{th}}$ powers in $Y'\otimes k_{P,\eta}$. So $\brac{w_P, a'_1}= 0 = \brac{w'_P, a'_2}$.

Thus Lemma \ref{lemmaugly} can be used to construct $\overline{E_{j,\eta}}$ over $k_{\eta}$ for $j=1,2$. Setting $E_{j,\eta}$ to be their respective unramified lifts over $F_{\eta}$, it is immediate that Properties 2, 4, 6 and 8 are satisfied. To complete the proof of Property 3, note that as $E_{j,\eta}$ splits $\beta_{rbc,\eta}$, $Y_{\eta} = F_{\eta}\brac{\sqrt[\ell]{u_{\eta}}}$ and $D = \beta_{rbc,\eta} + \left(u_{\eta}, \pi_{\eta}\right) \in \Br\left(F_{\eta}\right)$, it is immediate that $\ind\left(D\otimes_F E_{j,\eta}\right) \leq \ell$ and that $D\otimes Y_{\eta} \otimes E_{j,\eta}$ is split. \end{proof}

\begin{proposition}[Yellow/Orange/Red/White]
\label{propositionE-1bSPLIT-Red}
\label{propositionE-1bSPLIT-YellowOrange}
Let $\eta\in N_0$ be coloured yellow, orange or white. Set $E_{2,\eta}$ to be the lift of residues at $\widehat{A_{\eta}}$. Then there exists a ramified cyclic extension $E_{1,\eta}/F_{\eta}$ of degree $\ell$ and elements $a_{1,\eta} = \brac{\tilde{a}_{1,i,\eta}}$ and  $a_{2,\eta} = \brac{\tilde{a}_{2,i,\eta}} \in \prod F_{\eta}$ such that for $j=1,2$, the following holds:
\begin{enumerate}
\item 
$a_{1,\eta}a_{2,\eta} = a \in Y_{\eta}$, i.e. $\tilde{a}_{1,i,\eta}\tilde{a}_{2,i,\eta}= a'_{i,\eta}$ for each $i$.
\item
$E_{j, \eta}$ splits $D$.

\item
$a_{j,\eta}$ is a norm from $E_{j,\eta}\otimes Y_{\eta}/Y_{\eta}$, i.e . $\tilde{a}_{j,i,\eta}$ is a norm from $E_{j,\eta}$ for each $i$.

\item
$\N_{Y_{\eta}/F_{\eta}}\brac{a_{j,\eta}}=1$, i.e. $\prod_i \tilde{a}_{j,i,\eta} = 1$.
\item
$E_{j,\eta}\otimes F_{P,\eta} \simeq E_{j,P} \otimes F_{P,\eta}$ for each point $P\in S_0 \cap \overline{\eta}$.
\item
 $\tilde{a}_{j,i,\eta}{\mu_{j,i,P,\eta}} = \tilde{a}_{j,i,P} \in F_{P,\eta}$ for all $i$ at each point $P\in S_0 \cap \overline{\eta}$ for some $\mu_{j,i,P,\eta}\in F_{P,\eta}^{*\ell}$ such that $\prod_i \mu_{j,i,P,\eta}=1$.
  \item
 $E_{2,\eta}/F_{\eta}$ is unramified and cyclic of degree $\ell$.
\end{enumerate}
\end{proposition}

\begin{proof}By hypothesis, $\eta$ is a Ch/H/C/Z curve of Type 1b with $Y_{\eta}$ of Type SPLIT. The proof is similar to the proof of Proposition \ref{propositionE-1bSPLIT-Blue} (we appeal to Proposition \ref{otherproperties-yellow-orange} to ensure compatibility at branches). \end{proof}

\section{Patching data at uncoloured points of $N_0$}
\label{section-patchingdataatuncolouredN0}
Let $\eta \in N_0$ be of Type 1a or 0 and let $\pi_{\eta}$ be a parameter of $F_{\eta}$ as before.  Set $\mathcal{P}'_{\eta} := \overline{\eta}\cap S_0$. If $\eta$ is of Type 0, set $\mathcal{Q}'_{\eta}:=\emptyset$. If $\eta$ is of Type 1a, set $D' = \overline{D\otimes F_{\eta}}$ over the residue field $k_{\eta}$. Thus $D'$ is a central simple algebra over the global field $k_{\eta}$ of exponent and index dividing $\ell$.  Let $\mathcal{Q}'_{\eta}$ denote the set of closed points $Q\in \overline{\eta}$ not in $\mathcal{P}'_{\eta}$ such that $D'\otimes k_{Q,\eta} \neq 0$.  For $j=1,2$ and any $P\in S_0$, let $E_{j,P}$ and $a_{j,P}$ be as prescribed in Propositions \ref{propositionEatpoints} and \ref{propositionEatpoints-A00}. We now prescribe the choices for $E_{j,\eta}$ and ${a}_{j,\eta}$. Tables \ref{TableEatA10spoint}, \ref{TableEatB10spoint}, \ref{TableEatB11nspoint}, \ref{TableEatB20spoint}, \ref{TableEatB21nspoint}, \ref{TableEatA00spoint-D3}, \ref{TableEatA00spoint-D2} and \ref{TableEatA00spoint-D1} are relevant in this section. 

\begin{proposition}[0/1a-RAM]
\label{propositionE-Ia-RAM}
\label{propositionE-0-RAM}
Let $\eta\in N_0$ be of Type 0 or 1a and let $Y_{\eta}$ be of Type RAM. Set $a_{1,\eta}=a$ and $a_{2,\eta}=1$. Then for $j=1,2$, there exist $E_{j,\eta}/F_{\eta}$, unramified cyclic extensions of degree $\ell$ such that \begin{enumerate}
\item 
$a_{1,\eta}a_{2,\eta}=a$ in $Y_{\eta}$.

\item
$D\otimes {E}_{j,\eta}$ is split. If $\eta$ is of Type 0, then $E_{j,\eta}\simeq \prod F_{\eta}$.

\item
$a_{j,\eta}$ is a norm from $E_{j,\eta}\otimes Y_{\eta}/Y_{\eta}$. 

\item
$\N_{Y_{\eta}/F_{\eta}}\brac{a_{j,\eta}}=1$.
\item
$E_{j,\eta}\otimes F_{P,\eta} \simeq E_{j,P} \otimes F_{P,\eta}$ for each point $P\in S_0 \cap \overline{\eta}$.
\item
 $a_{j,\eta} = a_{j,P} \in Y\otimes F_{P,\eta}$ for each point $P\in S_0 \cap \overline{\eta}$.
\end{enumerate}
\end{proposition}
\begin{proof} 
By Proposition \ref{propaisaunitifYnotsplit} and Lemma \ref{lemmanormoneramified-dim1}, $a \in \mathcal{O}_{Y_{\eta}}^{*\ell}$. By Proposition \ref{otherproperties-1a/0}, $a_{1,P}=a$ and  $a_{2,P} =1$ for each $P\in \mathcal{P}'_{\eta}$. Hence Properties 1, 3, 4 and 6 hold.

\textbf{Let $\eta$ be of Type 0}. Note that by Remark \ref{remark-ram-split-not-intersect}, $\eta$ cannot intersect $\eta'\in N'_0$ with $Y_{\eta'}$ of Type SPLIT. By inspection of the relevant tables, we see that $E_{1,P}=E_{2,P}=\prod F_P$ for any $P\in \mathcal{P}'_{\eta}$. Set $E_{j,\eta} = \prod F_{\eta}$. Hence Properties 2 \& 5 hold in this case.

\textbf{Let $\eta$ be of Type 1a}. Let $j=1$ or $2$. For $P\in \mathcal{P}'_{\eta}$, set $E'_{j,P} := \overline{E_{j,P}\otimes F_{P,\eta}} / k_{P,\eta}$. For $P\in \mathcal{Q}'_{\eta}$, set $E'_{j,P}$ to be the unique unramified field extension of $k_{P,\eta}$ of degree $\ell$.  $D'\otimes E'_{j,P}$ is trivial for all $P\in \mathcal{P}'_{\eta}\cup \mathcal{Q}'_{\eta}$ (cf. Proof of Proposition \ref{propositionEatpoints} for $P\in \mathcal{P}'_{\eta}$ and local class field theory for $P\in \mathcal{Q}'_{\eta}$). Also clearly for each $Q\in \overline{\eta}\setminus \brac{ \mathcal{P}'_{\eta}\cup \mathcal{Q}'_{\eta}}$, $D'\otimes k_{Q,\eta}$ is split already. Set $u'=1$ and use Lemma \ref{lemmauglyinitial} to construct $\overline{E_{j,\eta}}$. Set $E_{j,\eta}/F_{\eta}$ to be the unramified lift of $\overline{E_{j,\eta}}/k_{\eta}$ to see that Properties 2 \& 5 hold. \end{proof}

\begin{proposition}[0/1a-SPLIT]
\label{propositionE-Ia-SPLIT}
\label{propositionE-0-SPLIT}
Let $\eta\in N_0$ be of Type 0 or 1a and let $Y_{\eta}$ be of Type SPLIT. Then for $j=1,2$, there exist $E_{j,\eta}/F_{\eta}$, unramified cyclic extensions of degree $\ell$ and elements $a_{j,\eta} = \brac{\tilde{a}_{j,i,\eta}}_i\in \prod F_{\eta}$ such that 
\begin{enumerate}
\item 
$a_{1,\eta}a_{2,\eta}= a = \brac{a'_{i,\eta}}_i$ in $Y_{\eta}$, i.e ${\tilde{a}_{1,i,\eta}}{\tilde{a}_{2,i,\eta}} = a'_{i,\eta}$ for each i.

\item
$D\otimes {E}_{j,\eta}$ is split. If $\eta$ is of Type 0, then $E_{j,\eta}\simeq \prod F_{\eta}$.

\item
$a_{j,\eta}$ is a norm from $E_{j,\eta}\otimes Y_{\eta}/Y_{\eta}$, i.e. ${\tilde{a}_{j,i,\eta}}$ is a norm from $E_{j,\eta}$ for each i.
\item
$\N_{Y_{\eta}/F_{\eta}}\brac{a_{j,\eta}}=1$, i.e. $\prod {\tilde{a}_{j,i,\eta}} = 1$.
\item
$E_{j,\eta}\otimes F_{P,\eta} \simeq E_{j,P} \otimes F_{P,\eta}$ for each point $P\in S_0 \cap \overline{\eta}$.
\item
 $\tilde{a}_{j,i,\eta}{\mu_{j,i,P,\eta}} = \tilde{a}_{j,i,P} \in F_{P,\eta}$ for all $i$ at each point $P\in S_0 \cap \overline{\eta}$ for some $\mu_{j,i,P,\eta}\in F_{P,\eta}^{*\ell}$ such that $\prod_i \mu_{j,i,P,\eta}=1$.
\end{enumerate}
\end{proposition}

\begin{proof} Let $a = \brac{a'_{i,\eta}}\in \prod {F_{\eta}}$ where $a'_{i,\eta}=x'_{i}\pi_{\eta}^{m_i}$ where $m_i\in \mathbb{Z}$ and $x'_{i}\in \widehat{A_{\eta}}^*$. Because $\N(a)=1$, we have $\sum m_i=0$ and $\prod x'_i = 1$. 

\textbf{Let $\eta$ be of Type 0}. Since $D_{\eta}$ is already split, Property 2 is satisfied. Choose $\{j, j'\}=\{1,2\}$ as in Proposition \ref{otherproperties-1a/0}. 

First let's construct $a_{j,\eta}\in \prod F_{\eta}$ approximating ${a_{j,P}}$ for each $P\in \mathcal{P}'_{\eta}$. By inspection of the relevant tables, we see that $a_{j,P} = \brac{\tilde{a}_{j,i,P}}$ is such that $\tilde{a}_{j,i,P} = x_{i,P}\pi_P^{m_i}$ where $x_{i,P}\in \mathcal{O}_{F_{P,\eta}}^*$. Since $\N\brac{a_{j,P}}=1$, we have $\prod x_{i,P}=1$. For $1\leq i\leq \ell-1$ by weak approximation, find $c_{i}\in k_{\eta}$ which is close to $\overline{x_{i,P}}$ in $k_{P,\eta}$ and let $\tilde{c_i}$ be a lift of $c_i$ in $F_{\eta}$. Let $c_{\ell} = \prod_{r=1}^{\ell-1}\brac{c_r}^{-1}$ and $\tilde{c}_{\ell} = \prod_{r=1}^{\ell-1}\brac{\tilde{c}_r}^{-1}$. Set $\tilde{a}_{j,i,\eta}= \tilde{c_i}\pi_{\eta}^{m_i}$ and $a_{j',\eta} = aa_{j,\eta}^{-1}$. Thus $\tilde{a}_{j',i,\eta}= x'_i\tilde{c_i}^{-1}\in \widehat{A_{\eta}}^*$. Let $c'_i = \overline{\tilde{a}_{j',i,\eta}}\in k_{\eta}$. Therefore Properties 1, 4 \& 6 are immediate.

Set $E_{j,\eta} = \prod F_{\eta}$. Note that by Proposition $\ref{otherproperties-1a/0}$, $E_{j,P}=\prod F_P$. Thus Properties 3 \& 5 are satisfied for $a_{j,\eta}$ and $E_{j,\eta}$. For $P\in \mathcal{P}'_{\eta}$, set $E'_{j',P} := \overline{E_{j',P}\otimes F_{P,\eta}} / k_{P,\eta}$. Let $X' = k_{\eta}\brac{\sqrt[\ell]{c'_1}, \ldots, \sqrt[\ell]{c'_{\ell}}}$. By inspection of the relevant tables, we find that one of the following hold for $P\in \mathcal{P}'_{\eta}$: 
\begin{itemize}
\item[-]{$E'_{j',P}=\prod k_{P,\eta}$} : In this case, set $w'_P=1$.
\item[-] {$E'_{j',P}= k_{P,\eta}\brac{\sqrt[\ell]{\overline{\delta_P}}}$} : In this case, also note that $c'_i =\overline{\delta_P^{n_i}}$. Hence $X'\otimes k_{P,\eta} = k_{P,\eta}\brac{\sqrt[\ell]{\overline{\delta_P}}}$. Set $w'_P = \overline{\delta_P}$. 
\end{itemize}
Thus for $P\in \mathcal{P}'_{\eta}$ we have found $w'_P \in k_{P,\eta}$ so that $E'_{j',P} = k_{P,\eta}[t]/\brac{t^{\ell}-w'_P}$  and $w'_P=\N_{X'\otimes k_{P,\eta}/k_{P,\eta}}\brac{z'_P}$ for suitable elements $z'_P$. By weak approximation, we can find $z'\in X'$ close to $z'_P$. Let $w' =\N(z')$. Set $\overline{E_{j',\eta}} =k_{\eta}[t]/(t^{\ell}-w')$. Thus $\brac{w', c'_i}=0$ in $\Br\brac{k_{\eta}}$ for all $i$. Let $E_{j',\eta}$ be the unramified lift of $\overline{E_{j',\eta}}$. This shows that Properties 3 \& 5 hold for $a_{j',\eta}$ and $E_{j',\eta}$ as well. 

\textbf{Let $\eta$ be of Type 1a}. By Lemma \ref{lemmaalong1a} we have that $m_i = \ell m'_i$ and $a'_{i,\eta} = {x'_{i}\pi_{\eta}^{\ell m'_i}}$. For  $Q\in \mathcal{Q}'_{\eta}$, set $a_{1,Q}=a$ and $a_{2,Q}=1$. Since $a$ is arranged to be in good shape (Proposition \ref{propaisingoodshape2}), we see that at these points $a_{1,Q}=\brac{x_{i,Q}\pi_Q^{\ell m'_i}}$ where $\pi_Q$ is some prime in a regular system of parameters defining $\eta$ at $Q$ and $x_{i,Q}\in \widehat{A_Q}^*$.

First let's construct $a_{1,\eta}\in \prod F_{\eta}$ approximating $a_{1,P}$ for each $P\in \mathcal{P}'_{\eta} \cup \mathcal{Q}'_{\eta}$. By the above discussion and inspection of the relevant tables, we see that in $\prod F_{P,\eta}$, $a_{1,P} = \brac{x_{i,P}}_i$ or $\brac{x_{i,P}\pi_P^{\ell m'_i}}_i$ where $x_{i,P}\in \widehat{A_{P,\eta}}^*$. Since $\N\brac{a_{1,P}}=1$, we have $\prod x_{i,P}=1$. For $1\leq i\leq \ell-1$, by {weak approximation}, find $c_{i}\in k_{\eta}$ which is close to $\overline{x_{i,P}}$ in $k_{P,\eta}$ and let $\tilde{c_i}$ be a lift of $c_i$ in $F_{\eta}$. Let $c_{\ell} = \prod_{r=1}^{\ell-1}\brac{c_r}^{-1}$, $\tilde{c}_{\ell} = \prod_{r=1}^{\ell-1}\brac{\tilde{c}_r}^{-1}$ and $c'_r = \overline{x'_r\tilde{c_r}^{-1}}$ for $r\leq \ell$. Let $a_{1,\eta} = \brac{\tilde{c_i}}$ and let $a_{2,\eta} = aa_{1,\eta}^{-1}$. Thus Properties 1, 4 \& 6 are immediate.

Let $j=1$ or $2$. For $P\in \mathcal{P}'_{\eta}$, set $E'_{j,P} := \overline{E_{j,P}\otimes F_{P,\eta}} / k_{P,\eta}$. For $P\in \mathcal{Q}'_{\eta}$, set $E'_{j,P}$ to be the unique unramified field extension of $k_{P,\eta}$. $D'\otimes E'_{j,P}$ is trivial for all $P\in \mathcal{P}'_{\eta}\cup \mathcal{Q}'_{\eta}$ (cf. Proof of Proposition \ref{propositionEatpoints} for $P\in \mathcal{P}'_{\eta}$ and local class field theory for $P\in \mathcal{Q}'_{\eta}$). Also clearly for each $Q\in \overline{\eta}\setminus \brac{ \mathcal{P}'_{\eta}\cup \mathcal{Q}'_{\eta}}$, $D'\otimes k_{Q,\eta}$ is split already.

Let $X = k_{\eta}\brac{\sqrt[\ell]{c_1}, \ldots, \sqrt[\ell]{c_{\ell}}}$ and let $X' = k_{\eta}\brac{\sqrt[\ell]{c'_1}, \ldots, \sqrt[\ell]{c'_{\ell}}}$. Then $X\otimes_{k_{\eta}} k_{P,\eta}\simeq \prod_{i=1}^{g} X_P$ (resp. $X'\otimes_{k_{\eta}} k_{P,\eta}\simeq \prod_{j=1}^{h} X'_P$) where  $X_P$ and $X'_P$ are as in Proposition \ref{prop-points-normsfromGaloisclosures-1a-split} if $P\in \mathcal{Q}'_{\eta}$ and are unramified field extensions\footnote{since $x_{i,P}\in \widehat{A_P}^*$ at these points.} of $k_{P,\eta}$ if $P\in \mathcal{Q}'_{\eta}$. 

For each $P\in \mathcal{Q}'_{\eta}$, let $w_P = w'_P$ be any unit in $\mathcal{O}_{k_{P,\eta}}$ which is not an $\ell^{\mrm{th}}$ power, which therefore are norms from unramified extensions $X_P$ and $X'_P$ respectively. For each $P\in \mathcal{P}'_{\eta}$ choose $w_P, w'_P \in k_{P,\eta}$ as in Proposition \ref{prop-points-normsfromGaloisclosures-1a-split}. Thus for $P\in  \mathcal{P}'_{\eta}\cup \mathcal{Q}'_{\eta}$, we have $E'_{1,P} = k_{P,\eta}[t]/\brac{t^{\ell}-w_P}$, $E'_{2,P} = k_{P,\eta}[t]/\brac{t^{\ell}-w'_P}$ with $w_P = \N_{X\otimes k_{P,\eta}/k_{P,\eta}}\brac{z_P}$ and $w'_P=\N_{X'\otimes k_{P,\eta}/k_{P,\eta}}\brac{z'_P}$ for suitable elements $z_P$ and $z'_P$. By weak approximation, we can find $z\in X$ and $z'\in X'$ close to $z_P$ and $z'_P$ respectively. Let $w = \N(z)$ and $w' =\N(z')$. Set $E'_{1} =k_{\eta}[t]/(t^{\ell}-w)$ and $E'_{2} =k_{\eta}[t]/(t^{\ell}-w')$. Thus $\brac{w, c_i}=0$ and $\brac{w', c'_i}=0$ in $\Br\brac{k_{\eta}}$ for all $i$. Let $E_{j,\eta}$ be unramified lifts of $E'_{j}$. These extensions approximate $E_{j,P}$ and Properties 2, 3 \& 5 hold. \end{proof}

\begin{proposition}[0/1a-NONRES]
\label{propositionE-0-NONRES}
\label{propositionE-Ia-NONRES}
Let $\eta\in N_0$ be of Type 0 or 1a and let $Y_{\eta}$ be of Type NONRES. Then for $j=1,2$, there exist $E_{j,\eta}/F_{\eta}$, unramified cyclic extensions of degree $\ell$ and elements $a_{j,\eta}\in \mathcal{O}_{Y,\eta}$ such that 
\begin{enumerate}
\item 
$a_{1,\eta}a_{2,\eta}=a$ in $Y_{\eta}$.
\item
$D_{\eta}$ is split if $\eta$ is of Type 0. Else $D\otimes {E}_{j,\eta}$ has index at most $\ell$ and $D\otimes Y\otimes E_{j,\eta}$ is split. Further if $\eta$ is of Type 0, then $E_{j,\eta}\simeq \prod F_{\eta}$.
\item
$a_{j,\eta}$ is a norm from $E_{j,\eta}\otimes Y_{\eta}/Y_{\eta}$. 
\item
$\N_{Y_{\eta}/F_{\eta}}\brac{a_{j,\eta}}=1$.
\item
$E_{j,\eta}\otimes F_{P,\eta} \simeq E_{j,P} \otimes F_{P,\eta}$ for each point $P\in S_0 \cap \overline{\eta}$.
\item
 $a_{j,\eta}{\mu_{j,P,\eta}} = a_{j,P}$ in $Y\otimes F_{P,\eta}$ for each point $P\in S_0 \cap \overline{\eta}$ for some $\mu_{j,P,\eta}\in  \mathcal{O}_{Y\otimes F_{P,\eta}}$ such that $\mu_{j,P,\eta} \cong 1$ mod  $\brac{\pi_{\eta}}$ and $\N\brac{\mu_{j,P,\eta}}=1$. 
\end{enumerate}
\end{proposition}

\begin{proof} By Proposition \ref{propaisaunitifYnotsplit} we have that $a\in \mathcal{O}_{Y,\eta}^*$.  Let $Y' = \overline{Y_{\eta}} = k_{\eta}\brac{\sqrt[\ell]{u'}}$, $a'=\overline{a} \in Y'$ and $\Gal\brac{Y'/k_{\eta}} = \langle \psi \rangle$. For  $P\in \mathcal{Q}'_{\eta}$, set $a_{1,P}=a$ and $a_{2,P}=1$. Further since $a$ is in good shape and $P\not\in S_0$, we see that $a_{j,P}$ are units along $\eta$ and further, $\overline{a_{j,P}} \in \mathcal{O}_{Y'_{P}}^*$. Since $Y$ is also arranged to be in good shape, $Y'_P$ is an unramified (possibly split) extension of $k_{P,\eta}$. By inspecting the relevant tables, we see that $a_{j,P}$ are units along $\eta$ for $P\in \mathcal{P}'_{\eta}$ also. 

First let's construct $a'_1\in Y'$ approximating $\overline{a_{1,P}}\in \overline{Y\otimes F_{P,\eta}}$ for each $P\in \mathcal{P}'_{\eta}\cup \mathcal{Q}'_{\eta}$. Since $\N\brac{\overline{a_{1,P}}}=1$, by Hilbert 90 there exists $c_{P}\in Y'\otimes k_{P,\eta}$ such that $c_P^{-1}\psi\brac{c_P}=\overline{a_{1,P}}$. Using weak approximation, find $c\in Y'$ which is close to $c_P$ for each $P$. Set $a'_1 = c^{-1}\psi(c)$ and set $a'_2 = a'{a'_1}^{-1}$. Let $a_{1,\eta}$ denote a lift of $a'_1$ and let $a_{2,\eta} = aa_{1,\eta}^{-1}$. Then Properties  1, 4 \& 6 are immediate. 

\textbf{Let $\eta$ be of Type 0}. Property 2 is satisfied by the definition of Type $0$. Choose $\{j, j'\} = \{1, 2\}$ as in Proposition \ref{otherproperties-1a/0}. Set $E_{j,\eta} = \prod F_{\eta}$. Since by the same proposition, $E_{j,P}=\prod F_P$, Properties 3 \& 5 are satisfied for $a_{j,\eta}$ and $E_{j,\eta}$. 

For $P\in \mathcal{P}'_{\eta}$, set $E'_{j',P} := \overline{E_{j',P}\otimes F_{P,\eta}} / k_{P,\eta}$. Let $\tL'$ denote the Galois closure of $Y'(\sqrt[\ell]{a'_{j'}})$. Letting $D'=0\in \Br\brac{k_{\eta}}$, we would like to apply Lemma \ref{lemmaugly} to construct $E'_{j'} = \overline{E_{j',\eta}}$ first. Thus for each $P\in \mathcal{P}'_{\eta}$,  we would like to first find $w'_P \in k_{P,\eta}$ and $z'_P\in \tL'\otimes k_{P,\eta}$ so that $E'_{j',P} = k_{P,\eta}[t]/\brac{t^{\ell}-w'_P}$ and $\N_{\tL'\otimes k_{P,\eta}/k_{P,\eta}}\brac{z'_P} = w'_P$. 
 
By inspection of the relevant tables, we find that one of the following hold : 
\begin{itemize}
\item[-]{$E'_{j',P}=\prod k_{P,\eta}$} : In this case, set $w'_P=1$.
\item[-]{$E'_{j',P}= k_{P,\eta}\brac{\sqrt[\ell]{\overline{\delta_P}}}$} : In this case, also note that $a_{j',P} = \brac{\overline{\delta_P^{n_i}}}_i$. Hence {$\tL'\otimes k_{P,\eta} = k_{P,\eta}\brac{\sqrt[\ell]{\overline{\delta_P}}}$}. So set $w'_P = \overline{\delta_P}$
\end{itemize}

Similarly it is an immediate check that $\brac{w'_P, a_{j'}}_y = 0$ for every valuation $y\in \Omega_{Y'}$ lying over $P$. Thus Lemma \ref{lemmaugly} can be used to construct $E'_{j'}$. Setting $E_{j',\eta}$ to be its unramified lift over $F_{\eta}$, it is immediate that Properties 3 \& 5 are satisfied. 

\textbf{Let $\eta$ be of Type 1a}. Let $j=1$ or $2$.  We would like to apply Lemma \ref{lemmaugly} to first construct $E'_j = \overline{E_{j,\eta}}$. For $P\in \mathcal{P}'_{\eta}$, set $E'_{j,P} := \overline{E_{j,P}\otimes F_{P,\eta}} / k_{P,\eta}$. For $P\in \mathcal{Q}'_{\eta}$, set $E'_{j,P}$ to be the unique unramified field extension of $k_{P,\eta}$ of degree $\ell$.

$D'\otimes Y' \otimes E'_{j,P}$ is trivial for all $P\in \mathcal{P}'_{\eta}\cup \mathcal{Q}'_{\eta}$ (cf. Proof of Proposition \ref{propositionEatpoints} for $P\in \mathcal{P}'_{\eta}$ and local class field theory for $P\in \mathcal{Q}'_{\eta}$). Also clearly for each $Q\in \overline{\eta}\setminus \brac{ \mathcal{P}'_{\eta}\cup \mathcal{Q}'_{\eta}}$, $D'\otimes k_{Q,\eta}$ is split already.

Let $\tL$ denote the Galois closure of $Y'(\sqrt[\ell]{a_1'})$ and let $\tL'$ denote the Galois closure of $Y'(\sqrt[\ell]{a_2'})$. Note that whenever $Y'\otimes k_{Q,\eta}$ is not split, since $a'_1$ and $a'_2$ have norm 1, they are also units (and in fact $\ell^{\mrm{th}}$ powers) in the complete discretely valued field $Y'\otimes k_{Q,\eta}$ for every $Q\in \overline{\eta}$ by Lemmata \ref{lemmanormoneramified-dim1} and \ref{lemmanormoneunramified}. Then as in the previous proof, $\tL\otimes_{k_{\eta}} k_{P,\eta}\simeq \prod_{i=1}^{g} \tL_P$ (resp. $\tL'\otimes_{k_{\eta}} k_{P,\eta}\simeq \prod_{j=1}^{h} \tL'_P$) where  $\tL_P$ and $\tL'_P$ are as in Proposition \ref{prop-points-normsfromGaloisclosures-1a} if $P\in \mathcal{P}'_{\eta}$ and are unramified field extensions\footnote{By the remark in the beginning of this proof, $Y'_P$ is an unramified (possibly split) extension and $a'_1$ and $a'_2$ are units at these points} if $P\in \mathcal{Q}'_{\eta}$.

For each $P\in \mathcal{Q}'_{\eta}$, let $w_P = w'_P$ be any unit in $\mathcal{O}_{k_{P,\eta}}$ which is not an $\ell^{\mrm{th}}$ power, which therefore are norms from unramified extensions $\tL_P$ and $\tL'_P$ respectively. For each $P\in \mathcal{P}'_{\eta}$ choose $w_P, w'_P \in k_{P,\eta}$ as in Proposition \ref{prop-points-normsfromGaloisclosures-1a}. Thus for $P\in  \mathcal{P}'_{\eta}\cup \mathcal{Q}'_{\eta}$, we have $E'_{1,P} = k_{P,\eta}[t]/\brac{t^{\ell}-w_P}$, $E'_{2,P} = k_{P,\eta}[t]/\brac{t^{\ell}-w'_P}$ and $\N_{\tL_P/k_{P,\eta}}\brac{z_P} = w_P$, $\N_{\tL'_P/k_{P,\eta}}\brac{z'_P} = w'_P$for suitable elements $z_P$ and $z'_P$.

We would like to use a modified version of Lemma \ref{lemmaugly}. Let $P\in \mathcal{P}'_{\eta}\cup \mathcal{Q}'_{\eta}$. By the above discussion, we can find $\tilde{z}_P\in \tL\otimes k_{P,\eta} =\prod \tL_P$ (resp $\tilde{z}'_P\in \tL'\otimes k_{P,\eta} =\prod \tL'_P$) such that its norm to $k_{P,\eta}$ is $w_P$ (resp $w'_P$). We verify that $\brac{w_P, a'_1}_y = \brac{w'_P, a'_2}_y = 0$ for every valuation $y\in \Omega_{Y'}$ lying over $P$. For $P\in \mathcal{P}'_{\eta}$, this is by Proposition \ref{propositionEatpoints} (4). For $P\in \mathcal{Q}'_{\eta}$, by construction $a'_1$ is a unit in $\mathcal{O}_{Y'_P}^*$ while $a'_2 = \brac{1}$. Since $w_P\in \mathcal{O}_{k_{P,\eta}}^*$ also, we are done in this case. Thus Lemma \ref{lemmaugly} can be used to construct $E'_1$ and $E'_2$ over $k_{\eta}$ though $E'_i$ will not split $D'$. Setting $E_{1,\eta}$ and $E_{2,\eta}$ to be their respective unramified lifts over $F_{\eta}$, we see that Properties 2, 3 \& 6 are satisfied. \end{proof}

\section{Spreading and patching of $E_j$ and $a_j$}
Recall that $F=K(X)$ is the function field of a smooth projective geometrically integral curve $X$ over a $p$-adic field $K$ and $\mathcal{X}\to \Spec \mathcal{O}_K$, a normal proper model of $F$ as fixed in Section \ref{subsectionfinalmodel}. We recall some further notation from (\cite{HH}, Section 6) and (\cite{HHK09}, Section 3.3). Let $\eta\in N_0$ and let $U_{\eta}\subset X_0$ be a non-empty open subset containing $\eta$. Then $A_{U_{\eta}}$ denotes the ring of functions regular on $U_{\eta}$. Fix a parameter $t$ of $K$. Thus $t\in A_{U_{\eta}}$. Then $\widehat{A_{U_{\eta}}}$ denotes the completion of $A_U$ at ideal $(t)$ and $F_U$, the fraction field of $\widehat{A_{U_{\eta}}}$. Further $F\subseteq F_{U_{\eta}}\subseteq F_{\eta}$. Let $\pi_{\eta}$ be the parameter of $F_{\eta}$ fixed as in Section \ref{subsectionfinalmodel}.

\begin{proposition}
\label{propositionEaatUsinpatching}
Let $j=1$ or $2$. For each $\eta\in N_0$, there exist a neighbourhood $U_{\eta}$ of $\eta$ in $X_0$ such that $U_{\eta}\subseteq \overline{\eta}\setminus S_0$, elements $a_{j,U_{\eta}} \in Y\otimes F_{U_{\eta}}$ and cyclic or split extensions $E_{j,U_{\eta}}/F_{U_{\eta}}$ of degree $\ell$ such that
\begin{enumerate}
\item
$a_{1, U_{\eta}}a_{2,U_{\eta}}=a$.
\item
$D\otimes E_{j,U_{\eta}}$ has index dividing $\ell$.
\item
$D\otimes Y\otimes E_{j,U_{\eta}}$ is split.
\item
$a_{j,U_{\eta}}$ is a norm from $Y\otimes E_{j,U_{\eta}}/Y\otimes F_{U,{\eta}}$. 
\item
$\N_{Y\otimes F_{U_{\eta}}}\brac{a_{j, U_{\eta}}} = 1$.
\item
$E_{j,U_{\eta}}\otimes F_{\eta} \simeq E_{j,\eta}$.
\item
$E_{j,U_{\eta}}\simeq \prod F_{U_{\eta}}$ whenever $E_{j,{\eta}}\simeq \prod F_{\eta}$ .
\item
$E_{j,U_{\eta}}\simeq F_{U_{\eta}}[t]/\brac{t^{\ell}-e_{j,U_{\eta}}}$ for some $e_{j,U_{\eta}}\in \widehat{A_{U_{\eta}}}$ . Further if $E_{j,\eta}$ is unramified, then $e_{j,U_{\eta}}\in  \widehat{A_{U_{\eta}}}^*$.
\item
$a_{j,U_{\eta}}v_{j,\eta}^{\ell}  = a_{j, \eta} \in Y\otimes F_{\eta}$ for some $v_{j,\eta}\in Y\otimes F_{\eta}$ of norm one.
\item
$D\otimes F_{U_{\eta}}$ is split whenever $D\otimes F_{{\eta}}$ is split.
\item
$D\otimes E_{j,U_{\eta}}$ is split whenever $D\otimes E_{j,{\eta}}$ is split.

\end{enumerate} 
\end{proposition}

\begin{proof} By the propositions in Section \ref{sectionpatchingdataatN0} and \ref{section-patchingdataatuncolouredN0}, we know $D\otimes E_{j,\eta}\otimes Y$ is split and $D\otimes E_{j,\eta}$ has index dividing $\ell$. Further, we also know $E_{j,\eta} = F_{\eta}[t]/(t^{\ell}-e'_{j,\eta})$ where $e'_{j,\eta} = \pi_{\eta}^{\epsilon_j}e_{j,\eta}$ with $\epsilon_j \in \{0, 1\}$ and $e_{j,\eta}\in \widehat{A_{\eta}}^*$. Finally we have norm one elements $a_{j,\eta}\in Y_{\eta}$ such that $a_{1,\eta}a_{2,\eta}=a$ and $\brac{a_{j,\eta}, e'_{j,\eta}}=0\in \Br(Y_{\eta})$.

For $d=2$ or $\ell+1$, by (\cite{HHK15}, Proposition 5.8) and (\cite{KMRT}, Proposition 1.17), there exists non-empty open set $V'_{\eta}$ of $\eta$ such that $D\otimes F_{V'_{\eta}}$ (resp. $D\otimes Y \otimes F_{V'_{\eta}}$) has index $< d$ whenever $D\otimes F_{\eta}$ (resp $D\otimes Y\otimes F_{\eta}$) has index $<d$. If $E_{j,\eta}\simeq \prod F_{\eta}$, set $V_{j,\eta}:=V'_{\eta}$. Set $E_{V_{j,\eta}}=\prod F_{V_{j,\eta}}$ and $e_{j,V_{j,\eta}}=1$. Then Properties 2, 3, 6, 7, 8 , 10 \& 11 clearly hold.

If $E_{j,\eta}/F_{\eta}$ is a field extension, choose $e_j\in F^*$ such that $e_j^{-1}e_{j,\eta}$ is $1$ mod  $(\pi_{\eta})$ in $\widehat{A_{\eta}}$. Set $e'_j = \pi_{\eta}^{\epsilon_j}e_j\in F^*$ and $E'_j = F[t]/(t^{\ell}-e'_j)$. Since $e'_j=e'_{j,\eta}x^{\ell}$ for some $x\in \widehat{A_{\eta}}^*$, $E'_j\otimes F_{\eta}\simeq E_{j,\eta}$. Now again by (\cite{HHK15}, Proposition 5.8) and (\cite{KMRT}, Proposition 1.17), for $d = 2$ or $\ell+1$, there exists non-empty open set $V_{j,\eta}\subseteq V'_{\eta}$ of $\eta$ such that $D\otimes E'_j\otimes F_{V_{j,\eta}}$ (resp. $D\otimes Y \otimes E'_j  \otimes F_{V_{j,\eta}}$) has index $< d$ whenever $D\otimes E'_j\otimes F_{\eta}$ (resp $D\otimes Y\otimes E'_j\otimes F_{\eta}$) has index $<d$. Setting $E_{j,V_{j,\eta}}:= E'_j\otimes F_{V_{j,\eta}}$, it is clear that Properties $2,3,6,7, 10$ and $11$ hold. Shrink $V_{j,\eta}$ further to assume $e_j\in \widehat{A_{V_{j,\eta}}}^*$. Setting $e_{V_{j,\eta}}:= e'_j$, it is clear that Property $8$ holds. Shrink $V_{1,\eta}$ and $V_{2,\eta}$ to assume they are both equal and call them $V_{\eta}$. To address Properties $1, 4, 5$ and $9$, we distinguish between the cases when $Y_{\eta}/F_{\eta}$ is a field extension and when $Y_{\eta}\simeq \prod F_{\eta}$.

\textit{Suppose that $Y_{\eta}/F_{\eta}$ is a field extension}: Let $F^h_\eta$ be the henselization of $F$ at the discrete valuation $\eta$. Set $Y^h_{\eta} = Y\otimes_{F}{F^h_\eta}$ and identify it as a subfield of $Y_{\eta}$ via the canonical morphism $Y^h_{\eta}\to Y_{\eta}$. Let $\tilde{\pi}_\eta \in Y^h_\eta$ be a parameter. Then $\tilde{\pi}_\eta $ is also a parameter in $Y_\eta$. Since $\N_{Y_{\eta}/F_{\eta}}\brac{a_{1,\eta}}=1$, by Hilbert 90, let $a_{1,\eta} = b_{1,\eta}^{-1}\psi\brac{b_{1,\eta}}$ for some $b_{1,\eta}\in Y_{\eta}^*$. Write $b_{1,\eta}  = u_\eta\tilde{\pi}_\eta^r$ for some $u_\eta \in Y_\eta$ which is a unit at $\eta$. Since $u_\eta \in Y_\eta$ is a unit at $\eta$, by (\cite{Artin}, Theorem 1.10), there exists $u^h_\eta \in Y^h_\eta$ such that $u_\eta^h \equiv u_\eta$ modulo the maximal ideal of valuation ring of $Y_\eta$. Let $b^h_{1,\eta} = u^h_\eta \tilde{\pi}_\eta^r \in Y^h_\eta$. Set $a^h_{1,\eta}=  \brac{b^h_{1,\eta}}^{-1}\psi\brac{b^h_{1,\eta}}$. Thus $\N(a^h_{1,\eta})=1$ and $a^h_{1,{\eta}}a_{1, \eta}^{-1}$ is a norm one element in $Y_{\eta}$ which is $1$ modulo the maximal ideal of valuation ring of $Y_\eta$. Thus by Lemma \ref{lemmasimultaneousnormoneandlthpower} again, $a^h_{1,{\eta}}\brac{v_{1,\eta}}^{\ell}  = a_{1, \eta}$ for some $v_{1,\eta}\in Y_{\eta}$ of norm one. Thus $\brac{a_{1,\eta}, e'_1} = \brac{a^h_{1,\eta}, e'_1} = 0 \in \Br(Y_{\eta})$ and we have $\brac{a^h_{1,\eta}, e'_1} = 0 \in \Br(Y^h_{\eta}) $ (cf. proof of (\cite{HHK14}, Proposition 3.2.2)).

Since $F^h_\eta$ is the filtered  direct limit of the fields $F_V$, where $V$  ranges over the non-empty open subset of $\eta$ (\cite{HHK14}, Lemma 2.2.1), there exist a non-empty open subset  $U_\eta\subseteq V_{\eta}$ of $\eta$ and $a_{1, U_\eta} \in Y\otimes F_{U_\eta}$ such that $N_{Y\otimes F_{U_\eta}/F_{U_\eta}}(a_{1,U_\eta}) =1$ and the image of $a_{1, U_{\eta}}$ in $Y^h_{\eta}$ is equal to $a^h_{1,\eta}$.

By  shrinking $U_\eta$, we can assume that $\brac{a_{1,U_{\eta}}, e'_1}  = 0 \in \Br\brac{Y_{U_\eta}}$ (\cite{HHK14}, Proposition 3.2.2). Hence Property $4$ holds for $a_{1, U_{\eta}}$. Finally set $a_{2,U_{\eta}}=aa_{1,U_{\eta}}^{-1}$. Thus for $j=1$ and $2$, it is clear that Properties $1, 5$ and $9$ are satisfied. Since $\brac{a_{2,\eta}, e'_2}= \brac{a_{2, U_{\eta}}, e'_2}=0\in \Br(Y_{\eta})$, by using (\cite{HHK14}, Proposition 3.2.2) again and shrinking $U_{\eta}$, we can show that Property $4$ holds for $a_{2,U_{\eta}}$ also. 

\textit{Suppose that $Y_{\eta}$ is split}:  Then shrink $V_{\eta}$ further such that $Y\otimes F_{V_{\eta}}\simeq \prod F_{V_{\eta}}$ also (\cite{Artin}, Theorem 1.10 \& \cite{HHK14}, Lemma 2.2.1). We have $a_{1, \eta} = (\tilde{a}_{1,i,\eta})_{i\leq \ell}$ where $\tilde{a}_{1,i,\eta} = c_{i,\eta}\pi_{\eta}^{m_i}\in F_{\eta}$ for $m_i\in \mathbb{Z}$ and $c_{i,\eta}\in \widehat{A_{\eta}}^*$. For $1\leq i\leq \ell-1$, choose $c_i\in F^*$ such that $c_i^{-1}c_{i,\eta}$ is $1$ mod $(\pi_{\eta})$ in $\widehat{A_{\eta}}$.

Set $\tilde{a}_{1,i,V_{\eta}} = c_i\pi_{\eta}^{m_i}$ for $i\leq \ell-1$ and set $\tilde{a}_{1,V_{\eta}, \ell} = \brac{\prod_{r=1}^{\ell-1}\tilde{a}_{1,r,V_{\eta}, }}^{-1}$. Finally set $a_{1,V_{\eta}}:= \brac{\tilde{a}_{1,i,V_{\eta}}}_i$ and $a_{2,V_{\eta}} = \brac{\tilde{a}_{2,i,V_{\eta}}}_i = \brac{a_{1,V_{\eta}}}^{-1}$ in $\prod F_{V_{\eta}}$.  Thus Properties $1$, $5$ and $9$ (using\footnote{We note that there exists an $m\geq 1$ such that $F_{\eta}$ doesn't contain a primitive $\ell^m$th root of unity.  Look at any branch field $F_{\eta}\subset F_{P,\eta}$ and observe that its residue field $k_{P,\eta}$ has a further residue field $k_P$ which is a finite field of characteristic not $\ell$). } Lemma \ref{lemmasimultaneousnormoneandlthpower}) are satisfied. Let $j=1$ or $2$ and $i\leq \ell$. Since $\brac{a_{j,\eta}, e'_j}=0\in \Br\brac{Y_{\eta}}$, we have $\brac{\tilde{a}_{j, i, V_{\eta}}, e'_j}=0\in \Br(F_{\eta})$. By (\cite{HHK14}, Proposition 3.2.2) and shrinking further if neccesary, there exists a neighbourhood $U_{\eta}\subseteq V_{\eta}$ of $\eta$ such that $\brac{\tilde{a}_{j, i, V_{\eta}}, e'_j}=0\in \Br(F_{U_{\eta}})$ which shows that Property 4 holds.

\end{proof}
 
Recall that $\mathcal{P}'_{\eta}$ denotes the finite set of marked closed points in $\overline{\eta}\cap S_0$.  For each $\eta$ in $N_0$, choose $U_{\eta}$ as in Proposition \ref{propositionEaatUsinpatching} and let $\mathcal{R}'_{\eta}$ denote the finite set of closed points $\brac{\overline{\eta}\setminus U_{\eta}}\setminus S_0$. 

\begin{proposition}
\label{propositionEaatnewpoints} Let $j=1$ or $2$ and let $\eta\in N_0$. For each $P\in \mathcal{P}'_{\eta}\cup \mathcal{R}'_{\eta}$, there exist elements $a_{j,P} \in Y_{P}$ and cyclic or split extensions $E_{j,P}/F_{P}$ of degree $\ell$ such that
\begin{enumerate}
\item
$a_{1, P}a_{2,P}=a$.
\item
$D\otimes E_{j,P}$ has index at most $\ell$.
\item
$D\otimes Y\otimes E_{j,P}$ is split.
\item
$a_{j,P}$ is a norm from $Y\otimes E_{j,P}/Y\otimes F_P$. 
\item
$\N_{Y\otimes F_{P}}\brac{a_{j, P}} = 1$.
\item
$E_{j,U_{\eta}}\otimes F_{P,\eta} \simeq E_{j,P}\otimes F_{P,\eta}$.
\item
$E_{j,P}\simeq \prod F_P$ or $D\otimes E_{j,P}$ is split.
\item
There exists ${\mu_{j, P, \eta}}\in \brac{Y\otimes F_{P,\eta}}^*$ such that $a_{j,P} = a_{j,\eta}{\mu_{j,P,\eta}}$ where $\N\brac{\mu_{j,P,\eta}}=1$ and 
\begin{itemize}
\item[-] $\mu_{j,P,\eta} =1$ if $Y_{\eta}$ is of Type RAM.
\item[-] $\mu_{j,P,\eta} = \brac{\mu_{j,i,P,\eta}}_i$ for $i\leq \ell$ where $\mu_{j,i,P,\eta}\in F_{P,\eta}^{*\ell}$ if $Y_{\eta}$ is of Type SPLIT.
\item[-] $\mu_{j,P,\eta} \cong 1 \mod \brac{\pi_{\eta}}$ if $Y_{\eta}$ is of Type RES/NONRES.
\end{itemize} \end{enumerate} \end{proposition}

\begin{proof} If $P\in \mathcal{P}'_{\eta}$, the proof follows from Propositions \ref{propositionEatpoints}, \ref{propositionEatpoints-A00} and those in Section \ref{sectionpatchingdataatN0}. Assume therefore that $P\in \mathcal{R}'_{\eta}$, i.e. it is a curve point. For $j=1,2$, let $E_{j,\eta} = F_{\eta}[t]/\brac{t^{\ell}-e_{j,\eta}}$ where $e_{j,\eta}\in F_{\eta}$ with $v_{\eta}\brac{e_{j,\eta}} = 0$ or $1$. Let $\brac{\pi_P, \delta_P}$ denote a system of regular parameters at $A_P$ such that $D_P = \brac{u_P,\pi_P}$ where $u_P\in\widehat{A_P}^*$ (\cite{S97}). Let $\pi_{\eta} = \theta_P\pi_P$ in $F_{\eta}$ where $\theta_P\in \widehat{A_{\eta}}^*$. Let $e_{j,\eta} = x_{j,\eta}\pi_{P}^{\epsilon_j}\in F_{P,\eta}$ where $x_j\in \widehat{A_{P,\eta}}^*$ and $\epsilon_j\in \{0,1\}$. Let $\overline{x_{j,\eta}} = y_j\overline{\delta_P}^{r_j}$ up to $\ell^{\mrm{th}}$ powers where $y_j\in \mathcal{O}_{k_{P,\eta}}^*$ and $0 \leq r_j < \ell$.

Let $\tilde{y_j}\in \widehat{A_P}^*$ be such that it matches with $\overline{y_j}$ in $k_P$. Set $e_{j,P} = \tilde{y_j}\delta_P^{r_j}\pi_P^{\epsilon_j}$ and $E_{j,P} = F_P[t]/\brac{t^{\ell}-e_{j,P}}$. Using Proposition \ref{propositionEaatUsinpatching}, Property 6 is satisfied. Also note that by (\cite{S97}), Property 2 is satisfied. As $Y$ is arranged to be in good shape, $Y_P$ is unramified or $Y_P = F_P\brac{\sqrt[\ell]{v_P\pi_P}}$ where $v_P\in \widehat{A_P}^*$. Thus if $Y_P$ is not split, then $D\otimes Y_P$ is already split.
 
As $D$ is ramified at most along $\eta$ at $P$, this implies by Lemma \ref{lemmasplittingfields2} that if $D\otimes E_{j,P}\otimes F_{P,\eta}$ is split, then so is $D\otimes E_{j,P}$. Using Propositions in Section \ref{sectionpatchingdataatN0}, it is clear that $D\otimes Y\otimes E_{j,\eta}$ is split and hence so is $D\otimes Y\otimes E_{j, P}\otimes F_{P,\eta}$. If $Y_P$ is split, therefore we see that $D\otimes E_{j,P}\otimes F_{P,\eta}$ is split. Hence Property 3 is satisfied.

Note that Property 7 holds in the following situations:
\begin{itemize}
\item[-] $\eta$ is of Type 0 or 1a. This is because $D = 0 \in \Br(F_P)$ already. 
\item[-] $E_{j,P}=\prod F_P$.
\item[-] $D\otimes E_{j,P}\otimes F_{P,\eta}$ is split.
\item[-] $Y_{\eta}$ is of Type SPLIT as the discussion above shows.
\item[-] $E_{j,P} = L_P$, the unique field extension of $F_P$ of degree $\ell$ unramified at $\widehat{A_P}$. This is because $u_P$ becomes an $\ell^{\mrm{th}}$ power in $E_{j,P}$. 
\end{itemize}

Recall that we choose $E_{j,\eta}$ to be ramified along $\eta$ in some cases only when $Y_{\eta}$ is SPLIT, where Property 7 already holds. Thus to check that this Property holds in general, we have to investigate only the cases when $E_{j,P} = F_P[t]/(t^{\ell}- \tilde{y_j}\delta_P^{r_j})$ where $0< r_j <\ell$.

We now discuss the proof of the rest of the Properties 1-8 depending on the type of $Y_{\eta}$. 

\textit{$Y_{\eta}$ is RAM}: Thus $\eta$ can be of  Type 0, 1, 1b and coloured green or 2. Set $a_{1,P}=a$ and $a_{2,P}=1$ to see Properties 1, 5 \& 8 hold by construction (Propositions \ref{propositionE-1bRAM} and \ref{propositionE-Ia-RAM}). By Lemma \ref{lemmanormfromE-nonsplit} and Proposition \ref{propaisaunitifYnotsplit}, these are $\ell^{\mrm{th}}$ powers in $Y_P$ and hence Property 4 holds. To check that Property 7 holds, we can assume $\eta$ is of Type 1b or 2. Proposition \ref{propositionE-1bRAM} also implies that each $E_{j,P}\otimes F_{P,\eta}$ is unramified. Thus the only case to check is when $E_{j,P} \simeq  F_P[t]/(t^{\ell}- \tilde{y_j}\delta_P)$. However, the same proposition gives that $\overline{u_{\eta}}$ is a norm from $\overline{E_{j,\eta}}$ where $\overline{u_P}=\overline{u_{\eta}}\in k_{P,\eta}$ is the residue of $D$ along the branch. Since we are in the case when $\overline{E_{j,\eta}\otimes k_{P,\eta}}$ is ramified, this implies that $\overline{u_P}\in k_P^{*\ell}$ and hence $u_P\in \widehat{A_P}^{*\ell}$. Therefore $D=0\in \Br(F_P)$ already.

\textit{$Y_{\eta}$ is SPLIT}: For $j=1,2$, write $a_{j,\eta} = \brac{\tilde{a}_{j,i,\eta}}_i \in \prod F_{P,\eta}$ where $\tilde{a}_{j,i,\eta} = x_{j,i,P}\pi_{P}^{m_{j,i}}$ where $x_{j,i,P}\in \widehat{A_{P,\eta}}^*$. Let $\overline{x_{j,i,P}} = x'_{j,i,P}\overline{\delta_P}^{s_{j,i}} \in k_{P,\eta}$ where $x'_{j,i,P} \in \mathcal{O}_{k_{P,\eta}}^*$. Let $\tilde{x'}_{j,i,P}\in \widehat{A_P}^*$ be a lift of $\overline{x'_{j,i,P}}$. 

Set $a_{1,P} = \brac{\tilde{a}_{1,i,P}}_i$ where $\tilde{a}_{1,i,P} = \tilde{x'}_{1,i,P}\pi_P^{m_{1,i}}\delta_P^{s_{1,i}}$ for $1\leq i \leq \ell-1$. Set  $\tilde{a}_{1,\ell,P} = \brac{\tilde{a}_{1,1, P}\ldots \tilde{a}_{1, \ell-1,P}}^{-1}$. And set $a_{2,P} = aa_{1,P}^{-1}$. Thus Properties 1, 5 \& 8 hold. We have already checked that Property 7 holds in this case ($Y_{\eta}$ being SPLIT).

Since $\tilde{a}_{j,i,\eta}$ is a norm from $E_{j,\eta}$ for each $i$, we have $\brac{\tilde{a}_{j,i,\eta}, e_{j,P}}= \brac{\tilde{a}_{j,i,P}, e_{j,P}} = 0\in \Br\brac{F_{P,\eta}}$.  By construction, $\brac{\tilde{a}_{j,i,P}, e_{j,P}}$ is ramified at most along $\pi_P$ and $\delta_P$ in $\Br\brac{F_P}$. Hence by (\cite{PPS}, Corollary 5.5), we have $\brac{\tilde{a}_{j,i,P}, e_{j,P}}=0\in \Br\brac{F_P}$ also for each $i$. Therefore Property 4 holds.

\textit{$Y_{\eta}$ is RES/NONRES}: Since $P$ is a curve point, $Y$ is arranged to be in good shape and $Y_{\eta}/F_{\eta}$ is unramified, we have $Y_P = F_P[t]/\brac{t^{\ell}-v_P}$ for some $v_P\in \widehat{A_P}^*$. Hence $Y_P$ is either split or $L_P$, the unique unramified extension of $F_P$ of degree $\ell$. Thus $Y\otimes F_{P,\eta}$ is unramified over $F_{P,\eta}$ as also $\overline{Y\otimes F_{P,\eta}}$ over $k_{P,\eta}$. Note that $a_{j,\eta}\in \mathcal{O}_{Y\otimes F_{P,\eta}}^*$ by construction and $E_{j,\eta}$ is unramified along $\eta$ (cf. proofs of Propositions \ref{propositionE-choosingliftofresidues}, \ref{propositionE-2RESY} and \ref{propositionE-0-NONRES}). 

Let $\overline{a_{j,\eta}} = x'_j \overline{\delta_P}^{s_j}\in {\overline{Y\otimes F_{P,\eta}}}$ where $x'_j\in \mathcal{O}_{\overline{Y\otimes F_{P,\eta}}}^*$. Set $a_{1,P} = \tilde{x'_1}\delta_P^{s_1}\in Y_P$ where $\tilde{x'_1}\in Y_P$ is a lift of $\overline{x'_1}$ and $a_2 = aa_{1,P}^{-1}$. Thus Properties 1, 5 \& 8 hold. Since ${a}_{j,\eta}$ is a norm from $E_{j,\eta}$, we have $\brac{{a}_{j,\eta}, e_{j,P}}= \brac{{a}_{j,P}, e_{j,P}} = 0\in \Br\brac{Y\otimes F_{P,\eta}}$.  Note that $Y_P$ is an unramified extension of $F_P$. By construction, $\brac{{a}_{j,P}, e_{j,P}}$ is ramified at most along $\pi_P$ and $\delta_P$ in $\Br\brac{Y_P}$. Hence by (\cite{PPS}, Corollary 5.5), we have $\brac{{a}_{j,P}, e_{j,P}}=0\in \Br\brac{Y_P}$ also. Therefore Property 4 holds.

To check Property 7, we can assume $\eta$ is Type 1b or 2. When $Y_{\eta}$ is of Type NONRES, $E_{j,\eta}$ is the lift of residues. Thus, $E_{j,P}=L_P$ or $\prod F_P$ where we have checked that Property 7 holds. When $Y_{\eta}$ is of Type RES, Proposition \ref{propositionE-1bRESY} guarantees that $\overline{u_{\eta}}$ is a norm from $\overline{E_{j,\eta}}$ where $\overline{u_P}=\overline{u_{\eta}}\in k_{P,\eta}$ is the residue of $D$ along the branch. Arguing as in the case when $Y_{\eta}$ is Type RAM, we are done. \end{proof}

\begin{proposition}
\label{propositionfinala} Let $j=1$ or $2$ and let $\eta\in N_0$. Let $\Gal\brac{Y/F} = \langle \psi \rangle$. For each $P\in \mathcal{P}'_{\eta}\cup \mathcal{R}'_{\eta}$, there exist elements $h_{j,P,\eta}\in Y\otimes F_{P,\eta}$ such that\[a_{j,U_{\eta}}{h_{j,P,\eta}^{-\ell}\psi\brac{h_{j,P,\eta}}^{\ell}} = a_{j, P} \in Y\otimes F_{P,\eta}.\] 
\end{proposition}
\begin{proof} Let $m\geq 1$ be\footnote{As before, such an $m$ exists because the residue field $k_P$ of its residue field $k_{P,\eta}$ is a finite field (of characteristic not $\ell$).} such that $F_{P,\eta}$ does not contain a primitive $\ell^{m}$-th root of unity. By Propositions \ref{propositionEaatUsinpatching} and \ref{propositionEaatnewpoints}, there exist norm one elements $v_{j,\eta}$ and $\mu_{j,P,\eta}$ in $Y\otimes F_{P,\eta}$ such that $a_{j, U_{\eta}}{v_{j,\eta}^{\ell}}{\mu_{j,P,\eta}} = a_{j,P}\in Y\otimes F_{P,\eta}$.

Proposition \ref{propositionEaatnewpoints} also gives us that $\mu_{j,P,\eta}\in \brac{Y\otimes F_{P,\eta}}^{*\ell^{2m}}$ if $Y\otimes F_{P,\eta}$ is a field extension and $\mu_{j,P,\eta}\in \prod F_{P,\eta}^{*\ell}$ if $Y\otimes F_{P,\eta}$ is split. Therefore by Lemma \ref{lemmasimultaneousnormoneandlthpower} and Hilbert 90, there exists $h_{j,P,\eta}\in Y\otimes F_{P,\eta}$ such that $a_{j,U_{\eta}}{h_{j,P,\eta}^{-\ell}\psi\brac{h_{j,P,\eta}}^{\ell}} = a_{j, P}.$ \end{proof}

\begin{remark}
\label{patchingsetupp} Note that $\{\mathcal{P}'_{\eta}\cup \mathcal{R}'_{\eta}, U_{\eta}\}_{\eta\in N_0}$ forms a patching set $\mathcal{P}$ as in defined in (\cite{HH}). \end{remark}

\begin{proposition}
\label{propositionEpatching}
Let $j=1$ or $2$. Then there exist $E_j/F$, degree $\ell$ extensions of $F$ which are subfields of $D/F$ and elements $a_j\in Y$ such that 
\begin{itemize}
\item
$a_1a_2=a$ and $\N_{Y/F}\brac{a_j}=1$.
\item
$E_{j}\otimes_F F_{U_{\eta}}\simeq E_{j,U_{\eta}}$ and $E_j\otimes_F F_{P}\simeq E_{j,P}$ for the patching set-up $\mathcal{P}$.
\item
$D\otimes E_j\otimes Y$ is split and $E_j\subseteq C_D\left(Y\right)$.
\item
There exist $\theta_j\in E_jY \subseteq D$ such that $\N_{E_jY/Y}\left(\theta_j\right) = a_j$.
\end{itemize}
\end{proposition}
\begin{proof} Let $j=1$ or $2$. In this proof, by $x\in \mathcal{P}$ we mean $x \in \{U_{\eta}, P'_{\eta}\cup \mathcal{R}'_{\eta}\}$ of the patching set up $\mathcal{P}$ defined in Remark \ref{patchingsetupp}.

From Propositions \ref{propositionEaatUsinpatching} and \ref{propositionEaatnewpoints}, we see that by (\cite{HH}, Theorem 7.1), there exists a degree $\ell$ etale algebra $\tilde{E_j}/F$ such that $\tilde{E_j}\otimes_F F_{U_{\eta}}\simeq E_{j, U_{\eta}}$ and $\tilde{E_j}\otimes_F F_{P}\simeq E_{j, P}$ for the patching set-up $\mathcal{P}$. Since at least one of the $E_{j, P}$ (or the $E_{j, U_{\eta}}$) is a nonsplit field extension, clearly $\tilde{E_j}/F$ is a field.

These propositions also guarantee that $\ind\left(D\otimes_ F E_{j,x}\right)\leq \ell$ for each $x\in \mathcal{P}$. Therefore by (\cite{HHK09}, Theorem 5.1), we have that $\ind\left(D\otimes_F \tilde{E_j}\right) \leq \ell$ and hence there exists a subfield of $D$ isomorphic to $\tilde{E_j}/F$ which we again call $\tilde{E_j}$.

We also have that $Y_x\otimes_F E_{j,x}$ splits $D$ for each $x\in \mathcal{P}$. Thus $D\otimes_F Y\otimes_F \tilde{E_j}$ is split. And therefore $C_D(Y)\otimes_Y (Y\otimes_F \tilde{E_j})$ is split. As $D$ is a divison algebra of degree ${\ell}^2$, we have that $C_D(Y)/Y$ is division of degree $\ell$, and hence $Y\otimes_F \tilde{E_j}$ splits $C_D(Y)$. Thus it is a degree $\ell$ field extension of $Y$ and therefore a degree ${\ell}^2$ field extension of $F$.

Since $Y\otimes_F \tilde{E_j}$ is a splitting field of $D$, which is a division algebra of degree ${\ell}^2$, there exists $L'_j$, a  maximal subfield of $D$ which is isomorphic to $Y\otimes_F \tilde{E_j}$. Let $E''_j$ denote the subfield of $L'_j$ which is isomorphic to $\{1\}\otimes_F \tilde{E_j}$ in $L'_j$ and $Y'_j$, the isomorphic copy of $Y\otimes_F \{1\}$. Thus $E''_j$ and $Y'_j$ are commuting degree $\ell$ subfields of $D$. 

By Skolem-Noether, $Y= b_jY'_jb_j^{-1}\subseteq D$ for some unit $b_j\in D^*$. Set $L_j= b_jL_j'b_j^{-1}\subseteq D$ and $E_j=b_jE_j''b_j^{-1}\subseteq D$. Thus $E_j$ and $Y$ commute in $D$ (they are subfields of the maximal subfield $L_j$).

We now construct $a_1\in Y$ using the norm one elements $a_{1,x}\in Y\otimes F_x$ for $x\in \mathcal{P}$. By Proposition \ref{propositionfinala}, for each branch in the patching set-up corresponding to a pair $\left(U_{\eta}, P\right)$, we have $a_{1,P} = a_{1, U_{\eta}}{h_{1,P,\eta}^{-\ell}\psi\brac{h_{1,P,\eta}}^{\ell}}$ for some $h_{1,P,\eta}\in \brac{Y\otimes F_{P,\eta}}^*$ where $\Gal\brac{Y/F} = \langle{\psi}\rangle$. By simultaneous factorization for curves for the rational group $\mathrm{R}_{Y/F}\brac{\mathbb{G}_m}$ (\cite{HHK09}, Theorem 3.6), we can find $h_{1,x}\in \brac{Y\otimes F_x}^*$ for each $x\in \mathcal{P}$ such that for every pair $\left(U_{\eta}, P\right)$, we have $h_{1,P,\eta} = h_{1,U_{\eta}}h_{1,P}^{-1}$. Thus for every branch defined by $\left(U_{\eta}, P\right)$, 
\begin{align*}
& a_{1,P} = a_{1, U_{\eta}}h_{1,P,\eta}^{-\ell}\psi\brac{h_{1,P,\eta}}^{\ell} \\
& \implies a_{1,P} = a_{1, U_{\eta}} h_{1,U_{\eta}}^{-\ell} h_{1,P}^{\ell}\psi\brac{h_{1,U_{\eta}}}^{\ell}\psi\brac{h_{1,P}}^{-\ell}\\
& \implies a_{1,P}h_{1,P}^{-\ell}\psi\brac{h_{1,P}}^{\ell} =  a_{1, U_{\eta}} h_{1,U_{\eta}}^{-\ell}\psi\brac{h_{1,U_{\eta}}}^{\ell}.
\end{align*}
Let $x\in \mathcal{P}$. Thus by (\cite{HH}, Proposition 6.3 \& Theorem 6.4), we have an element $a_1\in Y$ such that $a_1 = a_{1,x}h_{1,x}^{-\ell}\psi\brac{h_{1,x}}^{\ell} \in Y\otimes F_x$ and $\N\brac{a_1}=1$. Set $a_2 = aa_1^{-1}$. Note that $a_j \cong a_{j,x}$ up to $\ell^{\mrm{th}}$ powers in $Y\otimes F_x$.

Now we only have to verify that $a_j$ is a norm from $E_jY$. Wlog let $j=1$ (the same proof works for $j=2$). By Propositions \ref{propositionEaatUsinpatching} and \ref{propositionEaatnewpoints}, we see that $\left(a_{1,x}, E_{1,x}\right)_{Y_{x}}$ is split for each $x$. This implies that $a_{1,x}$ and hence $a_1$ is a norm from $E_1Y\otimes_Y Y\otimes_F F_x$ over $Y\otimes F_x$ as $a_1$ differs from each $a_{1,x}$ by an $\ell^{\mrm{th}}$ power.

There exists a field extension $N/Y$ of degree coprime to $\ell$ such that $E_1Y\otimes_{Y}N$ is a cyclic field extension of degree $\ell$ (\cite{Albert}, Chapter IV, Theorem 31). Let $\mathcal{Y}$  (resp. $\mathcal{Z}$) denote the normal closure of $\mathcal{X}$ in $Y$ (resp. $N$) with special fiber $Y_0$ (resp. $Z_0$). Let $\gamma: Z_0\to Y_0$ and $\phi: Y_0\to X_0$ be the induced morphisms. Then, as in the proof of (\cite{PPS}, Proposition 7.5), we have induced patching systems $\mathcal{Y}'$ of $Y_{0}$ (resp. $\mathcal{Z}'$ of $Z_{0}$) consisting of open sets $U_{y}$ (resp. $U_z$) and closed points $P_{y}$ (resp. $P_z$) such that  $F_U\subset Y_{U_y} \subset N_{U_z}$, $F_P\subset Y_{P_y}\subset N_{P_z}$ for $U,P\in \mathcal{P}$ with $\gamma(U_z) \subset U_y$, $\phi(U_y)\subset U$, $\gamma(P_z)= P_y$ and $\phi(P_y)=P$. 

\newpage
Then for $x=U$ or $P$, we have the following commutative diagram induced by norm maps
\[\begin{tikzcd}
   E_1Y\otimes_Y  Y \otimes_F F_x \arrow{r}  \arrow{d} &  E_1Y\otimes_Y Y_{x_y}\arrow{r} \arrow{d} & E_1Y\otimes_Y N_{x_z} = \brac{(E_1Y\otimes_Y N)\otimes_N N_{x_z}}  \arrow{d} \\
  Y \otimes_F F_x \arrow{r}  & Y_{x_y}\arrow{r}  & N_{x_z} \\
    \end{tikzcd}
    \] This shows that $(a_1, E_1Y\otimes_Y N)$ is trivial over each $N_{x_z}$ for each $x_z\in \mathcal{Z}'$ and hence trivial over $N$ (\cite{HHK09}, Theorem 5.1). Thus $a_1$ is a norm of the extension $E_1Y\otimes_Y N/N$ and hence $a_1^{[N:Y]}$ is a norm of $E_1Y/Y$. Since $[N:Y]$ is coprime to $\ell$, this implies that there exists $\theta_1\in E_1Y$ such that $\N_{E_1Y/Y}\left(\theta_1\right)=a_1$. \end{proof}

\section{Solving the problem over $E_j$}
\label{section-f}
Recall that we started with $z\in \SL1(D)$ living in a maximal subfield $M$ of $D$ which contains a cyclic degree $\ell$ subfield $F\subseteq Y \subseteq M$ with $\N_{M/Y}(z):=a$. Let $a_j, E_j$  and $\theta_j$ be as in Proposition \ref{propositionEpatching} for $j=1, 2$. Note that $\N_{E_jY/F}(\theta_j)=1$ and hence $\theta_j \in \SL1(D)$. If we can prove that $\theta_j \in\left[D^*, D^*\right]$, then by Proposition \ref{propositionEconsequence}, $z\in \left[D^*, D^*\right]$.

 Let $c_j:= \N_{E_jY/E_j}(\theta_j)$. Since the proofs for the cases $j=1$ and $j=2$ are similar, without loss of generality, assume $j=1$. We also drop the suffixes in the remainder of this paper, i.e. we set $E:=E_1$, $\theta:=\theta_1$, $c:=c_1$ etc. 

\subsection{Strategy \`a la Platonov}
To show $\theta \in\left[D^*, D^*\right]$, we adapt the basic strategy underlying the proof of the triviality of $\SK1(D)$ over global fields (\cite{Plat76}, Theorem 5.4) as follows:

There exists a suitable\footnote{We can and do choose the coprime extension $N/F$ carefully as follows: Let $N'$ be the Galois closure of $E/F$ and take $N$ to be the fixed field of an $\ell$-Sylow group of $\Gal(N/F)$. Thus $[N:F]$ is coprime to $\ell$ and $E\otimes_F N = N'$ which is indeed cyclic over $N$ of degree $\ell$ (cf. \cite{Albert}, Chapter IV, Theorem 31)} field extension $N/F$ such that $[N:F]$ is coprime to $\ell$ with $E_{N}:= E\otimes_F N$, a cyclic subfield of $D_{N}:= D\otimes N$. By (\cite{Plat76}, Lemma 2.2, Section 2.4), it suffices to show that $\theta \in\left[D_{N}^*, D_{N}^*\right]$. Let $Y_{N}:= Y\otimes_F N$ and $\Gal(E_{N}/N) = \langle \sigma \rangle$. Note that $\theta\in E_{N}Y_{N}\subseteq C_{D_{N}}(E_{N})$ and $\N_{E_{N}Y_{N}/E_{N}}(\theta)=c$. Therefore the further norm, $\N_{E_{N}/N}(c)=1$. Now, because $E_{N}/N$ is a cyclic extension with Galois group $\langle \sigma\rangle$, by Hilbert 90, there exists a $b\in E_{N}$ such that $c = b^{-1}\sigma(b) \in E_{N}$. Note that $c=  b^{-1}\sigma(b)$ is a reduced norm in $E_N$ from $C_{D_{N}}(E_{N})$.

\begin{proposition}
\label{propositionf} For $N, \sigma, b, c$ as above, if there exists $f\in N$ such that $bf$ is a reduced norm in $E_N$ from $\left(C_{D_{N}}(E_{N})\right)$, then $\theta\in \left[D_{N}^*, D_{N}^*\right] $. \end{proposition}
\begin{proof} Set $b'=bf$. Note that $c = b^{-1}\sigma(b) = (bf)^{-1}\sigma(bf) = b'^{-1}\sigma(b')$. By Skolem Noether, extend $\sigma : E_N\to E_N\subseteq D_N$ to an automorphism of $D_N$ given by $\tilde{\sigma} = \Int(v) : D_N\to D_N, \ d\leadsto vdv^{-1}$. Note that $\tilde{\sigma}$ restricts to an $N$-automorphism of $C_{D_N}(E_N)$ since $\tilde{\sigma}|_{E_N} = \sigma$. Set $D_1=C_{D_N}(E_N)$.

By hypothesis, there exists $g\in D_1$ such that $\Nrd_{D_1/E_N}(g) = b'$.  Thus, 
\begin{align*}
\Nrd_{D_1/E_N}\left(g^{-1}vgv^{-1}\right) &= \Nrd_{D_1/E_N}\left(g^{-1}\right) \Nrd_{D_1/E_N}\left(vgv^{-1}\right)  \\
& = b'^{-1} \Nrd_{D_1/E_N}\left(\tilde{\sigma}(g)\right) \\
& =  {b'^{-1} \tilde{\sigma}\left(\Nrd_{D_1/E_N}(g)\right)}\\
& = b'^{-1} \sigma(b') \\
& = c
\end{align*}

Since $\Nrd_{D_1/E_N}\brac{\theta}=c$, we have $\Nrd_{D_1/E_N}\left(\theta vg^{-1}v^{-1}g \right) = 1$. Since $D_1$ is a central simple algebra of square-free index $\ell$, $\SL1\left(D_1\right) = \left[D_1^*, D_1^*\right]$ (\cite{W}). Hence we have that $\theta vg^{-1}v^{-1}g \subseteq  \left[D_1^*, D_1^*\right]\subseteq \left[D_N^*, D_N^* \right]$. \end{proof}

We will find $f\in N$ satisfying the hypothesis of Proposition \ref{propositionf} by patching suitable elements $f_x\in \brac{N\otimes_F F_x}^*$ for $x$ in a refinement of the patching system $\mathcal{P}$ used to construct $E$ (Remark \ref{patchingsetupp}). 

\subsubsection{The shapes of $E_N$ and $b$}
We investigate the shape of $E$ after the coprime base change $N$. Let $x\in \mathcal{P}$. Since $E_x:= E\otimes_F F_x$ is a cyclic extension by construction and $E_N = N'$, the Galois closure of $E/F$, we see that $E_N\otimes_F F_x\simeq \prod_{[N:F]} E_x$. Let $N\otimes_F F_x = \prod_{i=1}^{r_x} N_{i,x}$. Since $[E_x:F_x]=\ell$, this forces each $N_{i,x}$ to be isomorphic to $F_x$ or $E_x$. Hence $E_N\otimes_F F_x$ as an $N\otimes_F F_x$ algebra is the product of an appropriate number of copies of the cyclic extensions $E_x/F_x$ and the split extensions $\prod_{\ell} E_x/E_x$.

Let $b\otimes 1\in E_N\otimes_F F_x$ correspond to the entry $\prod_q b_q \times \prod_i (b_{i,1,P}, b_{i,2,P}, \ldots, b_{i,\ell,P})$ in $\prod_q E_x/F_x \times \prod_i \brac{\prod_{\ell} E_x/E_x}$.  The $\sigma$ action is componentwise and further in $\prod_{\ell} E_x/E_x$, it permutes the entries of each tuple $\left(b_{i,j,P}\right)_{j\leq \ell}$ amongst themselves, i.e. $\sigma\brac{\prod_i (b_{i,j,P})_{j\leq \ell}} = \prod_i \brac{ b_{i,\sigma(j),P} }_{j\leq \ell}$. The $\sigma$ action on the $E_x/F_x$ components can be similarly described if $E_{x}\simeq \prod_{\ell} F_{x}$ is itself split.

For $\eta\in N_0$ and closed point $P\in \overline{\eta}$, let $x = \eta$  or $(P,\eta)$. Then we denote the integral closure of $\widehat{A_x}$ in $E\otimes F_{x}$ by $\widehat{B_x}$. Let its residue field be denoted $k'_{x}$. Similarly, let $\widehat{C_{x}}$ denote the integral closure of $\widehat{A_x}$ in $N\otimes F_{x}$ with residue field $k''_{x}$. Thus $\widehat{C_{x}} \simeq \prod \widehat{A_{x}} \times \prod \widehat{B_{x}}$.

We begin with the following broad modification of $b$: Let $\eta\in N_0$ be such that $E_{\eta}/F_{\eta}$ is an unramified field extension. Thus $E_N \otimes_F F_{\eta}/N\otimes_F F_{\eta}$ is the unramified (possibly split or partially split) extension $\prod E_{\eta}/F_{\eta} \times \prod \brac{\prod_{\ell} E_{\eta}/E_{\eta}}$. By weak approximation, modify $b$ by a suitable element of $N$ so that if $b = \prod_q b_q \times \prod_i (b_{i,1,\eta}, b_{i,2,\eta}, \ldots, b_{i,\ell,\eta})$ in $\prod_q E_{\eta}/F_{\eta} \times \prod_i \brac{\prod_{\ell} E_{\eta}/E_{\eta}}$, then

\begin{enumerate}
\item
Each $b_q$ living in any component of shape $E_{\eta}/F_{\eta}$ is a unit in $\widehat{B_{\eta}}$. This can be done by knocking off an appropriate power of $\pi_{\eta}$ from $F_{\eta}$.
\item
If $D\otimes E_{\eta}$ is an unramified algebra of index $\ell$, then each entry $b_{i,j,\eta}$ in the tuple $(b_{i,j,\eta})_{j\leq \ell}$ living in any component of shape $\prod_{\ell}E_{\eta}/E_{\eta}$ is a unit in $\widehat{B_{\eta}}$. This can be done as follows:

Let $m_j$ denote the valuation of $b_{i,j,\eta}$ in $E_{\eta}$. It suffices to check that all $m_j$s are equal because then we can again make $(b_{i,j,\eta})_{j\leq \ell}\in \prod \widehat{B_{\eta}}^*$ by knocking off $\pi_{\eta}^{m_j}$ from $E_{\eta}$. Since $c=b^{-1}\sigma(b)$ is a reduced norm from $D\otimes E$, this implies $b_{i,1,\eta}^{-1}b_{i,j,\eta}$ is a reduced norm from $D\otimes E_{\eta}$ for each $j\leq \ell$. Since every unit in $\widehat{B_{\eta}}$ is a reduced norm from $D\otimes E_{\eta}$ (Proposition \ref{lemmareducednormsunramified}), this forces all valuations $m_j$ to equal each other. 
\item
If $D\otimes F_{\eta}$ is an unramified algebra of index $\ell$ and if $E_{\eta}\simeq \prod F_{\eta}$ is split, then each $b_q = (b_{q, j, \eta})_{j\leq \ell}$ living in any component of shape $E_{\eta}/F_{\eta}$ is a unit in $\widehat{B_{\eta}}$, i.e. each $b_{q,j,\eta}\in \widehat{A_{\eta}}^*$. This can be achieved by a similar argument as in 2).
\end{enumerate}

 \subsection{Preliminary patching data of $f$}
 Recall that for each $\eta \in N_0$, $\mathcal{P}'_{\eta} := \overline{\eta}\cap S_0$ and $\mathcal{R}'_{\eta} := \brac{\overline{\eta}\setminus U_{\eta}}\setminus S_0$. Thus $S_0 = \cup_{\eta\in N_0}\mathcal{P}'_{\eta}$.

\begin{proposition}[$f$ at closed points]
\label{prop-f-atclosedpoints}
Let $\eta\in N_0$ and $P\in \mathcal{P}'_{\eta} \cup \mathcal{R}'_{\eta}$. Then there exists $f_P\in \brac{N\otimes_F F_P}^*$ such that $bf_P$ is a reduced norm from $D_{N}\otimes E_P$. Further $f_P\in \widehat{C_{P,\eta}}^*$.
\end{proposition}
\begin{proof}
If $D\otimes E_P$ is split, set $f_P = 1$. Note that $bf_P$ (and indeed any other element in $N\otimes E_P$) is a reduced norm from $D_N\otimes E_P$. Clearly $1\in \widehat{C_{P,\eta}}^*$.

Therefore assume $D\otimes E_P$ is not split and let $P\in \overline{\eta}\cap \overline{\eta'}$. We can in fact pinpoint precisely when this happens by a closer inspection of the proofs of Propositions \ref{propositionEatpoints} and \ref{propositionEaatnewpoints} - at points in Rows 2.2* of Table \ref{TableEatB11nspoint}, 4.1* of Table \ref{TableEatB21nspoint}, 8.5-8.6 of Table \ref{TableEatC21hotpoint} and at some innocuous curve points in $\mathcal{R}'_{\eta}$ where both $Y_{\eta}$ and $Y_{\eta'}$ are not SPLIT. Note that in all these cases, $E_{\eta}/F_{\eta}$ and $E_{\eta'}/F_{\eta'}$ are unramified field extensions by construction, $\{\eta, \eta'\} = \{\mathrm{Type\ 1b}, \mathrm{Type\ 1a}\}$ or $\{\mathrm{Type\ 1b}, \mathrm{Type\ 2}\}$ and $D_P\simeq (u_P, \pi_P)$ for some unit $u_P\in \widehat{A_P}^*$ and $\pi_P$ defines one of $\eta$ or $\eta'$ at $P$.

By Proposition \ref{propositionEaatnewpoints}, $E_P \simeq \prod F_P$ and therefore $E_N \otimes_F F_P/N\otimes_F F_P\simeq \prod \brac{\prod_{\ell} F_P/F_P}$. Let $b\otimes 1$ correspond to the entry $\prod_i (b_{i,1,P}, b_{i,2,P}, \ldots, b_{i,\ell,P})$. As discussed before, $\sigma$ permutes the entries of each tuple $\left(b_{i,j,P}\right)_{j\leq \ell}$ amongst themselves. Since $c=  b^{-1}\sigma(b) \in \Nrd_{E_N}\left(C_{D_N}(E_N)\right)$, we have that $\left(b_{i,1,P}\right)[D_P] = \left(b_{i,2,P}\right)[D_P] = \ldots = \left(b_{i,\ell,P}\right)[D_P] \in \HH^3\left(F_P, \mu_{\ell}\right)$. Let $b_{i,j,P}$ have valuation $m_j$ in $F_{P,\eta}$.

We first look at the case when $\pi_P$ defines $\eta$.  Set $f_{i,P}:= b_{i,1,P}^{-1}\pi_P^{m_1}$. Since $\pi_P$ is a parameter of $F_{P,\eta}$ also, $f_{i,P}$ is a unit along $\eta$. Define $f_P = \prod_i (f_{i,P})\in N\otimes F_P$. Thus $f_P\in \widehat{C_{P,\eta}}^*$. It now suffices to see that each $b_{i,j,P}f_{i,P}$ is a reduced norm from $D_P$.  For $j=1$, we have $\left(b_{i,1,P}f_{i,P}\right)[D_P] = (\pi_P^{m_1})(u_P, \pi_P)=0$. Since the cup-products $(b_{i,j,P})[D_P]$ equal each other for $j\leq \ell$, we have $\left(b_{i,j,P}f_{i,P}\right)[D_P]=  0$ for each $j$.

Identifying $\HH^1\brac{F_P, \SL1\brac{D_P}}$ with $F_P^*/\Nrd\brac{D_P}^*$, recall Suslin's invariant \[R  : \HH^1\brac{F_P, \SL1\brac{D_P}} \to \HH^3\brac{F_P, \mu_{\ell}^{\otimes 2}}, \  \lambda \leadsto \brac{\lambda}\cup [D_P].\]

Since index of $D_P$ is $\ell$ and in particular square-free, $R$ is injective (\cite{Suslin}, Theorem 12.2). Hence $b_{i,j,P}f_{i,P}$ is a reduced norm from $D_P$ and hence we are done in this case.

Now let's look at the case when $\pi_P$ defines $\eta'$. Thus $\eta$ is either of Type 1a or $P$ is a hot point and $\eta$ is of Type $2$ with $Y_{\eta}$ of Type NONRES (Rows 8.5-8.6 of Table \ref{TableEatC21hotpoint}). In either case $D\otimes E_{\eta}$ is an unramified index $\ell$ algebra\footnote{It is unramified if $\eta$ is Type 1a and by Proposition \ref{propositionE-choosingliftofresidues} otherwise. It is non-split since $D\otimes E_{P}$ and hence $D\otimes E_{P,\eta}$ is non-split.}. Thus by our initial modification, $b_{i,\eta}\in \widehat{B_{\eta}}^*$ already which shows that $b_{i,j,P}\in \widehat{A_{P,\eta}}^*$ for each $j\leq \ell$. Since $\pi_P$ is a unit along $\eta$ now, so is $f_{i,P}$.  \end{proof}

 \begin{proposition}[$f$ at codimension one points]
\label{prop-f-atcodimensiononepoints}
Let $\eta\in N_0$. Then there exists $f_{\eta}\in {\widehat{C_{\eta}}^*}\subset N\otimes F_{\eta}$ such that 
\begin{itemize}
\item
$f_{\eta} = f_{P}\phi_{P,\eta}^{\ell}\in N\otimes F_{P,\eta}$ for some $\phi_{P,\eta}\in \brac{N\otimes F_{P,\eta}}^*$ for each $P\in \mathcal{P}'_{\eta}\cup \mathcal{R}'_{\eta}$.
\item
$bf_{\eta}$ is a reduced norm from $D_N\otimes E_{\eta}$.
\end{itemize}
\end{proposition}

\begin{proof}
Note that by Proposition \ref{prop-f-atclosedpoints}, we see that $f_P \in \widehat{C_{P,\eta}}^*$ for each $P\in \mathcal{P}'_{\eta}\cup \mathcal{R}'_{\eta}$.

\textbf{$\eta$ is of Type 0}: $D_{\eta}$ is split and so is $D\otimes F_P$ for every $P\in \overline{\eta}$. Thus each $f_P=1$ by choice for every marked point $P$ on $\overline{\eta}$ (Proposition \ref{prop-f-atclosedpoints}). Choose $f_{\eta}=1$. Clearly  $bf_{\eta}$ (and indeed any other element in $N\otimes E_{\eta}$) is a reduced norm from $D_N \otimes E_{\eta}$.

\textbf{$\eta$ is of Type 1a}: By construction, $E_{\eta}$ is unramified (Propositions \ref{propositionE-0-RAM}, \ref{propositionE-0-SPLIT} and \ref{propositionE-0-NONRES}). By weak approximation, find $\overline{f_{\eta}}\in k''_{\eta}$ which is close to $\overline{f_P}\in k''_{P,\eta}$ for each marked $P$ in $\overline{\eta}$. Let $f_{\eta}$ be a lift of $\overline{f_{\eta}}$ in $\widehat{C_{\eta}}^*$.

If $D\otimes E_{\eta}$ is split, then clearly $bf_{\eta}$ (and indeed any other element in $N\otimes E_{\eta}$) is a reduced norm from $D_N \otimes E_{\eta}$. So assume $D\otimes E_{\eta}$ is not split. Since $\eta$ is Type 1a, $D_{\eta}$ is an unramified index $\ell$ algebra and hence so is $D\otimes E_{\eta}$. By our initial modification of $b$, this implies all components of $b$ are units along $\eta$. Thus by Lemma \ref{lemmareducednormsunramified}, $bf_{\eta}$ is a reduced norm from $D_N\otimes E_{\eta}$.

\textbf{$\eta$ is of Type 1b/2}: Let $u_{\eta}\in F_{\eta}$ such that $u'=\overline{u_{\eta}}\in k_{\eta}/k_{\eta}^{*\ell}$ is the residue of $D_{\eta}$. There are three possible shapes of $E_{\eta}$.

\textit{Shape A}:  $E_{\eta}$ is a ramified/unramified field extension which splits $D_{\eta}$ (Propositions \ref{propositionE-choosingliftofresidues-violetandindigo}, \ref{propositionE-1bSPLIT-Blue} and \ref{propositionE-1bSPLIT-YellowOrange}).

\textit{Shape B}: $E_{\eta}$ is the lift of residues which might or might not split $D_{\eta}$ (Proposition \ref{propositionE-choosingliftofresidues}). Though in particular, it is an unramified field extension of $F_{\eta}$.

\textit{Shape C}: $E_{\eta}/F_{\eta}$ is an unramified field extension which is not the lift of residues of $F_{\eta}$. Then ${u'}$ is a norm from $\overline{E_{\eta}}$ and $E_{\eta}\otimes \beta_{\rbc,\eta}$ is split. (Propositions \ref{propositionE-1bRAM} and \ref{propositionE-1bRESY}).

For each shape, we prescribe $f_{\eta}\in N\otimes F_{\eta}$ as follows:

\textit{$E_{\eta}$ of Shape A/B}: As before find $\overline{f_{\eta}}\in k''_{\eta}$ which is close to $\overline{f_P}\in k''_{P,\eta}$ for each $P\in \mathcal{P}'_{\eta}\cup \mathcal{R}'_{\eta}$. Let $f_{\eta}$ be a lift of $\overline{f_{\eta}}$ in $\widehat{C_{\eta}}^*$. If $E_{\eta}$ is of Shape A , since $D\otimes E_{\eta}$ is split, every element in $N\otimes E_{\eta}$ is a reduced norm from $D_N\otimes E_{\eta}$.

Let $E_{\eta}$ be of Shape B. Note that $D = \beta_{rbc,\eta} + \brac{u_{\eta}, \pi_{\eta}} \in \Br\brac{F_{\eta}}$ where $\pi_{\eta}$ is a parameter of $F_{\eta}$. Thus $D\otimes E_{\eta}$ is an unramified algebra. If it is split, our choice of $f_{\eta}$ clearly works. So assume $D\otimes E_{\eta}$ has index $\ell$. Then by our initial modification of $b$, each component of $b$ is a unit along $\eta$. Thus by Lemma \ref{lemmareducednormsunramified}, $bf_{\eta}$ is a reduced norm from $D_N\otimes E_{\eta}$.

\textit{$E_{\eta}$ of Shape C} : Note that $D = \beta_{rbc,\eta} + \brac{u_{\eta}, \pi_{\eta}} \in \Br\brac{F_{\eta}}$ where $\pi_{\eta}$ is a parameter of $F_{\eta}$. Since $E_\eta$ splits $\beta_{rbc,\eta}$, we have $D\otimes E_{\eta} = \left( u_{\eta}, \pi_{\eta}\right)$. Let $E_N\otimes F_{\eta}/N\otimes F_{\eta}\simeq \prod_{q} E_{\eta}/F_{\eta} \times \prod_{i} \brac{\prod_{\ell} E_{\eta}/E_{\eta}}$ and let $b = \prod b_q \times \prod_i \brac{b_{i,j,\eta}}_{j\leq \ell}$. We will prescribe $f_{\eta}  = \prod f_q \times \prod_i e_i \in N\otimes F_{\eta}\simeq \prod_q F_{\eta} \times \prod_i E_{\eta}$ by prescribing each of its components $f_q\in \widehat{A_{\eta}}^*$ and $e_i \in \widehat{B_{\eta}}^*$ individually.

Let us look at the case of $b_q\in E_{\eta}/F_{\eta}$. By our initial modification of $b$, we have $b_q\in \widehat{B_{\eta}}^*$ for each $q$ and hence $\sigma(b_q)\in \widehat{B_{\eta}}^*$ also. Set $E' = \overline{E_{\eta}}$ and $b' = \overline{b_q}\in E'$. By abuse of notation, let $\Gal\brac{E'/k_{\eta}} = \langle \sigma \rangle$ also. Since $c$ is a reduced norm from $D_N\otimes E$,  $\left(b_q^{-1}\sigma(b_q) \right)\left(u_{\eta}, \pi_{\eta}\right)  = 0 \in \HH^3\left( E_{\eta}, \mu_{\ell}\right)$. This gives  $\left( b'^{-1}\sigma(b'), u' \right)  = 0 \in \HH^2\left(E', \mu_{\ell} \right)$. Thus $\brac{b', u'} = \brac{\sigma\brac{b'}, u'} \in \HH^2\left(E', \mu_{\ell} \right)$.

We would like to apply Lemma \ref{lemmainvariantalgebras} to find an $\overline{f_q}\in k_{\eta}$ and hence an $f_{q}\in F_{\eta}$ with the required properties. To do so, we proceed to verify that the rest of the hypotheses of the lemma are indeed satisfied by  $u'$, $E'/k_{\eta}$ and $b'$. 

By (\cite{S97}, \cite{S98}, Proposition 1.2), we see that the residue $u'$ is up to $\ell^{\mrm{th}}$ powers, a unit at almost all places $v$ of $k_{\eta}$ except at those given by cold points (Type C-Cold) $P$ on $\overline{\eta}$. Recall that by the choice of $E_P$ at cold points (cf Tables \ref{TableEatC11coldpoint} and \ref{TableEatC21coldpoint}), at such places $E'/k_{P,\eta}$ is given by adjoining the $\ell^{\mrm{th}}$ root of the residue $u'$ and hence $u'\in {E'_{P,\eta}}^{\ell}$. In particular, this discussion shows that at every place $v$ where $E'$ is unramified and inert, $u'\in \mathcal{O}_{E'_v}^*$ up to $\ell^{\mrm{th}}$ powers in ${E'_v}^*$.

Let $w$ be a place where $E'/k_{\eta}$ is ramified. We have already seen that if $w$ corresponds to a cold point $P$, then $u'\in {E'_w}^{*\ell}$. Therefore assume $w$ corresponds to a non-cold point $P$. Hence $u'\in \mathcal{O}_{k_{P,\eta}}^*$.  Since we know $u'$ is a norm from $E'$ and hence from $E'_w$,  Lemma \ref{lemmanormfromramified-dim1} implies that $u'\in {E'_w}^{*\ell}$.

Finally for $P\in \mathcal{P}'_{\eta}\cup \mathcal{R}'_{\eta}$, we have $\brac{b_qf_{q,P}}$ is a reduced norm from $D_N\otimes E_P$. This implies $\brac{b_qf_{q,P}}\brac{u_{\eta}, \pi_{\eta}}=0\in \HH^3\brac{E_{P,\eta}, \mu_{\ell}}$. Taking residues, this implies $\brac{u', b'} = \brac{u', \overline{f_{q,P}}^{-1}}\in E'\otimes k_{P,\eta}$. 

Thus we can apply Lemma \ref{lemmainvariantalgebras} to find $f_{q}\in F_{\eta}$ such that  $f_{q}\equiv f_{q,P} \in F_{P,\eta}$ up to $\ell^{\mathrm{th}}$ powers for marked points $P\in \mathcal{P}'_{\eta}\cup \mathcal{R}'_{\eta}$. Further, $\brac{u_{\eta}, b_qf_{q}}=0\in \Br\brac{E_{\eta}}$. This implies $\brac{u_{\eta}, \pi_{\eta}}\brac{b_qf_{q}}=0$. Since $D\otimes E_{\eta} = \brac{u_{\eta}, \pi_{\eta}}$, we have $b_qf_{q}\in \Nrd_{E_{\eta}}\left(D\otimes E_{\eta} \right)$ using injectivity of Suslin's invariant for index $\ell$ algebras again (\cite{Suslin}, Theorem 12.2).

Now let us look at the case of $\brac{b_{i,j,\eta}}_{j\leq \ell}\in \brac{\prod_{\ell} E_{\eta}}/E_{\eta}$. Since $c$ is a reduced norm from $D_N\otimes E$, we have that $\left(b_{i,1,\eta}\right)[D\otimes E_{\eta}] = \left(b_{i,2,\eta}\right)[D\otimes E_{\eta}] = \ldots = \left(b_{i,\ell,\eta}\right)[D\otimes E_{\eta}] \in \HH^3\left(E_{\eta}, \mu_{\ell}\right)$. Let $b_{i,j,\eta}$ have valuation $m_j$ in $E_{\eta}$.
 
Set $e'_{i}:= b_{i,1,\eta}^{-1}\pi_{\eta}^{m_1}$. Since $\pi_{\eta}$ is a parameter of $F_{\eta}$ and hence of $E_{\eta}$ also, $e'_{i}\in \widehat{B_{\eta}}^*$. Since $\left(b_{i,1,\eta}e'_i\right)[D\otimes E_{\eta}] = (\pi_{\eta}^{m_1})(u_{\eta}, \pi_{\eta})=0$. Since the cup-products $(b_{i,j,\eta})[D\otimes E_{\eta}]$ equal each other for $j\leq \ell$, we have $0=\left(b_{i,j,\eta}e'_i\right)[D\otimes E_{\eta}]$ for each $j$. Thus by injectivity of Suslin's invariant for index $\ell$ algebras, each $b_{i,j,\eta}e'_i$ is a reduced norm from $D\otimes E_{\eta}$.

However, $e'_i$ might not approximate the choice at marked points on $\overline{\eta}$. So we find a suitable correcting factor $\theta\in \widehat{B_{\eta}}^*$ such that $\theta\in \Nrd(D\otimes E_{\eta})$ and $e'_i\theta$ is close to the choice along marked points. Then $e_i = e'_i\theta$ is still in $\widehat{B_{\eta}}^*$ and each $b_{i,j,\eta}e_i$ is a reduced norm from $D\otimes E_{\eta}$.

Note that since $D\otimes E_{P,\eta}$ is still ramified, if $P\in \mathcal{P}'_{\eta}\cup \mathcal{R}'_{\eta}$, then $D\otimes E_P\simeq (u_P,\pi_P)$ where $\pi_P$ defines $\eta$ at $P$. Let $\pi_P = \pi_{\eta}\theta'$ for some $\theta'\in \widehat{A_{\eta}^*}$. Following the proof of Proposition \ref{prop-f-atclosedpoints}, we see that we are in the case when $\eta$ is Type 1b/2 and $E_P\simeq \prod F_P$. Thus $E_{P,\eta}\simeq \prod F_{P,\eta}$ and under this identification, $(b_{i,j,\eta})_{j\leq \ell} \in \brac{\prod_{\ell} E_{\eta}/E_{\eta}}$ goes to $(\sigma^{j-1}\brac{b_{i,1,\eta}}, \sigma^{j-1}\brac{b_{i,2,\eta}}, \ldots, \sigma^{j-1}\brac{b_{i,\ell,\eta}})_{j\leq\ell}$ in $\prod_{\ell} \brac{\prod_{\ell} F_{P,\eta}/F_{P,\eta}}$ over the branch. Our choice of $e_{i,P}$ along the branch corresponds to \[\brac{b_{i,1,\eta}^{-1}\pi_P^{m_1}, \sigma(b_{i,1,\eta})^{-1}\pi_P^{m_1}, \ldots, \sigma^{\ell-1}\brac{b_{i,1,\eta}}^{-1}\pi_P^{m_1} } \in \prod \widehat{A_{P,\eta}}^* \simeq \widehat{B_{P,\eta}}^*,\] i.e. $e_{i,P} = b_{i,1,\eta}^{-1}\pi_P^{m_1}\in \widehat{B_{P,\eta}}^*$ and $e_{i,P}{e'_i}^{-1} = {\theta'}^{m_1}\in \widehat{B_{P,\eta}}^*$.

Since both $\pi_{\eta}$ and $\pi_P$ are reduced norms from $D\otimes E_{P,\eta}$, so is $\theta'$ and hence ${\theta'}^{m_1}$. Therefore $(\theta'^{m_1}, u_{\eta})=0\in \HH^2(F_{P,\eta}, \mu_{\ell})$ and $(\overline{\theta'}^{m_1}, \overline{u_{\eta}}) = 0\in \HH^2(k_{P,\eta}, \mu_{\ell})$. Find $\theta_1\in F_{P,\eta}\left(\sqrt[\ell]{u_{\eta}}\right)$ such that $\N\left(\theta_1\right) = \theta'^{m_1}$. Note that since $D\otimes E_P = (u_P, \pi_P)$, $F_{P,\eta}\brac{\sqrt[\ell]{u_{\eta}}}$ is an unramified field extension of $F_{P,\eta}$. Choose $\tilde{\theta_1}\in \mathcal{O}_{F_{\eta}\left(\sqrt[\ell]{u_{\eta}}\right)}^*$ such that its image is close to $\theta_1$ and set $\theta = \N\left(\tilde{\theta_1}\right)\in F_{\eta}$. \end{proof}

\subsection{Spreading and patching of $f$}
\begin{proposition}\label{propositionfatUs}
For each $\eta$ in $N_0$, there exist a neighbourhood $U'_{\eta}$ of $\eta$ in $X_0$ with $U'_{\eta}\subseteq \overline{\eta}\setminus \brac{\mathcal{P}'_{\eta}\cup \mathcal{R}'_{\eta}}$ and an $f_{U'_{\eta}}\in N\otimes F_{U'_{\eta}}$ such that 
\begin{enumerate}
\item
$U'_{\eta}\subseteq U_{\eta}$ where $U_{\eta}$ are the neighbourhoods in the patching set up $\mathcal{P}$ 
\item
$bf_{U'_{\eta}}$ is a reduced norm from $D_N\otimes E\otimes F_{U'_{\eta}}$.
\item
$f_{U'_{\eta}} \cong f_{\eta}$  upto ${\ell}^{\mrm{th}}$ powers in $N\otimes F_{\eta}$.
\end{enumerate}
\end{proposition}

\begin{proof} By Proposition \ref{prop-f-atcodimensiononepoints}, we see that $f_{\eta}\in \widehat{C_{\eta}}^*$ and that $bf_{\eta}\in \Nrd(D_N\otimes E_{\eta})$. Thus $\brac{bf_{\eta}}([D_N\otimes E_{\eta}])=0\in \HH^3\brac{N\otimes E_{\eta}, \mu_{\ell}}$. Let $f' \in N^*$ such that $f_{\eta}^{-1}f'$ is $1$ mod the maximal ideal of $\widehat{C_{\eta}}$. Note that $f_{\eta} = f' x^{\ell}\in \brac{N\otimes F_{\eta}}^{*}$ for some $x\in \widehat{C_{\eta}}^*$ and hence $\brac{bf'}([D_N\otimes E_{\eta}])=0\in \HH^3\brac{N\otimes E_{\eta}, \mu_{\ell}}$. By (\cite{PPS}, proof of Lemma 7.2 \& \cite{HHK14}, proof of Proposition 3.2.2) and shrinking further if necessary, there exists a neighbourhood $U'_{\eta}\subseteq U_{\eta}$ of $\eta$ such that $\brac{bf'}([D_N\otimes E\otimes F_{U'_{\eta}}])=0\in \HH^3\brac{E_N\otimes F_{U'_{\eta}}, \mu_{\ell}}$. Since $D_N\otimes E$ has index $\ell$, by injectivity of Suslin's invariant (\cite{Suslin}, Theorem 12.2), we have $bf'$ is a reduced norm from $D_N\otimes E \otimes F_{U'_{\eta}}$. The element $f_{U'_{\eta}}:=f'$ has the required properties. \end{proof}

\begin{remark}
\label{patchingsetupp'}
Let $T'_{\eta}$ denote the finite set of closed points $\overline{\eta}\setminus \brac{U'_{\eta}\cup \mathcal{P}'_{\eta}\cup \mathcal{R}'_{\eta}}$. Thus $\{\mathcal{P}'_{\eta}\cup \mathcal{R}'_{\eta}\cup \mathcal{T}'_{\eta}, U'_{\eta}\}_{\eta\in N_0}$ form a patching set up $\mathcal{P}'$ as in defined in (\cite{HH}).
\end{remark}


\begin{proposition}\label{propositionfatnewpoints}
Let $\eta\in N_0$ and let $P\in \brac{\mathcal{P}'_{\eta}\cup \mathcal{R}'_{\eta}\cup T'_{\eta}}$. Then there exists $f_P\in N\otimes F_P$ such that 
\begin{enumerate}
\item
$bf_{P}$ is a reduced norm from $D_N\otimes E_{P}$.
\item
$f_{U'_{\eta}}\psi_{P,\eta}^{\ell}= f_{P}$ in $N\otimes F_{P,\eta}$ for some $\psi_{P,\eta}\in N\otimes F_{P,\eta}$.
\end{enumerate}
\end{proposition}

\begin{proof}
If $P\in \mathcal{P}'_{\eta}\cup \mathcal{R}'_{\eta}$, the proposition follows from Propositions \ref{prop-f-atclosedpoints}, \ref{prop-f-atcodimensiononepoints} and \ref{propositionfatUs}. Hence assume $P\in \mathcal{T}'_{\eta}$. In particular, this implies $P\in U_{\eta}$ where $U_{\eta}$ is the neighbourhood of $\eta$ in the patching system $\mathcal{P}$ defined in Remark \ref{patchingsetupp}. Hence $F_{U_{\eta}}\subseteq F_P$. Let $(\pi_P,\delta_P)$ be a system of regular parameters of $A_P$ where $\pi_P$ defines the curve $\overline{\eta}$ at $P$.

We choose $f_P$ depending the shape of $D\otimes E_P$ as follows: 

\textit{$D\otimes E\otimes F_P$ is split}: Since $N\otimes F_P$ is dense in $N\otimes F_{P,\eta}$, pick an $f_P$ here which approximates $f_{\eta}\in N\otimes F_{\eta}$ treated as an element over the branch, i.e. $f_{\eta}\in N\otimes F_{P,\eta}$. The proposition is clearly true for this choice of $f_P$.

\textit{$D\otimes E \otimes F_P$ is not split}:  Since $P$ is a curve point, we have $D\neq 0\in \Br(F_P)$ possibly only if $\eta$ is of Type 1b or 2, in which case $D = \brac{u_P, \pi_P}\in \Br\brac{F_P}$ where $u_P\in \widehat{A_P}^*$ (\cite{S97}).  Let $u_{\eta}\in F_{\eta}$ be such that $u' = \overline{u_{\eta}} \in k_{\eta}^*/k_{\eta}^{*\ell}$ is the residue of $D_{\eta}$. Thus $\overline{u_P} \cong u' \in k_{P,\eta}$ up to $\ell^{\mrm{th}}$ powers.  

Except when $\eta$ is coloured green, $D\otimes E_{\eta}$ is split by construction (Propositions in \ref{sectionpatchingdataatN0}). By Proposition \ref{propositionEaatUsinpatching}, this implies $D\otimes E\otimes F_{U_{\eta}}$ is split and hence so is $D\otimes E\otimes F_P$. When $\eta$ is coloured green, by Propositions \ref{propositionE-choosingliftofresidues}, \ref{propositionE-1bRAM} and \ref{propositionE-1bRESY}, $E_{\eta}/F_{\eta}$ is unramified. By Proposition \ref{propositionEaatUsinpatching}, this implies $E\otimes F_{U_{\eta}}\simeq \frac{ F_{U_{\eta}}[t]}{(t^{\ell}-e)}$ for some unit $e\in \widehat{A_{U}}^*$. Therefore $E\otimes F_P = \prod F_P$ or $L_P$, the unique field extension of $F_P$ of degree $\ell$ unramified at $\widehat{A_P}$. If $E\otimes F_P$ is a nonsplit field extension, then $D\otimes E \otimes F_P$ is split. Thus $E\otimes F_P \simeq \prod F_P$. Therefore $E_N\otimes F_P/N\otimes F_P \simeq \prod_i \brac{\prod_{\ell} F_P}/F_P$. Let us look at the $i$-th component $\brac{\prod_{\ell} F_P}/F_P$ in $E_N\otimes F_P/N\otimes F_P$. We will prescribe $f_P$ by prescribing each of its components $f_i \in F_P$.

Let $b_{i} = \left(b_{i,1}, b_{i,2}, \ldots, b_{i,\ell}\right)\in \prod F_{P}$. By Proposition \ref{prop-f-atcodimensiononepoints}, we have $f_{P,\eta}\in \widehat{A_{P,\eta}}^*$ such that for each $j$, we have $\left(b_{i,j} f_{P,\eta}\right)\left(u_P, \pi_P\right) = 0 \in \HH^3\left(F_{P,\eta}, \mu_{\ell}\right)$. Let $b_{i,1}$ have valuation $m_1$ in $F_{P,\eta}$ and let $b'_{i,1}:= b_{i,1}f_{P,\eta}\pi_P^{-m_1}\in \widehat{A_{P,\eta}}^*$. Thus $\left(b'_{i,1}\right)\left(u_P, \pi_P\right) = 0$ also and taking residues, we get $\left(\overline{b'_{i,1}}, \overline{u_P} \right) = 0 \in \HH^2\left(k_{P,\eta}, \mu_{\ell} \right)$. Since $\left(b'_{i,1}, u_P\right)$ is unramified over $F_{P,\eta}$, it is also split over $F_{P,\eta}$ and we see that there exists $\theta_1\in F_{P,\eta}\left(\sqrt[\ell]{u_P}\right)$ such that $\N\left(\theta_1\right) = b'_{i,1}$.

Since we are in the case when $D\otimes E\otimes F_P$ is not split, $u_P\not\in F_P^{*\ell}$. Therefore $F_{P,\eta}\brac{\sqrt[\ell]{u_P}}$ is an unramified field extension of $F_{P,\eta}$ as also its residue field $k_{P,\eta}\brac{\sqrt[\ell]{\overline{u_P}}}/k_{P,\eta}$. As $b'_{i,1}\in \widehat{A_{P,\eta}}^*$, clearly $\theta_1 \in \mathcal{O}_{F_{P,\eta}\left(\sqrt[\ell]{u_P}\right)}^*$. Let $\overline{\theta_1} = \theta' \overline{\delta_P}^m$ where $\theta'\in \mathcal{O}_{k_{P,\eta}\left(\sqrt[\ell]{\overline{u_P}}\right)}$ and $m\in \mathbb{Z}$. Find $\tilde{\theta'}\in \mathcal{O}_{F_P\left(\sqrt[\ell]{u_P}\right)}^*$ such that its image matches that of $\theta'$. Set $\tilde{\theta_1} = \tilde{\theta'}\delta_P^m\in F_P\brac{\sqrt[\ell]{u_P}}$. Thus $\tilde{\theta_1}\cong \theta_1$ up to $\ell^{\mrm{th}}$ powers in $F_{P,\eta}\left(\sqrt[\ell]{u_P}\right)$. Set $f_{i,P} = b_{i,1}^{-1}\N\left(\tilde{\theta_1}\right)\pi_P^{m_1}\in F_P$.

Thus by construction $\left(b_{i,1}f_{i,P}\right)\left(u_P,\pi_P\right)=0\in \HH^3\left(F_P, \mu_{\ell}\right)$ and $f_{i,P}\cong f_{P,\eta}$ up to $\ell^{\mathrm{th}}$ powers in $F_{P,\eta}$. Finally, since $b^{-1}\sigma(b)$ is a reduced norm from $D\otimes E_P$, we have that for every $j$, the cup-products $\left(b_{i,j}\right)\left(u_P, \pi_P\right)$ are all equal in $\HH^3\left(F_P, \mu_{\ell}\right)$. Thus, we also have that for each $j$, $\left(b_{i,j}f_{i,P}\right)\left(u_P,\pi_P\right)=0\in \HH^3\left(F_P, \mu_{\ell}\right)$. Again using injectivity of Suslin's invariant for index $\ell$ algebras (\cite{Suslin}, Theorem 12.2), we can argue as before that this implies $b_{i,j}f_{i,P}$ is a reduced norm from $D\otimes F_P$ for each $j$ and hence that $b_if_{i,P}$ is a reduced norm from $D\otimes E\otimes F_P$.
\end{proof}

We are now in a position to find $f\in N$ satisfying the hypothesis of Proposition \ref{propositionf}. 
\begin{proposition}
\label{propositionfpatching}
There exists $f\in N$ such that $bf\in \Nrd_{E_N}\left(C_{D_N}\left(E_N\right)\right)$.
\end{proposition}
\begin{proof} By Propositions \ref{propositionfatUs} and \ref{propositionfatnewpoints}, we have $f_{x}\in N\otimes F_x$ for $x\in \left\{U'_{\eta}, \mathcal{P}'_{\eta} \cup \mathcal{R}'_{\eta}\cup \mathcal{T}'_{\eta}\right\}_{\eta\in N_0}$ in the patching set-up $\mathcal{P}'$ defined in Remark \ref{patchingsetupp'} such that for $b f_{ x} \in \Nrd D_N\otimes E_{ x}$. Further for each branch in the patching set-up corresponding to a pair $\left(U'_{\eta}, P\right)$, we have $f_{ P} = f_{ U'_{\eta}}\psi_{ P,\eta}^{\ell}$ for some $\psi_{ P,\eta}\in N\otimes F_{P,\eta}^*$.

By \textit{simultaneous factorization for curves} for the rational group $R_{N/F}\mathbb{G}_m$ (\cite{HHK09}, Theorem 3.6), we can find $\psi_{ x}\in \brac{N\otimes F_x}^*$ for each $x\in\mathcal{P}'$ such that for every branch defined by $\left(U'_{\eta}, P\right)$, we have $\psi_{ P,\eta} = \psi_{ U'_{\eta}}\psi_{ P}^{-1}$. Thus we have $f_{ U'_{\eta}}\psi_{ U'_{\eta}}^{\ell} = f_{ P}\psi_{ P}^{\ell}$ for every branch $\left(U'_{\eta}, P\right)$. Therefore there exists $f \in N$ such that $f  = f_{ x}\psi_{ x}^{\ell} \in N\otimes F_x$ for each $x\in \left\{U'_{\eta}, P \right\}$ (\cite{HH}, Proposition 6.3 \& Theorem 6.4). Thus $b f \in \Nrd \brac{D_N\otimes E_{x}}$ and therefore $\brac{bf}\cup [D_N\otimes E_x]=0\in \HH^3\left(N\otimes E_x, \mu_{\ell}\right)$  for each $x\in \mathcal{P}'$. This implies $\brac{bf}\cup [D_N\otimes E]=0\in \HH^3\left(E_N, \mu_{\ell}\right)$ (\cite{PPS}, proofs of Proposition 7.1 \& 7.4). Injectivity of Suslin's invariant for index $\ell$ algebras (\cite{Suslin}) shows that $bf \in \Nrd \brac{D_N\otimes E}$ which proves the proposition. \end{proof}

Thus we have our main theorem:

\begin{theorem}
\label{theoremSK1}
Let $F$ be the function field of a curve over a $p$-adic field. Let $D/F$ be a central division algebra of prime exponent $\ell$ which is different from $p$. Assume that $F$ contains a primitive ${{\ell}^2}^{\mathrm{th}}$ root of unity. Then $\SK1(D)$ is trivial.
\end{theorem}

\textit{Acknowledgements} : The author was partially supported by National Science Foundation grants DMS-1401319 and DMS-1463882 during the period in which this paper was written. She thanks Professors Parimala Raman and Suresh Venapally for very many wonderful discussions and valuable suggestions. Without their immense help and encouragement at each and every step, indeed this paper could not have been completed. She also thanks Bastian Haase for several helpful comments on the text.

\end{document}